\documentclass[a4paper, 10pt]{amsart}
\usepackage{amsmath}
\usepackage{amssymb, color, hyperref}
\usepackage[mathscr]{eucal}
\usepackage{accents}

\def\doi#1{{\small\href{https://doi.org/#1}{\path{doi:#1}}}}
\def\arxiv#1{{\small\href{http://www.arxiv.org/abs/#1}{\path{arXiv:#1}}}}
\def\url#1{{\small\href{#1}{\path{#1}}}}

\theoremstyle{plain}
\newtheorem{theorem}{\bf Theorem}[section]
\newtheorem{proposition}[theorem]{\bf Proposition}
\newtheorem{lemma}[theorem]{\bf Lemma}
\newtheorem{corollary}[theorem]{\bf Corollary}

\theoremstyle{definition}
\newtheorem{example}[theorem]{\bf Example}

\newtheorem{remark}[theorem]{\bf Remark}

\newcommand{\N}{\mathbb N}
\newcommand{\Z}{\mathbb Z}

\newcommand{\Q}{\mathbb Q}

\newcommand{\bdot}{\boldsymbol{\cdot}}

 \DeclareMathOperator{\ord}{ord}
\DeclareMathOperator{\lcm}{lcm} 
 \DeclareMathOperator{\Ca}{\mathsf {Ca}}
\DeclareMathOperator{\spec}{spec} \DeclareMathOperator{\supp}{supp}

\newcommand{\red}{{\text{\rm red}}}
\renewcommand{\t}{\, | \,}

\newcommand{\LK}{\,[\![}
\newcommand{\RK}{]\!]}

\newcommand{\Fc}{\mathcal F}
\newcommand{\vp}{\mathsf v}

\newcommand{\Sum}[2]{\underset{#1}{\overset{#2}{\sum}}}
\newcommand{\Summ}[1]{\underset{#1}{\sum}}
\newcommand{\la}{\langle}
\newcommand{\ra}{\rangle}
\newcommand{\be}{\begin{equation}}
\newcommand{\ee}{\end{equation}}
\newcommand{\und}{\;\mbox{ and }\;}

\numberwithin{equation}{section}

\subjclass[2010]{11B30, 11B50, 11B75, 20M13, 20M14}

\thanks{This work was supported by the Austrian Science Fund FWF, Projects W1230 and P33499.}

\begin{document}

\title[Product-one sequences  over dihedral groups]{On  product-one sequences  over dihedral groups}

\author{Alfred Geroldinger, David J. Grynkiewicz, Jun Seok Oh, and Qinghai Zhong}

\address{Institute for Mathematics and Scientific Computing \\
University of Graz, NAWI Graz \\
Heinrichstra{\ss}e 36\\
8010 Graz, Austria}
\email{alfred.geroldinger@uni-graz.at, junseok1.oh@gmail.com, qinghai.zhong@uni-graz.at}
\urladdr{https://imsc.uni-graz.at/geroldinger, https://imsc.uni-graz.at/oh, https://imsc.uni-graz.at/zhong/}

\address{Department of Mathematical Sciences, University of Memphis, Memphis, TN 38152, USA}
\email{diambri@hotmail.com}
\urladdr{https://diambri.org/}

\keywords{product-one sequences, Davenport constant, finite groups, partition theorem, sets of lengths, sets of distances, sets of catenary degrees}

\begin{abstract}
Let $G$ be a finite group. A sequence over $G$ means a finite sequence of terms from $G$, where repetition is allowed and the order is disregarded. A product-one sequence is a sequence whose elements can be ordered such that their product equals the identity element of the group. The set of all product-one sequences over $G$ (with concatenation of sequences as the operation) is a finitely generated C-monoid. Product-one sequences over dihedral groups have a variety of extremal properties. This article provides a detailed investigation, with methods from arithmetic combinatorics, of the arithmetic  of the monoid of product-one sequences over dihedral groups.
\end{abstract}

\maketitle


\section{Introduction} \label{1}

Let $G$ be a finite group. A sequence over $G$ means a finite sequence of terms from $G$, where repetition of terms is allowed and their order is disregarded. A sequence is called product-one free if no subproduct of terms (in any order) equals the identity of the group, and it is called a product-one sequence if its
terms can be ordered such that their product equals the identity of $G$. The small Davenport constant $\mathsf d (G)$ is the maximal length of a product-one free sequence and the large Davenport constant $\mathsf D (G)$ is the maximal length of a minimal product-one sequence (a minimal product-one sequence is a product-one sequence that cannot be factorized, or say partitioned, into two nontrivial product-one sequences). The study of sequences, their sequence subproducts, and their structure under extremal properties is a classical topic in additive combinatorics.

If $G$ is  additively written and abelian, then we speak of zero-sum free sequences,  zero-sum sequences, and of
sequence subsums. Their study is a main objective of zero-sum theory, which has intimate connections to various areas of combinatorics, graph theory, finite geometry, factorization theory, and invariant theory. Although,  for a long time, the focus of study was on the abelian setting, the study of combinatorial invariants in the general setting dates back at least to the 1970s when Olson gave an upper bound for $\mathsf d (G)$ (\cite{Ol-Wh77}). There are recent studies on (small and large) Davenport constants, on the {E}rd{\H{o}}s-Ginzburg-Ziv constant $\mathsf s (G)$, and on the constant $\mathsf E (G)$, which asks for the smallest integer $\ell$ such that every sequence over $G$ of length at least $\ell$ has a product-one subsequence of length $|G|$ (e.g., \cite{Ba07b, Ga-Li10b, Ha15a, Br-Ri18a, Ha-Zh19a, Oh-Zh20a, Oh-Zh20b}). These investigations were pushed forward by new applications to invariant theory and to factorization theory. To begin with invariant theory, let $\boldsymbol \beta (G)$ denote the Noether number of $G$. If $G$ is abelian, then  B. Schmid  \cite{Sc91a} observed that  $\mathsf d (G)+1 = \boldsymbol \beta (G) = \mathsf D (G)$. If $G$ has a cyclic subgroup of index two, then it was shown  by Cziszter, Domokos, and by two of the present authors that $\mathsf d (G)+1 \le \boldsymbol \beta (G) \le \mathsf D (G)$ (\cite{Ge-Gr13a, Cz-Do14a}). For general groups the relationship between the Davenport constants and the Noether number is open, but in all cases studied so far we have $\mathsf d (G)+1 \le \boldsymbol \beta (G)$ (\cite{Cz-Do-Ge16a, Cz-Do-Sz18, Cz19a}).

To discuss the connection with factorization theory, we first observe that the set $\mathcal B (G)$ of product-one sequences over $G$ is a finitely generated (commutative and cancellative) monoid with concatenation of sequences as
its operation. The atoms (i.e., the irreducible elements) of $\mathcal B (G)$ are precisely the minimal product-one sequences over $G$. First, let $G$ be abelian and, for simplicity, suppose that $|G| \ge 3$. Then $\mathcal B (G)$ is a Krull monoid with class group (isomorphic to) $G$ and every class contains precisely one prime divisor. If $H$ is any Krull monoid with class group $G$ and prime divisors in each class, then there is a transfer homomorphism $\theta \colon H \to \mathcal B (G)$ implying that arithmetical invariants  (such as sets of lengths, catenary degrees, and more) of $H$ and of $\mathcal B (G)$ coincide. The arithmetic of $\mathcal B (G)$ is studied with methods of additive combinatorics and the long-term goal is to determine the precise value of arithmetical invariants in terms of the group invariants of $G$ and/or in terms of classical combinatorial invariants such as the Davenport constant. We refer to \cite{Ge-Ru09} for the interplay of the arithmetic of Krull monoids and additive combinatorics, and to the survey \cite{Sc16a} for a discussion of the state of the art.

Monoids of product-one sequences over finite groups are C-monoids. C-domains and C-monoids are submonoids of factorial monoids with finite class semigroup. They include Krull monoids with finite class group (and in that case the class semigroup coincides with the usual class group) but also classes of non-integrally closed noetherian domains (such as orders in number fields). The finiteness of the class semigroup yields abstract  finiteness results for arithmetical invariants, but so far  no combinatorial
description of invariants in terms of the class semigroup are available (as is the case for Krull monoids) let alone any sort of precise results.

In the present paper, we study the arithmetic of the monoid of product-one sequences over dihedral groups $G$ of order $2n$ for odd $n \ge 3$, and we obtain precise results. These dihedral groups were chosen because their arithmetic shows extremal behavior among all finite groups (see Proposition \ref{2.3} and Theorem \ref{6.6}), such as  cyclic groups and elementary $2$-groups do among all finite abelian groups. We do not involve algebraic considerations (structural results on the class semigroup of monoids of product-one sequences  were recently established in \cite[Section 3]{Oh20a}) but work with methods from additive combinatorics. We use substantially the recent characterization of minimal product-one sequences of maximal length (Proposition \ref{2.4}) and a recent refinement (\cite{Gr21a}) of the Partition Theorem (\cite[Chapters 14 and 15]{Gr13a}). Let $G$ be a dihedral group of order $2n$ for some odd $n \ge 3$. In the short Section \ref{3}, we do some necessary algebraic clarifications. Theorem \ref{4.1} states that $\omega (G)=2n$. Theorem \ref{5.1} states that the set of distances $\Delta (G)$ is equal to $[1, 2n-2]$, and the set of catenary degrees $\Ca (G)$ equals $[2, 2n]$. Theorems \ref{6.6} and \ref{6.7} give detailed information on crucial subsets of $\Delta (G)$ which describe the structure of sets of lengths. Our results are strong enough to characterize $\mathcal B (G)$ arithmetically with respect to some other classes of monoids (Corollary \ref{6.11}).

\section{Background on the arithmetic of monoids} \label{2}

Our notation and terminology are  consistent with \cite{Ge-HK06a, Gr13a}. We briefly gather some key notions and fix notation. We denote by $\N$ the set of positive integers. For rational numbers $a, b \in \Q$, $[a, b ] = \{ x \in \Z \colon a \le x \le b\}$ means the discrete interval between $a$ and $b$. For an additive group $G$ and  subsets $A, B \subset G$, $A+B = \{a+b \colon a \in A, b \in B \}$ denotes their sumset.
For $A\subseteq \Z$, the set of distances $\Delta (A)$ is the set of all $d \in \N$ for which there is $a \in A$ such that $A \cap [a, a+d] = \{a, a+d\}$. If $A \subset \N_0$, then $\rho (A) = \sup (A \cap \N)/\min (A \cap \N) \in \Q_{\ge 1} \cup \{\infty\}$ denotes the elasticity of $A$ with the convention that  $\rho ( A)=1$ if $A \cap \N = \emptyset$.

\subsection{\bf Monoids.} \label{Monoids}
Throughout this paper, a {\it monoid} means a commutative, cancellative semigroup with identity.  Let $H$ be a monoid. Then $H^{\times}$ denotes the group of invertible elements, $\mathcal A (H)$ the set of atoms of $H$,  $\mathsf q (H)$ the quotient group of $H$, and $H_{\red} = \{aH^{\times} \colon a \in H \}$ the associated reduced monoid of $H$. A submonoid $S \subset H$ is said to be {\it divisor-closed} if $a \in H, b \in S$ and $a \mid b$ implies that $a \in S$. For a subset $E \subset H$ we denote by
\begin{itemize}
\item $[E] \subset H$ the smallest submonoid of $H$ containing $E$, and by
\item $\LK E \RK \subset H$ the smallest divisor-closed submonoid of $H$ containing $E$.
\end{itemize}
Clearly, $\LK E \RK$ is the set of all $a \in H$ dividing some element $b \in [E]$. If $E = \{a_1, \ldots, a_m\}$, then we write $[a_1, \ldots, a_m] = [E]$ and $\LK a_1, \ldots, a_m \RK = \LK E \RK $. We denote by
\begin{itemize}
\item $H' = \{ x \in \mathsf q (H) \colon \ \text{there is $N \in \N$ such that $x^n \in H$ for all $n \ge N$} \}$ the {\it seminormal closure} of $H$,

\item $\widetilde H = \{x \in \mathsf q (H) \colon \ x^n \in H \ \text{for some} \ n \in \N \}$ the {\it root closure} of $H$, and by

\item $\widehat H = \{ x \in \mathsf q (H) \colon \text{there is $c \in H$ such that $cx^n \in H$ for all $n \in  \N$} \}$ the {\it complete integral closure} of $H$.
\end{itemize}
Then $H \subset H' \subset \widetilde H \subset \widehat H \subset \mathsf q (H)$, and  $H$ is called seminormal (root closed, resp. completely integrally closed) if $H = H'$ ($H = \widetilde H$, resp. $H = \widehat H$).
For a set $P$, we denote by $\mathcal F (P)$ the free abelian monoid with basis $P$ whose elements are written as
\[
a = \prod_{p \in P} p^{\mathsf v_p (a)} \in \mathcal F (P) \,,
\]
where $\mathsf v_p \colon H \to \N_0$ is the $p$-adic valuation of $a$. We call $|a| = \sum_{p \in P} \mathsf v_p (a) \in \N_0$ the {\it length} of $a$ and $\supp (a) = \{ p \in P \colon \mathsf v_p (a) > 0 \} \subset P$ the {\it support} of $a$.

The monoid $\mathsf Z (H) = \mathcal F (\mathcal A (H_{\red}))$ is the factorization monoid of $H$ and the unique epimorphism $\pi \colon \mathsf Z (H) \to H_{\red}$, satisfying $\pi (u) = u$ for all $u \in \mathcal A (H_{\red})$, denotes the factorization homomorphism. For $a \in H$, we denote by
\begin{itemize}
\item $\mathsf Z (a) = \pi^{-1} (aH^{\times})$  the {\it set of factorizations} of $a$,

\item $\mathsf L (a) = \{|z|  : z \in \mathsf Z (a) \} \subset \N_0$ the {\it set of lengths} of $a$, \quad \text{and}

\item $\mathcal L (H) = \{\mathsf L (a) \colon a \in H \}$ the {\it system of sets of lengths of $H$}.
\end{itemize}
Note that $\mathsf L (a) = \{0\}$ if and only if $a \in H^{\times}$ and that $\mathsf L (a) = \{1\}$ if and only if $a \in \mathcal A (H)$. The monoid $H$ is called {\it atomic} if $\mathsf Z (a) \ne \emptyset$ for all $a \in H$ (equivalently, if  every $a \in H \setminus H^{\times}$ has a factorization into atoms),   and   $H$ is called {\it half-factorial} if $|L (a)| = 1$ for all $a \in H$.
We denote by
\begin{equation} \label{def-set-of-distances}
\Delta (H) = \bigcup_{L \in \mathcal L (H)} \Delta (L) \ \subset \N
\end{equation}
the {\it set of distances} of $H$. If $\Delta (H) \ne \emptyset$, then
\begin{equation} \label{set-of-distances}
\min \Delta (H) = \gcd \Delta (H) \,.
\end{equation}
For an atomic monoid $H$ with $H \ne H^{\times}$ and every $k \in \N$, let
\begin{equation} \label{def-unions}
\mathcal U_k (H) = \bigcup_{k \in L \in \mathcal L (H)} L \quad \subset \N \quad
\end{equation}
denote the {\it union of sets of lengths} containing $k$. Then $\rho_k (H) = \sup \mathcal U_k (H)$ is the $k$-th {\it elasticity} of $H$ and (see \cite[Proposition 1.4.2]{Ge-HK06a})
\begin{equation} \label{elasticity}
\rho (H) = \sup \{\rho (L) \colon L \in \mathcal L (H) \} = \lim_{k \to \infty} \frac{\rho_k (H)}{k}
\end{equation}
is the {\it elasticity} of $H$. We define a distance function $\mathsf d$ on $\mathsf Z (H)$. If $z, z' \in \mathsf Z (H)$, then $z$ and $z'$ can be written uniquely in the form
\[
z = u_1 \cdot \ldots \cdot u_{\ell} v_1 \cdot \ldots \cdot v_{m} \quad \text{and} \quad z' = u_1 \cdot \ldots \cdot u_{\ell} w_1 \cdot \ldots \cdot w_n \,,
\]
where $\ell, m, n  \in \N_0$, all $u_i, v_j, w_k \in \mathcal A (H_{\red})$, and $\{v_1, \ldots, v_m\} \cap \{w_1, \ldots, w_n\} = \emptyset$, and we define $\mathsf d (z, z') = \max \{m,n\} \in \N_0$.

\subsection{\bf Product-one sequences over finite groups.} \label{Sequences}
Let $G$ be a multiplicatively written finite group with identity $1_G \in G$  and let $G_0 \subset G$ be a subset. Then $\langle G_0 \rangle \subset G$ is the subgroup generated by $G_0$ and $G' = \langle g^{-1}h^{-1}gh \colon g, h \in G \rangle \subset G$ is commutator subgroup of $G$. If $G$ is (additively written) abelian, then $\mathsf H (G_0) = \{ g \in G \colon g+G_0 = G_0\}$ denotes the {\it stabilizer} of $G_0$. We say that a subset $A\subseteq G$ is $H$-periodic if $H\leq \mathsf H(A)$, which is equivalent to $A$ being a union of $H$-cosets, and that $A$ is \emph{aperiodic} if $\mathsf H(A)$ is trivial.
 We use $\phi_H:G\rightarrow G/H$ to denote the natural homomorphism.
For every $n \in \N$, $C_n$ denotes a cyclic group of order $n$  and $D_{2n}$ denotes a dihedral group of order $2n$.

Elements of $\mathcal F (G_0)$ are called  {\it sequences} over $G_0$. Thus, in combinatorial language, a sequence means a finite sequence of terms from $G_0$ which is unordered
with the repetition of terms
allowed. In order to distinguish between the group operation in $G$ and the operation in $\mathcal F (G_0)$, we use the symbol $\bdot$ for the multiplication in $\mathcal F (G_0)$, so $G = (G, \cdot)$ and $\mathcal F (G) = (\mathcal F (G), \bdot)$. In order to avoid confusion between exponentiation of the group operation $\cdot$ in $G$ and exponentiation of the sequence operation $\bdot$ in $\mathcal F (G)$, we use brackets to denote exponentiation in $\mathcal F (G)$. Thus, for $g \in G$, $T \in \mathcal F (G)$, and $k \in \N_0$, we have
\[
g^{[k]} = \underbrace{g \bdot \ldots \bdot g}_{k} \in \mathcal F (G) \quad \text{and} \quad T^{[k]} = \underbrace{T \bdot \ldots \bdot T}_{k} \in \mathcal F (G) \,.
\]
Moreover, if $T^{[k]}$ divides a sequence $S$ in $\mathcal F (G)$, then $S \bdot T^{[-k]} = S \bdot (T^{[k]})^{[-1]} \in \mathcal F (G)$ denotes the subsequence of $S$ obtained by removing the terms of $T^{[k]}$ from $S$. Let
\[
S = g_1 \bdot \ldots \bdot g_{\ell} = {\prod}^\bullet_{g \in G_0} g^{[\mathsf v_g (S)]}
\]
be a sequence over $G_0$. Then
\begin{itemize}
\item $\mathsf h (S) = \max \{\mathsf v_g (S) \colon g \in G_0\} $ is the {\it maximum multiplicity} of a term of $S$,

\item $\mathsf k (S) = \sum_{i=1}^{\ell}  \frac{1}{\ord (g_i)} \in \Q$ is the {\it cross number} of $S$, and

\item $\pi (S) = \{ g_{\tau (1)} \cdot \ldots \cdot g_{\tau (\ell)} \in G \colon \tau \ \text{is a permutation of } \ [1, \ell] \} \subset G$ is the {\it set of products} of $S$,
\end{itemize}
and it is readily seen that $\pi (S)$ is contained in a $G'$-coset.
If $|S|= 0$, then we use the convention that $\pi (S) = \{1_G\}$. When $G$ is written additively with commutative operation, we likewise let $\sigma (S)=g_1+ \ldots + g_{\ell} \in G$ denote the {\it sum} of $S$. For $n \in \N_0$, the $n$-{\it sums} and $n$-{\it products} of $S$ are respectfully denoted by
\[
\Sigma_n (S) = \{\sigma (T) \colon T \mid S \ \text{and} \ |T|=n\} \subset G \quad \text{and} \quad \Pi_n (S) = \bigcup_{T\mid S, |T|=n} \pi (T) \subset G \,.
\]
The {\it sequence subsums} and {\it sequence subproducts} of $S$ are respectively denoted by
\[
\Sigma (S) = \bigcup_{n \ge 1} \Sigma_n (S) \quad \text{and} \quad \Pi (S) = \bigcup_{n \ge 1} \Pi_n (S) \subset G \,.
\]
A map of groups $\varphi \colon G \to H$ extends to a monoid homomorphism $\varphi \colon \mathcal F (G) \to \mathcal F (H)$ by setting $\varphi (S) = \varphi (g_1) \bdot \ldots \bdot \varphi (g_{\ell}) \in \mathcal F (H)$. Furthermore, we set $S^{-1} = g_1^{-1} \bdot \ldots \bdot g_{\ell}^{-1}$. If $G$ is written additively and $\varphi$ is the multiplication by some $m \in \N$, then we set $mS = (mg_1) \bdot \ldots \bdot (mg_{\ell})$.

For a  subset $X \subset G_0$, we let $$S_X={\prod}^\bullet_{g\in X}g^{[\vp_g(S)]}$$ denote the subsequence of $S$ consisting of all terms from $X$.
The sequence $S$ is called
\begin{itemize}
\item a {\it product-one sequence} if $1_G \in \pi (S)$,
\item {\it product-one free} if $1_G \notin \Pi (S)$.
\end{itemize}
The set
\[
\mathcal B (G_0) = \{ S \in \mathcal F (G_0) \colon 1_G \in \pi (S) \} \subset \mathcal F (G_0)
\]
is a finitely generated submonoid of $\mathcal F (G_0)$, called the {\it monoid of product-one sequences} over $G_0$. For all arithmetical invariants $* (H)$ defined for a monoid $H$, we write $* (G_0)$ instead of $* ( \mathcal B (G_0))$ (although  being an abuse of notation this is a usual convention that will not lead to confusion). Similarly, we say that $G_0$ is (non-)half-factorial if $\mathcal B (G_0)$ is (non-)half-factorial. The atoms of $\mathcal B (G_0)$ are also called {\it minimal product-one sequences}. Since $\mathcal B (G_0)$ is finitely generated, $\mathcal A (G_0)$ is finite,
\begin{itemize}
\item $\mathsf K (G_0) = \max \{\mathsf k (S) \colon S \in \mathcal A (G_0) \} \in \Q_{\ge 0}$ is the {\it cross number} of $G_0$,

\item $\mathsf D (G_0) = \max \{ |S| \colon S \in \mathcal A (G_0) \} \in \N_0$ is the {\it large Davenport constant} of $G_0$, and

\item $\mathsf d (G_0) = \max \{ |S| \colon S \in \mathcal F (G_0) \ \text{is product-one free} \} \in \N_0$ is the {\it small Davenport constant} of $G_0$.
\end{itemize}

It is easy to verify that
\begin{equation} \label{natural-bounds}
\mathsf d (G)+1 \le \mathsf D (G) \le |G| \quad \text{and} \quad \mathsf K (G) \le \max \bigl\{1, \frac{\mathsf D (G)}{2} \bigr\} \,.
\end{equation}
Let $G$ be a finite group and
$G_0<G$
a proper subgroup. Then $\mathsf d (G_0) < \mathsf d (G)$ and $\mathsf D (G_0) \le \mathsf D (G)$. If $G$ is abelian, then $\mathsf D (G_0) = 1+\mathsf d (G_0) < 1 + \mathsf d (G) = \mathsf D (G)$. If $G$ is not abelian, then we might have $\mathsf D (G_0) = \mathsf D (G)$, as it is outlined in the next example.

\smallskip

\begin{example}
	We consider the semidirect product
	\[
	G = C_5 \rtimes_2 C_4 = \langle a, b \colon a^5 = b^4 = 1_G, bab^{-1} = a^2 \rangle \,.
	\]
	Thus $G$ is a  group with $20$ elements and  Table 1 in \cite{Cz-Do-Sz18} shows  that $\mathsf D (G) = 10$.
	Let $G_0 = \langle a, b^{2} \rangle \subset G$. Then $G_0$ is  a dihedral group with $10$ elements, whence $\mathsf D (G_0) = 10 = \mathsf D (G)$.
\end{example}

Thus the example shows that the subgroup $G_0$ (with $G_0 <G$ proper and $\mathsf D (G_0) = \mathsf D (G)$) can be dihedral (for some consequences, see Proposition \ref{2.3}), but the next lemma shows that $G_0$ cannot be abelian.

\begin{lemma} \label{2.2}
Let $G$ be a finite group and $G_0 < G$
a subgroup
 with $\mathsf D (G_0) = \mathsf D (G)$.
\begin{enumerate}
\item  If $\mathsf D (G_0) = 1 + \mathsf d (G_0)$, then $G_0 = G$.

\item If $G$ is nilpotent but not a $2$-group, then $G_0$ is not generated by elements of order two.
\end{enumerate}
\end{lemma}

\begin{proof}
1. Let $S \in \mathcal F (G_0)$ be product-one free with
$|S|= \mathsf d (G_0)$. Assume to the contrary that there is an element $g  \in G \setminus G_0$. Then $S \bdot g$ is product-one free, whence
\[
\mathsf D (G) = \mathsf D (G_0) = 1 + \mathsf d (G_0) < 1 + |S \bdot g| \le 1 + \mathsf d (G) \le \mathsf D (G) \,,
\]
a contradiction.

2.
Let $G_1 \subset G$ be a subgroup that is generated by elements of order two. It suffices to show that $\mathsf D (G_1) < \mathsf D (G)$. Since finite nilpotent groups that are generated by elements of order $p$ are $p$-groups (\cite[Corollary 2.4]{G-K-O-R99}), it follows that $G_1$ is a $2$-group, thus contained in the Sylow $2$-group. As a finite nilpotent group is the direct product of its Sylow subgroups (which are each normal for nilpotent groups), there is a non-trivial group $G_2 <G$ (any nontrivial Sylow $p$-group with $p\neq 2$) such that $G_1 \times G_2 $ is a subgroup of $G$, which implies that $\mathsf D (G_1) < \mathsf D (G_1 \times G_2) \le \mathsf D (G)$.
\end{proof}

\smallskip
\begin{proposition} \label{2.3}~
Let $G$ be a finite group with $|G|>1$.
\begin{enumerate}
\item \begin{align*}
      \mathsf D(G)\left\{\begin{aligned}   &=|G|\quad \quad&&\text{$G$ is either cyclic or a dihedral group of order $2n$ for some odd $n \ge 3$}\,,\\
      &\le \frac{3|G|}{4} &&\text{otherwise}\,.
      \end{aligned}
      \right.
      \end{align*}

\item Consider the following two conditions{\rm \,:}
       \begin{enumerate}
       \item[(a)] $G$ is
       either an elementary $2$-group or  has a  subgroup $G_0 <G$ which is a dihedral group of order $2n$ for some odd $n \ge 3$  with $\mathsf D (G_0) = \mathsf D (G)$.

       \item[(b)] $\mathsf K (G) = \frac{\mathsf D (G)}{2}$.
       \end{enumerate}
       Then (a) implies (b). If $G$ is nilpotent and (b) holds, then $G$ is a $2$-group,
       and it is an elementary $2$-group in the abelian case.
\end{enumerate}
\end{proposition}

\begin{proof}
1. See \cite[Theorem 7.2]{Gr13b}.

2. (a) $\Rightarrow$ (b) If $G$ is elementary $2$-group, then by \cite[Corollaries 5.1.9 and 5.1.13]{Ge-HK06a}, we obtain that
\[
  \mathsf K (G) \, = \, \frac{1}{2} + \mathsf k (G) \, = \, \frac{1 + \mathsf d (G)}{2} \, = \, \frac{\mathsf D (G)}{2} \,.
\]
Suppose that  $G_0 = \langle \alpha, \tau \colon \alpha^{n} = \tau^{2} = 1_G \text{ and } \tau\alpha = \alpha^{-1}\tau \rangle <G$ is a dihedral group of order $2n$ for some odd $n \ge 3$ with $\mathsf D (G_0)=\mathsf D (G)$.
Then Equation \eqref{natural-bounds} shows that $\mathsf K (G) \le \frac{\mathsf D (G)}{2}$. On the other hand, it is easy to verify (or use Proposition \ref{2.4})  that $S = (\alpha \tau)^{[n]} \bdot \tau^{[n]}$
is a minimal product-one sequence  with
\[
|S| = 2n = \mathsf D (G_0) = \mathsf D (G) \quad \text{and} \quad \mathsf k (S) = \frac{2n}{2}  = \frac{\mathsf D (G)}{2} \,.
\]

Now suppose that $G$ is nilpotent and (b) holds. Note that $\mathsf K (G) \ge 1$ since $|G|>1$. Then there exists $S  = g_1 \bdot \ldots \bdot g_{\ell} \in \mathcal A (G)$ such that $\mathsf k (S) = \frac{\mathsf D (G)}{2}$.  Then $\ord (g_i) \ge 2$ for all $i \in [1,\ell]$, whence
\[
  \frac{\mathsf D (G)}{2} \, = \, \mathsf k (S) \, = \, \sum_{i=1}^{|S|} \frac{1}{\ord (g_i)} \, \le \, \frac{|S|}{2} \,.
\]
Therefore $|S| = \mathsf D (G)$ and $\ord (g_i) = 2$ for all $i \in [1, \ell]$. Thus $G_0 = \langle g_1, \ldots, g_{\ell} \rangle$ is generated by elements of order two and $\mathsf D (G_0)= \mathsf D (G)$. If $G_0$ is abelian, then $G_0$ is an elementary $2$-group and since $\mathsf D (G_0)=1+\mathsf d (G_0)$, Lemma \ref{2.2}.1 implies that $G=G_0$ is an elementary $2$-group. If $G$ is nilpotent, then Lemma \ref{2.2}.2 implies that $G$ is a $2$-group.
\end{proof}

The associated inverse problem with respect to the Davenport constant asks for the structure of minimal product-one sequences of length $\mathsf D (G)$. Even for abelian groups, the inverse problem is settled only for a small number of cases, namely for cyclic groups and elementary $2$-groups (for them the problem has a trivial answer), for groups of rank two, and for groups of the form $C_2 \oplus C_2 \oplus C_{2n}$ (\cite{Sc11b}). Dihedral
and dicyclic groups are the only non-abelian groups for which a characterization of product-one sequences of length $\mathsf D (G)$ is available. We cite the result for dihedral groups of order $2n$, where $n \ge 3$ is odd (see \cite[Theorem 4.1]{Oh-Zh20a}).

\begin{proposition} \label{2.4}
Let $G$ be a dihedral group of order $2n$, where $n \ge 3$ is odd. A sequence $S$ over $G$ of length $\mathsf D (G)$ is a minimal product-one sequence if and only if it has one of the following two forms{\rm \,:}
\begin{enumerate}
\item[(a)] There exist $\alpha, \tau \in G$ such that $G = \langle \alpha, \tau \colon \alpha^n = \tau^2 = 1_G \ \text{and} \ \tau \alpha = \alpha^{-1}\tau \rangle$ and $S = \alpha^{[2n-2]} \bdot \tau^{[2]}$.

\item[(b)]      There exist $\alpha, \tau \in G$ and $i,j \in [0,n-1]$ with $\gcd (i-j,n)=1$ such that $G = \langle \alpha, \tau \colon \alpha^n = \tau^2 = 1_G \ \text{and} \ \tau \alpha = \alpha^{-1}\tau \rangle$ and $S = (\alpha^i \tau)^{[n]} \bdot (\alpha^j \tau)^{[n]}$.
\end{enumerate}
\end{proposition}

Finally, we will make ample use of Kneser's Theorem \cite[Chapter 6]{Gr13a}.

\begin{theorem}[Kneser's Theorem] Let $G$ be an abelian group, let $A_1,\ldots,A_n\subseteq G$ be finite, nonempty subsets, and let $H=\mathsf H(\Sum{i=1}{n}A_i)$. Then
$|\Sum{i=1}{n}A_i|\geq \Sum{i=1}{n}|A_i+H|-(n-1)|H|.$
\end{theorem}

\smallskip
\section{Algebraic Properties} \label{3}
\smallskip

In this section, we study ideal theoretic properties of monoids of product-one sequences. Our references for ideal theory are \cite{HK98, Ge-HK06a}.
Let $H$ be a monoid. We denote by $s$-$\spec (H)$ the set of prime $s$-ideals of $H$ and by $\mathfrak X (H) \subset s$-$\spec (H)$ the set of minimal nonempty prime $s$-ideals of $H$. For a prime ideal $\mathfrak p \in s$-$\spec (H)$, we denote by $H_{\mathfrak p} = (H \setminus \mathfrak p)^{-1}H \subset \mathsf q (H)$ the localization at $\mathfrak p$.
The monoid $H$ is said to be {\it weakly Krull} if
\[
H = \bigcap_{{\mathfrak p} \in \mathfrak X (H)} H_{\mathfrak p}  \quad \text{and} \quad \{{\mathfrak p} \in \mathfrak X (H) \colon a \in {\mathfrak p}\} \quad \text{is finite for all} \ a \in H \,.
\]
The localizations $H_{\mathfrak p}$ are primary, and all primary monoids  are weakly Krull.
If all localizations $H_{\mathfrak p}$ are discrete valuation monoids, then $H$ is a Krull monoid.
A domain $R$ is weakly Krull if and only if its multiplicative monoid $R^{\bullet}$ of nonzero elements is weakly Krull.
Atomic domains having only finitely many non-associated atoms
 (i.e., $R^{\bullet}_{\red}$ is a finitely generated monoid) are called Cohen-Kaplansky domains (\cite[Theorem 4.3]{An-Mo92}), and they are weakly Krull  (\cite[Corollary 5]{An-An-Za92b}).  However, in contrast to the ring setting, finitely generated monoids are not  weakly Krull in general. Root-closed finitely generated monoids are Krull and hence weakly Krull. In this section, we show that the monoid of product-one sequences over a finite group  is weakly Krull if and only if the group  is abelian (Theorem \ref{3.3}).

\smallskip
\begin{proposition} \label{3.1}
Let $G$ be a finite group.
\begin{enumerate}
\item $\mathfrak X \big( \mathcal B (G) \big) = \{ {\mathfrak p}_g \colon g \in G\}$, where ${\mathfrak p}_g = \{ A \in \mathcal B (G) \colon \mathsf v_g (A) \ge 1 \}$ for each $g \in G$.

\item $$\widehat{\mathcal B (G)}=\{S \in \mathcal F(G) \colon \pi(S)\subset G'\}=\bigcap_{{\mathfrak p} \in \mathfrak X (\mathcal B (G))} \mathcal B (G)_{\mathfrak p}\,.$$
\end{enumerate}
\end{proposition}

\begin{proof}
We set $H = \mathcal B (G)$, $F=\mathcal F(G)$, and $n=\lcm\{\ord(g)\colon g\in G\}$.

\smallskip
1. Let $g \in G$. Clearly, ${\mathfrak p}_g$ is a prime $s$-ideal of $H$. Since $g^{[\ord (g)]} \in {\mathfrak p}_g \setminus {\mathfrak p}_h$ for all $h \in G \setminus \{g\}$, it follows that ${\mathfrak p}_g \ne {\mathfrak p}_h$ and ${\mathfrak p}_g \subsetneq {\mathfrak p}_h$ for all $h \in G \setminus \{g\}$. Thus it remains to show the following claim.

\begin{enumerate}
\item[{\bf A.}\,] Let ${\mathfrak p} \in s$-$\spec (H)$. Then there is a $g \in G$ such that ${\mathfrak p}_g \subset {\mathfrak p}$.
\end{enumerate}

{\it Proof of} \,{\bf A}.\,  Let $A = g_1 \bdot \ldots \bdot g_k \in {\mathfrak p}$.
 Then
\[
A^{[n]} = \Big( g_1^{[\ord (g_1)]} \Big)^{[n/\ord (g_1)]} \bdot \ldots \bdot \Big( g_k^{[\ord (g_k)]} \Big)^{[n/\ord (g_k)]} \in {\mathfrak p} \,,
\]
whence there is some $g \in \{g_1, \ldots, g_k\}$ such that $g^{[\ord (g)]} \in {\mathfrak p}$. We assert that ${\mathfrak p}_g \subset {\mathfrak p}$. Assume to the contrary that there is some $B \in  {\mathfrak p}_g \setminus {\mathfrak p}$, say $B = g \bdot h_2 \bdot \ldots \bdot h_{\ell}$. Since $g^{[\ord (g)]} \in {\mathfrak p}$, it follows that
\[
B^{[n]} =  \Big( g^{[\ord (g)]} \Big)^{[n/\ord (g)]} \bdot  \Big( h_2^{[\ord (h_2)]} \Big)^{[n/\ord (h_2)]} \bdot \ldots \bdot \Big( h_{\ell}^{[\ord (h_{\ell})]} \Big)^{[n/\ord (h_{\ell})]} \in {\mathfrak p} \,,
\]
whence $B \in {\mathfrak p}$, a contradiction. \qed[{\bf Proof of {\bf A.}}]

\smallskip
2. We proceed in three steps.

(i) Let $a=\frac{s_1}{s_2}\in \widehat{H}$, where $s_1,s_2\in H$. Then there exists $c\in H$ such that $ca^n\in H\subset F$ for all $n\in \N$. Since $F$ is completely integrally closed, we have $a\in F$ and $\pi(a)\subset\{xy^{-1}\colon x\in \pi(s_1)\text{ and }y\in \pi(s_2)\}\subset G'$. Thus we have $\widehat{H}\subset \{a\in F\colon \pi(a)\subset G'\}$.

(ii) In order to prove that  $\{a\in F\colon \pi(a)\subset G'\} \subset \bigcap_{{\mathfrak p} \in \mathfrak X (H)} H_{\mathfrak p}$ it suffices to verify that  $f\in \bigcap_{{\mathfrak p} \in \mathfrak X (H)} H_{\mathfrak p}$ for every $f\in G'$.  Since $G'=\langle ghg^{-1}h^{-1}\colon g,h\in G\rangle$, it is sufficient to show that $ghg^{-1}h^{-1} \in \bigcap_{{\mathfrak p} \in \mathfrak X (H)} H_{\mathfrak p}$ for all $g,h \in G$. Let $g, h \in G$ and $f = ghg^{-1}h^{-1}$. If $f=1$, then $f\in \bigcap_{{\mathfrak p} \in \mathfrak X (H)} H_{\mathfrak p}$. Suppose $f\neq 1$.
Since
\begin{align*}
f=\frac{f\bdot h^{[n]}\bdot g\bdot  g^{-1}}{h^{[n]}\bdot g\bdot  g^{-1}}=\frac{f\bdot h\bdot h^{-1}\bdot g^{[n]}}{h\bdot h^{-1}\bdot g^{[n]}}=\frac{f\bdot h\bdot h^{-1}\bdot (hg)^{[n]}}{h\bdot h^{-1}\bdot (hg)^{[n]}}=\frac{f\bdot g\bdot g^{-1}\bdot (hg)^{[n]}}{g\bdot g^{-1}\bdot (hg)^{[n]}}&\in \mathsf q(H)\,,\\
\text{ and }\quad f\bdot h^{[n]}\bdot g\bdot  g^{-1},\quad f\bdot h\bdot h^{-1}\bdot g^{[n]},\quad f\bdot h\bdot h^{-1}\bdot (hg)^{[n]},\quad f\bdot g\bdot g^{-1}\bdot (hg)^{[n]}&\in  H\,,
\end{align*} we have $f\in H_{{\mathfrak p}_x}$ for all $x\in G\setminus\big(\{g, g^{-1}, h\}\cap\{h,h^{-1}, g\}\cap \{h, h^{-1}, hg\}\cap \{g,g^{-1}, hg\}\big)=G$. Thus the assertion follows from 1.

(iii) Let $a=\frac{s_1}{s_2}\in \bigcap_{{\mathfrak p} \in \mathfrak X (H)} H_{\mathfrak p}$, where $s_1,s_2\in H$. Then for every $g\in G$, $a\in H_{{\mathfrak p}_g}$ implies that $\mathsf v_g(s_1)\ge \mathsf v_g(s_2)$. Therefore $s_2\mid_F s_1$ and hence $a\in F$. By the definition of $n$, we know that $a^n\in H$.
Let $c=s_2^n\in H$. Then, for every $k\in \N$ and every $r\in [0,n-1]$, we have
$ca^{kn+r}=a^{kn}s_1^{r}s_2^{n-r}\in H$ whence $a\in \widehat{H}$.
\end{proof}


\smallskip
\begin{lemma} \label{3.2}
Let $G$ be a finite group and $G_0 \subset G$ a subset. Consider the following conditions{\rm \,:}
\begin{enumerate}
\item[(a)] $1 + \mathsf d (G_0) < \mathsf D (G_0)$.
%
\item[(b)] There exist distinct $U, V \in \mathcal A (G_0)$, $1 \ne W \in \mathcal B (G_0)$ and $m \in \N$ such that $\mathsf Z (V^{[m]}) \mathsf Z (W) \subset \mathsf Z (U^{[m]})$.

\item[(c)] $G$ is not abelian.

\end{enumerate}
Then (a) $\Rightarrow$ (b) $\Leftrightarrow$ (c).
\end{lemma}

\begin{proof}
(a) $\Rightarrow$ (b) Let $U = g_0 \bdot \ldots \bdot g_{\ell} \in \mathcal A (G_0)$ with $|U|= \mathsf D (G_0)$. Since $1 + \mathsf d (G_0) < \mathsf D (G_0) = 1 + \ell$, it follows that $g_1 \bdot \ldots \bdot g_{\ell}$ is not product-one free. Thus there is $V \in \mathcal A (G_0)$ such that $U = V \bdot S$ for some $1 \ne S \in \mathcal F (G_0)$. There is
some $m \in \N_{\ge 2}$ such that $W = S^{[m]} \in \mathcal B (G_0)$, whence $U^{[m]} = V^{[m]} \bdot W$ and $\mathsf Z (V^{[m]}) \mathsf Z (W) \subset \mathsf Z (U^{[m]})$.


(b) $\Rightarrow$ (c) Let $U, V \in \mathcal A (G_0)$ and $1 \ne W \in \mathcal B (G_0)$ with $U^{[m]} = V^{[m]} \bdot W$ and assume to the contrary that $G$ is abelian.  Then $(V^{[-1]}\bdot U)^{[m]} \in \mathcal B (G)$ and since $\mathcal B (G)$ is root-closed, it follows that $V^{[-1]}\bdot U \in \mathcal B (G)$ whence $V=U$, a contradiction to $W \ne 1$.

(c) $\Rightarrow$ (b)  There exist $g,h \in G$ such that $gh \neq hg$. We consider the sequence
\[
  U = g\bdot h\bdot g^{-1}\bdot (gh^{-1}g^{-1}) \in \mathcal A(G)
\]
and set  $m = \ord (hgh^{-1}g^{-1}) \in \N$. Then $V = (g\bdot g^{-1}) \in \mathcal A (G)$, $W = (h\bdot (gh^{-1}g^{-1}))^{[m]} \in \mathcal B (G) \setminus \{1\}$, and
\[
U^{[m]} = V^{[m]}\bdot W \,. \qedhere
\]
\end{proof}

\medskip
\begin{theorem} \label{3.3}
Let $G$ be a finite group.
\begin{enumerate}
\item The following statements are equivalent{\rm \,:}
      \begin{enumerate}
      \item[(a)] $G$ is abelian.

      \item[(b)] $\mathcal B (G)$ is Krull.

      \item[(c)] $\mathcal B (G)$ is transfer Krull.

      \item[(d)] $\mathcal B (G)$ is weakly Krull.

      \item[(e)] If $U, V \in \mathcal A (G)$, $W \in \mathcal B (G)$, and $m \in \N$ such that $\mathsf Z (V^{[m]})\mathsf Z (W) \subset \mathsf Z (U^{[m]})$, then $U=V$ and $W=1$.
      \end{enumerate}

\item $\mathcal B (G)$ is seminormal if and only if $|G'| \le 2$. In particular, a dihedral group of order $2n$, where $n \ge 3$ is odd, is not seminormal.
\end{enumerate}
\end{theorem}

\begin{proof}
1. (a) $\Rightarrow$ (b) $\Rightarrow$ (c) This is obvious.

(c) $\Rightarrow$ (a) See \cite[Proposition 3.4]{Oh20a}.

(d) $\Leftrightarrow$ (b) Every Krull monoid is weakly Krull. If $\mathcal B (G)$ is weakly Krull, then Proposition \ref{3.1} implies that
\[
\mathcal B (G) =\bigcap_{\mathfrak p \in \mathfrak X \bigl( \mathcal B (G) \bigr)} {\mathcal B (G)}_{\mathfrak p} = \widehat{\mathcal B (G)} \,,
\]
whence $\mathcal B (G)$ is completely integrally closed and so $\mathcal B (G)$ is Krull.

(e) $\Leftrightarrow$ (a) This follows from Lemma \ref{3.2} (with $G_0=G$).

2. See \cite[Corollary 3.12]{Oh19b}.
\end{proof}

\smallskip
\section{On the omega invariant} \label{4}

Let $H$ be an atomic monoid. For an element $a \in H$, let $\omega (H,a)$ be the smallest $N \in \N_0 \cup \{\infty\}$ with the following property:
\begin{itemize}
\item[] If $n \in \N$ and $a_1, \ldots, a_n \in H$ with $a \mid a_1 \cdot \ldots \cdot a_n$, then there exists a subset $\Omega \subset [1,n]$ such that $|\Omega| \le N$ and $a \mid \prod_{\nu \in \Omega} a_{\nu}$.
\end{itemize}
Furthermore, we set
\[
\omega (H) = \sup \{ \omega (H,a) \colon a \in \mathcal A(H) \} \,.
\]
Thus $\omega (H,a)=1$ if and only if $a \in H$ is a prime element, and for an atomic monoid $H$ that is not a group we have $\omega (H)=1$ if and only if $H$ is factorial. If $H$ satisfies the ascending chain condition on divisorial ideals  or if $H$ is strongly primary, then $\omega (H, a) < \infty$ for all $a \in H$. Furthermore, $\omega (H) < \infty$ if and only if $H$ is globally tame (\cite[Proposition 3.6]{Ge-Ka10a}) whence $\omega (H) < \infty$ for all finitely generated monoids. If $G$ is a finite group, then we set $\omega (G) := \omega (\mathcal B (G))$ and since $\mathcal B (G)$ is finitely generated, we have $\omega (G) < \infty$. If $G$ is abelian with $|G| \ge 3$, then it is easy to see that $\omega (G) = \mathsf D (G)$. But so far the precise value of $\omega (G)$ has not been determined yet for any non-abelian group. We formulate the main result of this section.

\smallskip
\begin{theorem} \label{4.1}
Let $G$ be a dihedral group of order $2n$, where $n \ge 3$ is odd. Then $\omega (G) =\mathsf D(G)= 2n$.
\end{theorem}

The goal of this section is to prove Theorem \ref{4.1}. To do so, we
 make use of the following recent strengthenings of the Partition Theorem, formulated as Propositions \ref{thm-partition-thm-equi} and \ref{thm-special-dihedral-ample}, as well as a basic lemma from \cite{Gr21a}. Their proofs are given in  \cite[Theorem 3.2, Theorem 3.3, Lemma 2.4]{Gr21a} (for simplicity, we state here only the cyclic case in Proposition \ref{thm-partition-thm-equi}). A setpartition $\mathcal A=A_1\bdot\ldots\bdot A_n$ is a sequence of finite and nonempty subsets $A_i\subseteq G$. Then $\mathsf S(\mathcal A)=\prod_{i\in [1,n]}^\bullet \prod_{g\in A}^\bullet g\in \Fc(G)$ is the corresponding sequence of terms from $G$ partitioned by the sets $A_i$ in $\mathcal A$. Clearly, $\Sum{i=1}{n}A_i\subset \Sigma_n(S)$, where $S=\mathsf S(\mathcal A)$.

\begin{proposition} \label{thm-partition-thm-equi}
Let $G$ be a cyclic group, let $n\geq 1$, let $X\subset G$ be a finite, nonempty subset, let  $L\leq  \mathsf H(X)$, let $S\in \Fc(G)$ be a sequence, and let $S'\mid S$ be a subsequence with $\mathsf h(\phi_L(S'))\leq n\leq |S'|$. Suppose $|S'|\leq 2n$. Then there exists a  setpartition $\mathscr A=A_1\bdot\ldots\bdot A_n$  with $\mathsf S(\mathscr A)\mid S$, \ $|\mathsf S(\mathscr A)|=|S'|$ and $|\phi_L(A_i)|=|A_i|\leq 2$ for all $i$ such that either \begin{enumerate}

    \item[1.] $|X+\Sigma_n(S)|\geq |X+\Sum{i=1}{n}A_i|\geq (|S'|-n)|L|+|X|$, or
\item[2.] there is  a subgroup $K\leq H=\mathsf H(X+\Sigma_n(S))$ with $L<K$ proper and $\alpha\in G$ such that \begin{itemize}
\item[(a)] $X+\Sigma_n(S)=X+\Sum{i=1}{n}A_i$,
\item[(b)]  $\supp(\mathsf S(\mathscr A)^{[-1]}\bdot S)\subset \alpha+K=\bigcap_{i=1}^n(A_i+K)$ and  $|A_i\setminus (\alpha+K)|\leq 1$ for all $i$,
\item[(c)] $|X+\Sigma_n(S)|\geq |X+H|+|S_{G\setminus (\alpha+H)}|\cdot|H|$ and
 $|X+\Sigma_n(S)|\geq |X+K|+|S_{G\setminus (\alpha+K)}|\cdot|K|$,
\item[(d)] $L+\Summ{i\in I_K}A_i=\alpha |I_K|+K$, where $I_K\subset [1,n]$ is the nonempty subset of all $i\in [1,n]$ with $A_i\subset \alpha+K$.
\end{itemize}
\end{enumerate}
\end{proposition}

\begin{proposition} \label{thm-special-dihedral-ample}
Let $G$ be an abelian group, let $n\geq 1$, and let $S\in \Fc(G)$ be a sequence with $|S|>n$. Suppose $|\Sigma_n(S)|\leq m+1$, where $m=\min\{n,|S|-n,|S|-\mathsf h(S)\}$. Then one of the following holds, with Items 1--4 only possible if $|\Sigma_n(S)|=m+1$ or $|\supp(S)|=1$.
\begin{itemize}
\item[1.] $n=2$, $|S|=|\supp(S)|$, and $\supp(S)=x+K$ for some $K\leq G$ and $x\in G$ with $K\cong (\Z/2\Z)^2$.
\item[2.] $m=2$ and $\supp(S)=x+K$ for some $K\leq G$ and $x\in G$ with $K\cong \Z/3\Z$.
\item[3.] $|\supp(S)|\leq 2$.
\item[4.] $\supp(S)\subset \{x-d,x,x+d\}$ for some $x,\,d\in G$ with $\vp_x(S)=\mathsf h(S)\geq \mathsf |S|-m$.
\item[5.] There exists $x\in G$ and  a setpartition $\mathscr A=A_1\bdot\ldots\bdot A_n$ with
 $\mathsf S(\mathscr A)\mid S$, $|\mathsf S(\mathscr A)|=n+m$,
$\Sum{i=1}{n}A_i=\Sigma_n(S)$,  $\supp(\mathsf S(\mathscr A)^{[-1]}\bdot S)\subset x+H$, $|A_i|\leq 2$ and $(x+H)\cap A_i\neq \emptyset$ for all $i\in [1,n]$, and $|\Sum{i=1}{n}A_i|=|\underset{i\neq j}{\Sum{i=1}{n}}A_i|$ for some $j\in [1,n]$, where  $H=\mathsf H(\Sigma_n(S))$ is nontrivial.
\end{itemize}
\end{proposition}

\smallskip
\begin{lemma}\label{lem-subsums=sumset}
Let $G$ be an abelian group, let $n\geq 0$, let $X\subset G$ be a finite, nonempty subset,   let $S\in \Fc(G)$ be a sequence, let $H\leq G$, and let $x\in G$. Suppose  $\mathscr A=A_1\bdot\ldots\bdot A_n$ is a setpartition with $\mathsf S(\mathscr A)\mid S$, \  $\supp(\mathsf S(\mathscr A)^{[-1]}\bdot S)\subset x+H\subset \bigcap_{i=1}^n(A_i+H)$,  \ $|A_i\setminus (x+H)|\leq 1$ for all $i$,   and $H\leq \mathsf H(X+\Sum{i=1}{n}A_i)$. Then $X+\Sigma_{\ell}(S)=X+\Sum{i=1}{n}A_i+(\ell-n)x$ for any $\ell\in [n,n+|\mathsf S(\mathscr A)^{[-1]}\bdot S|]$.
\end{lemma}

Let $G$ be a dihedral group of order $2n$ with $n \ge 3$, say $G=\la \alpha,\,\tau:\; \alpha^n=\tau^2=1_G,\; \tau \alpha=\alpha^{-1}\tau\ra$. Then $\la \alpha\ra$ is a cyclic subgroup of index $2$. The commutator subgroup $G'=\la \alpha^2\ra$ is a cyclic group of order $n$ (when $n$ is odd) or order $\frac{n}{2}$ (when $n$ is even).
Let $S\in \Fc(G)$ be a sequence of terms from $G$. We have a natural partition $S=S_{\la \alpha\ra}\bdot S_{\tau\la \alpha\ra}$, where $S_{\la \alpha\ra}$ consists of all terms $\alpha^x\in \la \alpha\ra$ and $S_{\tau\la \alpha\ra}$ consists of all terms $\tau\alpha^y\in \tau\la \alpha\ra$,
where $x,\,y\in \Z$.
For $x\in \Z/n\Z$, let $\alpha^x$ be $\alpha^{x_0}$, where $x_0$ is any integer representative for $x$ modulo $n$.
The additive cyclic group $\Z/n\Z$ and the multiplicative cyclic group $\la \alpha\ra$ can be identified via the isomorphism  $\cdot ^*:\Z/n\Z\rightarrow \la \alpha\ra$ defined by
$x^*=\alpha^x$. The inverse isomorphism $\cdot^+:\la \alpha\ra\rightarrow \Z/n\Z$ is defined by $(\alpha^x)^+=x\mod n$. The notation is chosen so that $x^*$ lives in the multiplicative cyclic group $\la \alpha\ra$, while $g^+$ lies in the additive cyclic group $\Z/n\Z$. We extend the definition of $\cdot^+$ to all of $G$ by setting $(\tau \alpha^y)^+=y\mod n$. The definitions of $\cdot^*$ and $\cdot^+$ depend on the  fixed generating set $\{\alpha,\tau\}$ for $G$, with the map $\cdot ^*$ only depending on $\alpha$. If we exchange $\{\alpha,\tau\}$ for an alternative generating set, then the definitions of $\cdot^*$ and $\cdot^+$ are implicitly altered as well.  The maps $\cdot^*$ and $\cdot^+$ are extended to sequences/sets in the usual fashion of applying the corresponding map to each term/element. The effect of replacing the generator $\tau$ by $\tau \alpha^y$ is to translate all terms of $(\tau\la \alpha\ra)^+$ by $-y$.
To avoid confusion, when dealing with the dihedral group $G$, all subgroups of $\Z/n\Z$ will be notated in the form $K^+$ for the appropriate isomorphic subgroup  $K\leq \la \alpha\ra$. This will allow immediate visual recognition of whether a subgroup lies in the additive cyclic group $\Z/n\Z$ or in the multiplicative cyclic group $\la\alpha\ra$, and provides a strong visual  connection between the  linked  subgroups $K$ and $K^+$. Additionally, the map $\cdot ^+$ provides a one-to-one correspondence between the subgroups $H\leq G$ with $H\not\leq \la \alpha\ra$ and all subgroup-coset pairs $(K^+,y+K^+)$, where $K^+\leq \Z/n\Z$ and $y\in \Z/n\Z$, as follows.
 For a subgroup $K\leq \la \alpha\ra$ and $y\in \Z/n\Z$, we let $K_y=\la K,\tau\alpha^y\ra$.
 Note every $H\leq G$ with $H\not\leq \la \alpha\ra$ has some $\tau \alpha^y\in H$, where $y\in \Z/n\Z$, and then $H=\la H\cap \la\alpha\ra,\tau \alpha^y\ra$, ensuring that $H=K_y$ with $K=H\cap \la\alpha\ra$. The subgroup $K=H\cap \la \alpha\ra=K_y\cap \la \alpha\ra$ is uniquely defined. Since  all $\tau \alpha^z\in H$ with $z\in \Z/n\Z$ have $z\in y+(H\cap \la\alpha\ra)^+=y+K^+$, the  element $y\in \Z/n\Z$ is uniquely defined modulo $K^+$. Thus the map $H\mapsto ((H\cap \la\alpha\ra)^+,y+(H\cap \la\alpha\ra)^+)$ is well-defined, and clearly bijective as its inverse is the map $(K^+,y+K^+)\mapsto K_y$. In light of this, we will often denote subgroups of $G$ not contained in $\la \alpha\ra$ in the form $K_y$ for some $K\leq \la\alpha\ra$ and $y\in \Z/n\Z$. Moreover, when this is the case, we note that
 $$K_y^+=K^+\cup (K_y\setminus K)^+=K^+\cup (y+K^+)$$ with
 $K'_y=K_y\cap \la\alpha\ra=K$ (if $n$ is odd), and $K_y'=K^2=(2K^+)^*$ (if $n$ is even). When $y=0$ (equivalently, if we choose our generating set to be $\{\alpha,\tau\alpha^y\}$), then $$K_0\setminus K=\tau K\quad\und\quad K_0^+=K^+.$$ To help lighten the notation, we also use $\phi_K:\Z/n\Z\rightarrow (\Z/n\Z)/K^+$ to  denote the natural homomorphism modulo $K^+$.

The following proposition shows how the computation of $\pi(S)$ reduces to an additive question in $\Z/n\Z$ combining $\pm$ weighted subsums alongside ordinary $\lfloor \ell/2\rfloor$-term subsums.

\begin{proposition}\label{prop-add/mult-DefiningCorrespondance}
Let $G$ be a dihedral group of order $2n$ where $n \ge 3$, say $G=\la \alpha,\,\tau:\; \alpha^n=\tau^2=1_G,\; \tau \alpha=\alpha^{-1}\tau\ra$, and let $S\in \Fc(G)$ with $S_{\la\alpha\ra}=\alpha^{x_1}\bdot\ldots\bdot \alpha^{x_s}$, where $x_1,\ldots,x_s\in \Z/n\Z$.

\begin{itemize}
\item[1.] If  $S\in \Fc(\la \alpha\ra)$, then
$\pi(S)=\{\sigma(S^+)^*\}$.
\item[2.] If  $S\notin \Fc(\la \alpha\ra)$, then
$$\pi(S)=\tau^\ell\Big(\{x_1,-x_1\}+\ldots+\{x_s,-x_s\}+(-1)^\ell\Sigma_{\lfloor \ell/2\rfloor}(2S^+_{\tau\la \alpha\ra})
-(-1)^\ell\sigma(S^+_{\tau\la \alpha\ra})\Big)^*,$$
where $\ell=|S_{\tau\la \alpha\ra}|\geq 1$.
\end{itemize}
\end{proposition}

\begin{proof}
Item 1 is clear as $\la \alpha\ra$ is abelian. For Item 2, consider an arbitrary ordered product of all $s+\ell$ terms of $S$, say $g_1\cdot\ldots\cdot g_{s+\ell}\in \pi(S)$. By the defining relations for the dihedral group, we have $g_1\cdot\ldots\cdot g_{s+\ell}=\tau^\ell(\pm g_1^+\pm\ldots\pm g_{s+\ell}^+)^*$, where the sign of each $g_i$ depends upon the number of terms from $S_{\tau\la \alpha\ra}$ contained in $g_1\bdot\ldots\bdot g_i$: if the number of terms from $\tau\la \alpha\ra$ contained in $g_1\bdot\ldots\bdot g_i$ is congruent to $\ell$ modulo $2$, then it is positive, while if it congruent to $\ell+1$, then it is negative. Since $\ell\geq 1$ (in view of the hypothesis $S\notin \Fc(\la \alpha\ra)$), each term $\alpha^{x_j}$ from $S_{\la \alpha\ra}$ can be placed either in an even or odd slot relative to the fixed ordering of the sequence $S_{\tau\la \alpha\ra}$ in the product, the even slots being those places $i\in [1,s+\ell]$ where there are an even number of terms from $\tau\la \alpha\ra$ contain in $g_1\bdot\ldots\bdot g_i$, and the odd slots $i\in [1,s+\ell]$ being those for which the number of such terms is odd. The effect of moving $\alpha^{x_j}$ between an even and odd slot is to simply change its sign in the sum. There must be exactly $\lceil \ell/2\rceil$ terms from $S_{\tau\la \alpha\ra}$ placed in odd slots, and exactly $\lfloor  \ell/2\rfloor$ placed in even slots.
Thus the elements of $\pi(S)$ are those  from the sets $$\tau^\ell\Big((-1)^{\ell}\big(\sigma(T^+_{even})-\sigma(T^+_{odd})\big)+
\{x_1,-x_1\}+\ldots+\{x_s,-x_s\}\Big)^*$$ as we range over all partitions $S_{\tau\la \alpha\ra}=T_{odd}\bdot T_{even}$ with $|T_{even}|=\lfloor \ell/2\rfloor$. Observing that $\sigma(T_{odd}^+)=
\sigma(S_{\tau\la \alpha\ra}^+)-\sigma(T_{even}^+)$, we find that the elements of $\pi(S)$ are those from the sets $$\tau^\ell\Big((-1)^{\ell}\big(\sigma(2T^+_{even})-\sigma(S_{\tau\la \alpha\ra}^+)\big)+
\{x_1,-x_1\}+\ldots+\{x_s,-x_s\}\Big)^*$$ as we range over all subsequences $T^+_{even}\mid S^+_{\tau\la\alpha\ra}$ with $|T^+_{even}|=\lfloor \ell/2\rfloor$, which yields the desired result.
\end{proof}

\begin{corollary}\label{cor-zero-sum-defining}
Let $G$ be a dihedral group of order $2n$ where $n\geq 3$, say $G=\la \alpha,\,\tau:\; \alpha^n=\tau^2=1_G,\; \tau \alpha=\alpha^{-1}\tau\ra$, and let $S\in \Fc(G)$ with $S_{\la\alpha\ra}=\alpha^{x_1}\bdot\ldots\bdot \alpha^{x_s}$, where $x_1,\ldots,x_s\in \Z/n\Z$. Suppose $S\notin \Fc(\la \alpha\ra)$.  Then $S\in \mathcal B(G)$ if and only if $|S_{\tau\la\alpha\ra}|=2\ell$ is even and
$$0\in  \{x_1,-x_1\}+\ldots+\{x_s,-x_s\}+\Sigma_{\ell}(2S^+_{\tau\la \alpha\ra})-\sigma(S^+_{\tau\la \alpha\ra}).$$
\end{corollary}

\begin{proof}
This follows immediately from Proposition \ref{prop-add/mult-DefiningCorrespondance}.
\end{proof}

\begin{lemma}\label{lem-triv-multbound}
Let $G$ be a dihedral group of order $2n$ where $n\geq 3$, say $G=\la \alpha,\,\tau:\; \alpha^n=\tau^2=1_G,\; \tau \alpha=\alpha^{-1}\tau\ra$. If $U\in \mathcal A(G)$  with $|U_{\tau\la \alpha\ra}|>2$, then $\mathsf h(U_{\tau\la \alpha\ra})\leq \frac12|U_{\tau\la \alpha\ra}|$.
\end{lemma}

\begin{proof}
Since $U\in\mathcal A(G)$ is product-one and $|U_{\tau\la \alpha\ra}|>2$, we have $|U_{\tau\la \alpha\ra}|=2\ell$ even with $\ell \geq 2$. Assume by contradiction that $\mathsf h( U_{\tau\la \alpha\ra})\geq \frac12|U_{\tau\la \alpha\ra}|+1=\ell+1$. By replacing the generator $\tau$ by an appropriate alternative generator $\tau\alpha^x$, we can w.l.o.g. assume that $\vp_{\tau}(U)=\mathsf h( U_{\tau\la \alpha\ra})\geq \ell+1$.
 Then $\tau^{[2]}\mid U$ is a product-one subsequence. Since $|U|\geq |U_{\tau\la \alpha\ra}|=2\ell\geq 4$, it follows that $V:=U\bdot \tau^{[-2]}$ is a nontrivial sequence.
Since $\vp_{\tau}(U)\geq \ell+1=\frac12|U_{\tau\la \alpha\ra}|+1$, every $\ell$-term subsequence of $2U^+_{\tau\la \alpha\ra}$ must contain at least one term equal to $0$, and $\vp_{\tau}(V)\geq \ell-1$. It follows that $\Sigma_\ell(2U^+_{\tau\la \alpha\ra})=\Sigma_{\ell-1}(2V^+_{\tau\la \alpha\ra})$ and $\sigma(U^+_{\tau\la \alpha\ra})=\sigma(V^+_{\tau\la \alpha\ra})$. Thus $U\in \mathcal A(G)$ combined with Corollary \ref{cor-zero-sum-defining} applied to $U$ and $V$ (possible as  $U,\,V\notin \Fc(\la \alpha\ra)$ follows from $\ell\geq 2$) shows that $V\in \mathcal B(G)$, and now $U=(\tau^{[2]})\bdot V$ is a factorization of  $U$ into two nontrivial
product-one subsequences, contradicting that $U\in \mathcal A(G)$ is an atom.
\end{proof}

\begin{lemma}\label{lem-short}
Let $G$ be an abelian group, let $X,\,Y\subset G$ be finite subsets with $H=\mathsf H(X)$, $K=\mathsf H(Y)$ and $X+H=Y+H$. Then $K\leq H$.
\end{lemma}

\begin{proof}
 We have $X+H+K=Y+H+K=Y+K+H=Y+H=X+H=X$. Thus $H+K\leq \mathsf H(X)=H$, implying $K\leq H$.
\end{proof}

\begin{lemma} \label{lem-HinAlpha}
Let $G$ be a dihedral group of order $2n$ where $n\geq 3$, say $G=\la \alpha,\,\tau:\; \alpha^n=\tau^2=1_G,\; \tau \alpha=\alpha^{-1}\tau\ra$. If
 $H\leq \la \alpha\ra$ is a subgroup and $U\in \mathcal A(G)$ with $|U|>1$, then $|U_H|\leq 2|H|-2$.
\end{lemma}

\begin{proof}
As $|U|>1$, we see that $U$ is not the atom consisting of a single term equal to $1_G$, which ensures $1_G \notin \supp(U)$. Thus the lemma holds for $H$ trivial, and we may assume $|H|>1$. If $U\in \mathcal F(\la \alpha\ra)$, then $\la\supp(U)\ra$ is abelian, whence $|U_H|\leq \mathsf D(H)=|H|\leq 2|H|-2$, as desired. Therefore, we may assume $|U_{\tau \la\alpha\ra}|>0$, allowing us to use Proposition \ref{prop-add/mult-DefiningCorrespondance}.2. Then $|U_{\tau\la \alpha\ra}|\geq 2$ is even and there exists an ordering of the terms of $U$ whose product is one, say w.l.o.g. (as $\tau \alpha^{z_i}\tau\alpha^{z_{i+1}}\in \la \alpha\ra$ commutes with all terms $g,h\in \la \alpha\ra$) $$U=g_1\bdot\ldots\bdot g_r\bdot \tau \alpha^{z_1}\bdot h_1\bdot\ldots\bdot h_s\bdot  \tau\alpha^{z_2}\bdot\ldots\bdot \tau \alpha^{z_\ell},$$ where $\ell=|U_{\tau\la\alpha\ra}|$ and $U_{\la \alpha\ra}=g_1\bdot\ldots\bdot g_r\bdot h_1\bdot\ldots\bdot h_s$ with $r,s\geq 0$.
If $|U_{H}|\geq 2|H|-1$, then the pigeonhole principle ensures either $g_1\bdot\ldots\bdot g_r$ or $ h_1\bdot\ldots\bdot h_s$ contains at least $|H|=\mathsf D(H)$ terms from $H$. Thus, re-ordering the terms of $g_1\bdot\ldots\bdot g_r$ or $h_1\bdot\ldots\bdot h_s$ appropriately, we find a consecutive nontrivial product-one sequence in $g_1\bdot\ldots\bdot g_r$ or $h_1\bdot\ldots\bdot h_s$, forcing the complement of this sequence in $U$ to also have product-one, which contradicts  that $U$ is an atom in view of $\ell=|U_{\tau\la \alpha\ra}|>0$. Therefore $|U_H|\leq 2|H|-2$, as desired.\end{proof}


\begin{lemma}
\label{lem-yelp}
Let $G$ be an abelian group, let $S\in \Fc(G)$ be a sequence, let $X\subset G$ be a finite, nonempty set, let $H\leq G$, and let $U\bdot V\mid S$. Suppose $\supp(S\bdot(U\bdot V)^{[-1]})\subset H$, $|U_H|\geq |U|/2$, $|V_H|\geq |V|/2$,   and $X+\Sigma_{\lfloor |U|/2\rfloor}(U)+\Sigma_{\lfloor |V|/2\rfloor}(V)$ is $H$-periodic. Then $$X+\Sigma_{\lfloor |S'|/2\rfloor}(S')= X+\Sigma_{\lfloor |U|/2\rfloor}(U)+\Sigma_{\lfloor |V|/2\rfloor}(V)=X+(\Sigma(S)\cup \{0\})$$ for any subsequence $S'\mid S$ with $U\bdot V\mid S'$ and either $|U|$ even, $|V|$ even or $|U\bdot V|<|S'|$.
\end{lemma}

\begin{proof}
Let us begin by showing \be\label{teetop} X+\Sigma_{\lfloor |U|/2\rfloor}(U)+\Sigma_{\lfloor |V|/2\rfloor}(V)=X+(\Sigma(S)\cup \{0\}).\ee
The inclusion $ X+\Sigma_{\lfloor |U|/2\rfloor}(U)+\Sigma_{\lfloor|V|/2\rfloor}(V)\subset X+(\Sigma(S)\cup \{0\})$ is trivial in view of $U\bdot V\mid S$. Since $X+\Sigma_{\lfloor |U|/2\rfloor}(U)+\Sigma_{\lfloor|V|/2\rfloor}(V)$ is $H$-periodic, it suffices to prove  $\Sigma(S)\cup \{0\}\subset \Sigma_{\lfloor |U|/2\rfloor}(U)+\Sigma_{\lfloor|V|/2\rfloor}(V)$  holds modulo $H$.
Let $T\mid S$ be an arbitrary (possibly trivial) subsequence such that all terms of $T$ are nonzero modulo $H$. In view of the hypotheses $\supp(S\bdot(U\bdot V)^{[-1]})\subset H$, $|U_H|\geq |U|/2$, and $|V_H|\geq |V|/2$, we have  $T=T_U\bdot T_V$ for some $T_U\mid U$ and $T_V\mid V$ with $|T_U|\leq |U|/2$ and $|T_V|\leq |V|/2$. Moreover, $|U_H|\geq |U|/2$ and $|V_H|\geq |V|/2$ ensure  there are at least  $|U|/2$ terms in $U$ which are zero modulo $H$, and at least $|V|/2$ terms in $V$ which are zero modulo $H$.  It follows that we can extend the sequence $T_U\mid U$ to a subsequence $T'_U\mid U$ of length $|T'_U|=\lfloor |U|/2\rfloor$ by concatenating an additional $\lfloor |U|/2\rfloor-|T_U|$ terms from $U\bdot T_U^{[-1]}$, each zero modulo $H$. Likewise, we can extend the sequence $T_V\mid V$ to a subsequence $T'_V\mid V$ of length $|T'_V|=\lfloor|V|/2\rfloor$ by concatenating an additional $\lfloor|V|/2\rfloor-|T_V|$ terms from $V$, each zero modulo $H$. Let $T'=T'_U\bdot T'_V$. As we have only extended the sequences by terms zero modulo $H$, it follows that $\sigma(T)\equiv\sigma(T')\mod H$.
By construction, $\sigma(T')\in \Sigma_{\lfloor |U|/2\rfloor}(U)+\Sigma_{\lfloor|V|/2\rfloor}(V)$. Consequently, since $T\mid S$ was an arbitrary subsequence of terms nonzero modulo $H$, we conclude that the inclusion $\Sigma(S)\cup \{0\}\subset \Sigma_{\lfloor |U|/2\rfloor}(U)+\Sigma_{\lfloor |V|/2\rfloor}(V)$ holds modulo $H$, which establishes \eqref{teetop} as noted earlier.

For any subsequence $S'\mid S$, the inclusion $\Sigma_{\lfloor |S'|/2\rfloor}(S')\subset \Sigma(S)\cup \{0\}$ holds trivially. For any subsequence $S'\mid S$ with $U\bdot V\mid S'$ and either $|U|$ even, $|V|$ even or $|U\bdot V|<|S'|$, we have  $|S'|-|U|-|V|\geq \lfloor \frac{|S'|}{2}\rfloor-\lfloor \frac{|U|}{2}\rfloor -\lfloor \frac{|V|}{2}\rfloor\geq 0$, ensuring that the set $\Sigma_{\lfloor |S'|/2\rfloor}(S')$  contains a translate of  $\Sigma_{\lfloor |U|/2\rfloor}(U)+\Sigma_{\lfloor|V|/2\rfloor}(V)$. Hence \eqref{teetop} implies
$$X+\Sigma_{\lfloor |S'|/2\rfloor}(S')=X+\Sigma_{\lfloor |U|/2\rfloor}(U)+\Sigma_{\lfloor|V|/2\rfloor}(V)=X+(\Sigma(S)\cup \{0\})$$
for any sequence $S'\mid S$ with $U\bdot V\mid S'$ and either $|U|$ even, $|V|$ even or $|U\bdot V|<|S'|$, completing the proof.
\end{proof}

\begin{proposition}\label{prop-regulate}
Let $G$  be a dihedral group of order $2n$ where $n \ge 3$, say $G=\la \alpha,\,\tau:\; \alpha^n=\tau^2=1_G,\; \tau \alpha=\alpha^{-1}\tau\ra$, and let $H_x=\la H,\tau\alpha^x\ra\leq G$ be a subgroup with $H\leq \la\alpha\ra$ and $x\in \Z/n\Z$.
\begin{itemize}
\item[1.] Suppose $V\in \Fc(G)$ has a decomposition $V_{\tau\la\alpha\ra}^+=T_1\bdot\ldots\bdot T_\ell$ such that $\ell\geq 1$,
$|T_i|=2$ for all $i$, $X+\Sum{i=1}{\ell}A_i$ is $H^+$-periodic, and $ A_i\cap (2x+H^+)\neq \emptyset$ for all $i$,
 where  $A_i=\supp(2T_i)$ for all $i$,
 $2\ell=|V_{\tau\la\alpha\ra}|$,   $V^+_{\la \alpha\ra}=x_1\bdot\ldots\bdot x_s$ and $X=\{x_1,-x_1\}+\ldots+\{x_s,-x_s\}$. Then $$\pi(V)^+=X+\Sigma_{\ell}(2V_{\tau\la\alpha\ra}^+)-
 \sigma(V_{\tau\la\alpha\ra}^+)=X+\Sum{i=1}{\ell}A_i-\sigma(V_{\tau\la\alpha\ra}^+)$$ is $H^+$-periodic and $|V_{H_x\setminus H}|\geq \frac12 |V_{\tau \la \alpha\ra}|=\ell$.

\item[2.] Suppose $U,\,V\in \Fc(G)$ with $\pi(V)^+$ $H^+$-periodic,  $|V_{H_x\setminus H}|\geq \frac12 |V_{\tau \la \alpha\ra}|>0$ and $\supp(U)\subset H_x$. Then $\pi(U\bdot V)$ is a translate of $\pi(V)$. In particular, if $U$ is product-one, then $\pi(U\bdot V)=\pi(V)$.
\end{itemize}
\end{proposition}

\begin{proof}
1. The hypotheses  $|T_i|=2$ and $ A_i\cap (2x+H^+)\neq \emptyset$ for all $i$ ensure $|V_{H_x\setminus H}|\geq \frac12 |V_{\tau \la \alpha\ra}|$.
Applying Lemma \ref{lem-subsums=sumset} to the sub-sumset of $\Sum{i=1}{\ell}A_i$ consisting of all cardinality two summands yields $X+\Sum{i=1}{\ell}A_i= X+\Sigma_\ell(2V_{\tau \la\alpha\ra}^+)$. Since $\ell\geq 1$, Item 1 now follows by Proposition \ref{prop-add/mult-DefiningCorrespondance}.

2. By replacing the generating set $\{\alpha,\tau\}$ by $\{\alpha,\tau\alpha^x\}$, we can  w.l.o.g. assume $x=0$. We can also assume $U$ is nontrivial, else the item holds trivially.  Let $\ell=|V_{\tau\la\alpha\ra}|$. It follows in view of Proposition \ref{prop-add/mult-DefiningCorrespondance} and the hypothesis $\ell>0$ that  $\pi(V)^+=(-1)^\ell\Big(X+\Sigma_{\lfloor \ell/2\rfloor}(2V^+_{\tau\la\alpha\ra})-\sigma(V_{\tau\la\alpha\ra}^+)\Big)$, where
$V_{\la\alpha\ra}^+=x_1\bdot\ldots \bdot x_s$ and $X=\{x_1,-x_1\}+\ldots+\{x_s,-x_s\}$. Thus $X+\Sigma_{\lfloor \ell/2\rfloor}(2V^+_{\tau\la\alpha\ra})$ is $H^+$-periodic by hypothesis.
Since $x=0$, we have $|(2V_{\tau \la \alpha\ra}^+)_{H^+}|=|V_{H_0\setminus H}|\geq \frac12 |V_{\tau \la \alpha\ra}|=\frac12 \ell>0$ by hypothesis. The hypothesis $\supp(U)\subset H_x=H_0$ ensures $\supp(U^+)\subset H^+$.
 Thus Lemma \ref{lem-yelp} (applied with $V$ taken to be $2V_{\tau \la\alpha\ra}^+$, $U$ taken to be the trivial sequence, and $S'=S=2(U\bdot V)_{\tau\la\alpha\ra}^+$) yields \be\label{periodic}X+\Sigma_{\lfloor \ell/2\rfloor}(2V^+_{\tau\la\alpha\ra})=X+\Sigma_{\lfloor\ell'/2\rfloor}(2(U\bdot V)_{\tau\la\alpha\ra}^+),\ee where $\ell'=|(U\bdot V)_{\tau \la\alpha\ra}|$. Note the set in \eqref{periodic} is $H^+$-periodic by hypothesis.  Let $U^+_{\la\alpha\ra}=y_1\bdot\ldots\bdot y_r$ and $Y=\{y_1,-y_1\}+\ldots+\{y_r,-y_r\}$. Since $\supp(U^+)\subset H^+$, we have $Y\subset H^+$,
 whence it follows that  $X+\Sigma_{\lfloor\ell'/2\rfloor}(2(U\bdot V)_{\tau\la\alpha\ra}^+)=Y+X+\Sigma_{\lfloor\ell'/2\rfloor}(2(U\bdot V)_{\tau\la\alpha\ra}^+)$ as this set is $H^+$-periodic. As a result, in view of \eqref{periodic}, $\ell>0$ and Proposition \ref{prop-add/mult-DefiningCorrespondance}, it follows that $\pi(V)$ and $\pi(U\bdot V)$ are translates of each other, as desired.
\end{proof}

\begin{lemma}
\label{lem-easycond}
Let $G$ be a dihedral group of order $2n$ where $n\geq 3$ is odd, say $G=\la \alpha,\,\tau:\; \alpha^n=\tau^2=1_G,\; \tau \alpha=\alpha^{-1}\tau\ra$, and let $U\in \mathcal F(G)$.
Let $K_y=\la K,\tau \alpha^y\ra\leq G$ be a subgroup  with  $K\leq \la \alpha\ra$, let  $V\mid U$ be a subsequence, let $\ell_V=|V_{\tau\la \alpha\ra}|$, and let $Z=X_V+\Sigma_{\lfloor \ell_V/2\rfloor}(2V^+_{\tau\la\alpha\ra})$, where $V_{\la \alpha\ra}^+=x_1\bdot\ldots\bdot x_s$ and $X_V=\{x_1,-x_1\}+\ldots+\{x_s,-x_s\}$. Suppose  $|V_{K_y\setminus K}|\geq \frac12 |V_{\tau\la \alpha\ra}|$, \ $Z$ is $K^+$-periodic, \ $$|(U\bdot V^{[-1]})_{K_y}|\geq |K|+1+|G'/K|-|\phi_{K}(Z)|\quad\und\quad|(U\bdot V^{[-1]})_{K_y\setminus  K}|\geq |G'/K|-|\phi_{K}(Z)|$$ with both above inequalities strict when $K=G'$.    Then $U$ is not an atom.
\end{lemma}

\begin{proof}Let $\epsilon=1$ if $K=G'$ and $\epsilon=0$ otherwise. Assume by contradiction $U\in \mathcal A(G)$. By exchanging the generating $\tau$ for $\tau\alpha^y$, we can w.l.o.g. assume $y=0$, so that $K_0=\tau K\cup K$ and $K_0^+=K^+$.
Note $|\phi_{K}(Z)|\leq |G'/K|-1+\epsilon,$ so $$\ell:=|G'/K|-|\phi_{K}(Z)|+\epsilon>0.$$
By hypothesis (note $\mathsf d(K_0)=|K|$ as $K_0$ is dihedral \cite{Ge-Gr13a}),
\be\label{weelit} |(U\bdot V^{[-1]})_{K_0}|\geq \mathsf d(K_0)+1+\ell\quad\und\quad|(U\bdot V^{[-1]})_{\tau K}|\geq \ell>0.\ee
In particular,  $U\notin \Fc(\la\alpha\ra)$, ensuring that  $U$ is not the atom consisting of a single term equal to $1_G$, whence $1_G \notin\supp(U)$.
For a subsequence $T\mid U$, let $$\ell_T=|T_{\tau\la \alpha\ra}|\quad \und\quad X_T=\{x_1,-x_1\}+\ldots+\{x_t,-x_t\},\quad\mbox{ where $T_{\la \alpha\ra}^+=x_1\bdot\ldots\bdot x_t$},$$ and set $X_T=\{0\}$ if $T_{\la \alpha\ra}$ is the trivial sequence.

\smallskip
\noindent
CASE 1. \,  $|(U\bdot V^{[-1]})_{\tau\la \alpha\ra \setminus \tau K}|\leq \ell$.

If $\ell_V=0$ and $|(U\bdot V^{[-1]})_{\tau\la \alpha\ra \setminus \tau K}|=0$, then $U\notin \Fc(\la \alpha\ra)$ ensures $|(U\bdot V^{[-1]})_{\tau K}|=|U_{\tau \la\alpha\ra}|\geq 2$ is even ($|U_{\tau\la \alpha\ra}|$ must be even as $U$ is product-one).
In this case,  it follows in view of \eqref{weelit} that there exists a nontrivial product-one subsequence $W_0\mid U\bdot V^{[-1]}$ with $|W_0|\leq \mathsf d(K_0)+1$ and $W_0\in\Fc(K_0)$ such that $W_0$ does not contain all terms from $\tau K$.
This ensures that $U\bdot (V\bdot W_0)^{[-1]}$ contains an even positive number of terms from $\tau K$ (as $\ell_V=0$ and $U$ and $W_0$ are product-one), and we define $W\in \Fc(\tau K)$ to be any length two subsequence  of $(U\bdot (V\bdot W_0)^{[-1]})_{\tau K}$. In all other cases, we let  $W\in \Fc(\tau \la \alpha\ra)$ be a sequence of length $2|(U\bdot V^{[-1]})_{\tau\la \alpha\ra \setminus \tau K}|$ consisting of the terms from $(U\bdot V^{[-1]})_{\tau\la \alpha\ra \setminus \tau K}$ together with $|(U\bdot V^{[-1]})_{\tau\la \alpha\ra \setminus \tau K}|\leq \ell\leq |(U\bdot V^{[-1]})_{\tau K}|$ additional terms from $(U\bdot V^{[-1]})_{\tau K}$, which is possible in view of \eqref{weelit} and the case hypothesis. Since all terms from $(U\bdot V^{[-1]})_{\tau\la \alpha\ra \setminus \tau K}$ lie outside $K_0$, the first bound in \eqref{weelit} ensures there is a nontrivial product-one subsequence $W_0\mid U\bdot (V\bdot W)^{[-1]}$ with $W_0\in \Fc(K_0)$ and $|W_0|\leq \mathsf d(K_0)+1$.
In both cases, $W\in \Fc(\tau \la \alpha\ra)$ is a sequence of even length $\ell_W\leq 2\ell$ with $|W_{\tau K}|\geq \frac12\ell_W$ and  $V\bdot W\bdot W_0\mid U$, where $W_0$ is a nontrivial product-one sequence. Moreover, $|(V\bdot W)_{\tau \la \alpha\ra}|\geq 1$ and $\supp\big((U\bdot (V\bdot W)^{[-1]})_{\tau \la\alpha\ra}\big)\subset \tau K$.

 By hypothesis, $|V_{\tau K}|=|V_{K_0\setminus K}|\geq \frac12|V_{\tau\la \alpha\ra}|=\frac12\ell_V$. We also have $|W_{\tau K}|\geq \frac12\ell_W=\frac12 |W|$  and  $\supp\big((U\bdot (V\bdot W)^{[-1]})_{\tau \la\alpha\ra}\big)\subset \tau K$  by definition of $W$, while $Z=X_V+\Sigma_{\lfloor \ell_V/2\rfloor}(2V_{\tau \la\alpha\ra}^+)$ is $K^+$-periodic by hypothesis.  Thus Lemma \ref{lem-yelp} (applied with $U$ taken to be $2V_{\tau \la\alpha\ra}^+$, $V$ taken to be $2W^+$, $H$ taken to be $K^+$, $S$ taken to be $2U_{\tau \la\alpha\ra}^+$, and $X$ taken to be $X_V$) implies
\be\label{dd2}X_V+\Sigma_{\lfloor \ell_V/2\rfloor}(2V_{\tau\la \alpha\ra}^+)+\Sigma_{\ell_W/2}(2W^+)=X_V+\Sigma_{\lfloor \ell_S/2\rfloor}(2S^+)\ee
for any sequence $S\mid U_{\tau \la \alpha\ra}$ with $V_{\tau \la\alpha\ra}\bdot W\mid S$, with this set being $K^+$-periodic by hypothesis. Since $X_V+X_{U\bdot(V\bdot W_0)^{[-1]}}=X_{U\bdot W_0^{[-1]}}$, we derive from \eqref{dd2} that $$X_{U\bdot W_0^{[-1]}}+\Sigma_{\lfloor \ell_V/2\rfloor}(2V_{\tau\la \alpha\ra}^+)+\Sigma_{\ell_W/2}(2W^+)=X_{U\bdot W_0^{[-1]}}+\Sigma_{\lfloor \ell_S/2\rfloor}(2S^+)$$
for any sequence $S\mid U_{\tau \la \alpha\ra}$ with $V_{\tau\la\alpha\ra}\bdot W\mid S$, with this set being $K^+$-periodic by hypothesis. Considering the cases $S=(U\bdot W_0^{[-1]})_{\tau \la\alpha\ra}$ and $S=U_{\tau\la\alpha\ra}$, we find \be\label{dd3}X_{U\bdot W_0^{[-1]}}+\Sigma_{\lfloor (\ell_{U\bdot W_0^{[-1]}})/2\rfloor}(2(U\bdot W_0^{[-1]})_{\tau\la\alpha\ra}^+)=X_{U\bdot W_0^{[-1]}}
+\Sigma_{\lfloor \ell_U/2\rfloor}(2U^+_{\tau\la\alpha\ra}),\ee
 with this being a $K^+$-periodic set. By construction, the set $X_{W_0}$ consists of a sumset of sets $\{x_i,-x_i\}$ with $x_i\in\supp(W_0^+)\subset K^+$. Thus, since the quantity in \eqref{dd3} is $K^+$-periodic, we obtain
 \be\label{dd4}X_{U\bdot W_0^{[-1]}}+\Sigma_{\lfloor \frac12\ell_{U\bdot W_0^{[-1]}}\rfloor}(2(U\bdot W_0^{[-1]})_{\tau\la\alpha\ra}^+)=X_{U}
+\Sigma_{\lfloor \ell_U/2\rfloor}(2U^+_{\tau\la\alpha\ra}).\ee

Since $V\bdot W\mid U\bdot  W_0^{[-1]}$ and $|(V\bdot W)_{\tau \la \alpha\ra}|\geq 1$, \eqref{dd4} and Proposition \ref{prop-add/mult-DefiningCorrespondance} imply that $\pi(U)$ is a translate of the set $\pi(U\bdot W_0^{[-1]})$. However, since $W_0$ is  product-one, this forces them to be equal, in which case  $1_G \in \pi(U)=\pi(U\bdot W_0^{[-1]})$. Thus the factorization $U=W_0\bdot (U\bdot W_0^{[-1]})$ contradicts that $U$ is an atom (as $W_0$ is nontrivial) unless $U=W_0$. However, $|W_0|\leq \mathsf d(K_0)+1$ by construction, while \eqref{weelit} ensures that $|U|\geq |U\bdot V^{[-1]}|\geq \mathsf d(K_0)+1+\ell>\mathsf d(K_0)+1$, ensuring that $W_0\neq U$, which completes CASE 1.

\smallskip
\noindent
CASE 2. \,  $|(U\bdot V^{[-1]})_{\tau\la \alpha\ra \setminus \tau K}|\geq \ell$.

Since $\ell >0$, the case hypothesis ensures that $K<\la \alpha\ra=G'$ is a proper subgroup, forcing $\epsilon=0$.
Let $W\in \Fc(\tau \la \alpha\ra)$ be a sequence of length $2\ell$ consisting of $\ell>0$ terms from $(U\bdot V^{[-1]})_{\tau\la \alpha\ra \setminus \tau K}$ together with  $\ell>0$ terms from $(U\bdot V^{[-1]})_{\tau K}$, which exists in view of the case hypothesis and \eqref{weelit}.
Let $$V'=V\bdot (U\bdot V^{[-1]})_{\la \alpha\ra\setminus K}.$$
Since $Z=X_V+\Sigma_{\lfloor \ell_V/2\rfloor}(2V_{\tau\la\alpha\ra}^+)$ is $K^+$-periodic by hypothesis, it follows that
$$X:=
X_{V'}+\Sigma_{\lfloor \ell_V/2\rfloor}(2V_{\tau\la\alpha\ra}^+)$$ is also $K^+$-periodic with $|X|\geq |Z|$.

 Apply Proposition \ref{thm-partition-thm-equi}  to
   $X+\Sigma_\ell\Big(2(U\bdot V^{[-1]})^+_{\tau\la \alpha\ra}\Big)$ taking $L$ to be $K^+$ and using
  $2W^+\mid 2(U\bdot V^{[-1]})_{\tau \la \alpha\ra}^+$ (which has precisely $\ell$ terms equal to $0$ modulo $K^+$, and $\ell$ terms which are non-zero modulo $K^+$).
  Let $$H^+=\mathsf H\Big(X+\Sigma_{\ell}(2(U\bdot V^{[-1]})_{\tau \la \alpha\ra}^+)\Big).$$
  Note $K^+\leq H^+$ follows as $X$ is $K^+$-periodic.
  Since
$|X|+\ell|K|\geq |Z|+(|G'/K|-|\phi_{K}(Z)|)|K|=|G'|$,
Proposition \ref{thm-partition-thm-equi} ensures that $H/K$ is nontrivial (it it were trivial, then Proposition \ref{thm-partition-thm-equi}.1 must hold, in which case  the previous calculation combined with the bound in Proposition \ref{thm-partition-thm-equi}.1 forces $H=G'$, in which case $H/K=G'/K$ is nontrivial as $K<G'$ is a proper subgroup).
Regardless of whether Item 1 or 2 holds in Proposition \ref{thm-partition-thm-equi}, it follows that there is an $H^+$-coset that contains all but at most $\ell-1$ of the terms of $2(U\bdot V^{[-1]})_{\tau \la \alpha\ra}^+$: in case Item 1 holds, then $H^+=\Z/n\Z$, while if Item 2 holds, then this conclusion follows from Proposition \ref{thm-partition-thm-equi}.2(b)(d) with the desired coset equal to the coset $\alpha+H$ given by 2(b),
which fully contains all elements from some $A_i$ by 2(d).  Since there are at least $\ell$ terms from $K^+\leq H^+$ lying in $2(U\bdot V^{[-1]})_{\tau \la \alpha\ra}^+$ by \eqref{weelit}, this $H^+$-coset must equal the subgroup $H^+$.   Thus
Proposition \ref{thm-partition-thm-equi}  ensures that we can find a subsequence $W_1\mid (U\bdot V^{[-1]})_{\tau \la \alpha\ra}$  with $|W_1|=|W|=2\ell$, $\mathsf h\big(\phi_{K}(2W_1^+)\big)\leq \ell$,
 $|(2W_1^+)_{H^+}|\geq \ell=\frac12|W_1|$ and $\supp\big(2\Big(U_{\tau\la \alpha\ra}\bdot (V_{\tau \la\alpha\ra}\bdot W_1)^{[-1]}\Big)^+\big)\subset H^+$
 such that  $$X+\Sigma_{\ell}\big(2(U\bdot V^{[-1]})_{\tau \la \alpha\ra}^+\big)=X+\Sigma_{\ell}(2W_1^+)=X_{V'}+\Sigma_{\lfloor \ell_V/2\rfloor}(2V_{\tau\la\alpha\ra}^+)+\Sigma_\ell(2W_1^+)$$
 is $H^+$-periodic.
  Since $\mathsf h\big(\phi_{K}(2W_1^+)\big)\leq \ell$, it follows that at most $\ell$ terms of $W_1$ lie in $K_0$, while no terms in $V'\bdot V^{[-1]}$ lie in $K_0$ by its definition. Thus \eqref{weelit} ensures that
$U\bdot (V'\bdot W_1)^{[-1]}$ contains at least $\mathsf d(K_0)+1$ terms from $K_0$, meaning there exists a nontrivial product-one subsequence $W_0\mid U\bdot (V'\bdot W_1)^{[-1]}$ with $|W_0|\leq \mathsf d(K_0)+1$ and $W_0\in \Fc(K_0)$.
 By hypothesis,
  $|V_{\tau K}|=|V_{K_0\setminus K}|\geq \frac12|V_{\tau \la\alpha\ra}|$, meaning at least half the terms of $2V_{\tau \la\alpha\ra}^+$ lie in $(\tau K)^+=K^+\leq H^+$.
  As a result, we can apply Lemma \ref{lem-yelp} (taking $U$ to be $2V_{\tau \la\alpha\ra}^+$, taking $V$ to be $2W_1^+$, taking $H$ to be $H^+$, taking $S$ to be $2U_{\tau\la\alpha\ra}^+$, taking $X$ to be $X_{V'}$, and taking $S'$ to be
  $2U^+_{\tau \la\alpha\ra}$ as well as $2(U\bdot W_0^{[-1]})_{\tau \la \alpha\ra}^+$)
  to conclude
   \be\label{dekat}X_{V'}+\Sigma_{\lfloor \frac12\ell_{U\bdot W_0^{[-1]}}\rfloor}\big(2(U\bdot W_0^{[-1]})^+_{\tau \la\alpha\ra}\big)= X_{V'}+\Sigma_{\lfloor\ell_U/2\rfloor}(2U^+_{\tau \la \alpha\ra})=X+\Sigma_\ell(2W_1^+)\ee is $H^+$-periodic.
   By definition of $V'$, all $g\in \supp(U\bdot (V')^{[-1]})\cap \la \alpha\ra$ lie in $K$, ensuring $g^+\in K^+\leq H^+$.
   As a result, since the quantity in \eqref{dekat} is $H^+$-periodic, it follows that
   $$X_{U\bdot W_0^{[-1]}}+\Sigma_{\lfloor \frac12\ell_{U\bdot W_0^{[-1]}}\rfloor}\big(2(U\bdot W_0^{[-1]})_{\tau \la\alpha\ra}\big)= X_{U}+\Sigma_{\lfloor\ell_U/2\rfloor}(2U^+_{\tau \la \alpha\ra}),$$
    whence  $|\pi(U\bdot W_0^{[-1]})|=|\pi(U)|$ by Proposition \ref{prop-add/mult-DefiningCorrespondance}, forcing $1_G \in \pi(U)=\pi(U\bdot W_0^{[-1]})$ since both $U$ and $W_0$ are product-one.
    It follows that $U=W_0\bdot (U\bdot W_0^{[-1]})$ is a factorization of $U$ into product-one sequences, with $W_0$ nontrivial by definition.
     Since $U$ is an atom, this forces $W_0=U$. However, $|W_0|\leq \mathsf d(K_0)+1$ by definition, while $|U|\geq |U\bdot V^{[-1]}|\geq \mathsf d(K_0)+1+\ell>\mathsf d(K_0)+1\geq |W_0|$ by \eqref{weelit}, whence $W_0=U$ is impossible, completing the proof.
\end{proof}

\begin{lemma}\label{lem-greedy}
Let $G$ be an abelian group, let $A_1,\ldots,A_\ell\subset G$ be cardinality two subsets, let $X=\Sum{i=1}{\ell}A_i$, and let  $H=\mathsf H(X)$. Let $I_H\subset [1,\ell]$ be all those $i\in [1,\ell]$ with $|\phi_H(A_i)|=1$. Then there exists a subset $J\subset [1,\ell]$ with $|\Summ{i\in J}A_i|=|X|$, $J\setminus I_H=[1,\ell]\setminus I_H$, $|J\cap I_H|\leq |H|-1$, $|J\setminus I_H|\leq |\phi_H(X)|-1$ and $|J|\leq |H|+|\phi_H(X)|-2$.
\end{lemma}

\begin{proof}
Note $\phi_H(A_i)$ has cardinality one or two depending on whether $i\in I_H$ or $i\in [1,\ell]\setminus I_H$. Thus Kneser's Theorem implies $|[1,\ell]\setminus I_H|\leq |\phi_H(X)|-1$, while \cite[Proposition 2.2]{Gr05c} applied to $\Summ{i\in I_H}A_i$ ensures there is a subset $I'_H\subset I_H$ with $|\Summ{i\in I'_H}A_i|=|\Summ{i\in I_H}A_i|$ and $|I'_H|\leq |\Summ{i\in I_H}A_i|-1\leq |H|-1$. Setting $J=I'_H\cup ([1,\ell]\setminus I_H)$ now yields the desired index set.
\end{proof}

\begin{proposition}\label{prop-subgroupbound}
Let $G$ be a dihedral group of order $2n$ where $n\geq 3$ is odd, say $G=\la \alpha,\,\tau:\; \alpha^n=\tau^2=1_G,\; \tau \alpha=\alpha^{-1}\tau\ra$. If $H_z=\la H,\tau\alpha^z\ra<G$ is a proper subgroup with  $H\leq \la \alpha\ra$, and  $U\in \mathcal A(G)$, then  $|U_{H_z}|\leq n+|H|-1$, with equality only possible if $H$ is trivial.
\end{proposition}

\begin{proof}By replacing the generator $\tau$ by $\tau\alpha^z$, we can w.l.o.g. assume $z=0$.
Assume by contradiction $U\in \mathcal A(G)$ with \be\label{Ubig}|U_{H_0}|=|U_H|+|U_{\tau H}|\geq n+|H|-1+\epsilon,\ee where $\epsilon=1$ if $H$ is trivial and otherwise $\epsilon=0$.
Since $H_0<G$ is proper, it follows that $H<G'$ is proper. By \eqref{Ubig}, we have $|U|\geq n+1\geq 4$, ensuring that the atom $U\in \mathcal A(G)$ does not consist of a single term equal to $1_G$, forcing $1_G \notin\supp(U)$.
If $U\in \Fc(\la \alpha\ra)$, then $\la \supp(U)\ra$ is abelian and $U_{H_0}=U_{H}$. Thus
 $|U_{H_0}|=|U_{H}|\leq \mathsf D(H)=|H|$, again contradicting  \eqref{Ubig}.
 Therefore  $|U_{\tau\la \alpha\ra}|>0$.

Let $K\leq\la \alpha\ra$ be arbitrary. Lemma \ref{lem-HinAlpha} implies
\be\label{bigoff}|U_{K}|\leq 2|K|-2\quad\mbox{ for every $K\leq \la\alpha\ra$}, \und \ell:=|U_{\tau H}|\geq n-|H|+1+\epsilon\geq |G'/H|+1\geq 4,\ee with the latter inequality following from the former (taking $K=H$) combined with \eqref{Ubig} (and recalling that $H<G'$ is proper).
Since  $n\geq 3$ is odd and $\frac12|U_{\tau\la\alpha\ra}|\geq \frac12 \ell\geq 2$, Lemma \ref{lem-triv-multbound} gives
\be\label{trivupper}\mathsf h(2U^+_{\tau\la \alpha\ra})=\mathsf h(U_{\tau\la \alpha\ra})\leq \frac12 |U_{\tau\la \alpha\ra}|\leq n,\ee
 with the latter inequality in view of $|U_{\tau\la \alpha\ra}|\leq |U|\leq \mathsf D(G)=2n$.
If $H$ is trivial, then $H=\{1_G \}$ and $H_0=\la \tau\ra=\{1_G ,\tau\}$, so $U_{H}$ is the trivial sequence (as $1_G \notin \supp(U)$) and $\supp(U_{H_0})=\{\tau\}$. In such case, \eqref{Ubig} implies  $\vp_\tau(U)=|U_{H_0}|\geq n+1$, contradicting \eqref{trivupper}. Therefore we may now assume $H$ is nontrivial, and thus $\epsilon=0$.

 For a subsequence $T\mid U$, let $\ell_T=|T_{\tau\la \alpha\ra}|$ and $X_T=\{y_1,-y_1\}+\ldots+\{y_t,-y_t\}$, where $T_{\la \alpha\ra}^+=y_1\bdot\ldots\bdot y_t$.
Let $X=X_{U_H}$,
$\ell=\ell_{U_{H_0}}$ and $$L^+=\mathsf H(X)\leq H^+.$$
Since $\ell>0$, Proposition \ref{prop-add/mult-DefiningCorrespondance}.2 implies that $\pi(U_{H_0})^+$ is a translate of the set
$X+\Sigma_{\lfloor\ell/2\rfloor}(2U_{\tau H}^+)$.
Let $$m_X=\mathsf h\big(\phi_{L}(2U_{\tau H}^+)\big)=\mathsf h\big(\phi_{L}(U_{\tau H}^+)\big)$$
and let
$U'_{\tau H}\mid U_{\tau H}$ be a maximal length subsequence with $\mathsf h\big(\phi_{L}(2(U'_{\tau H})^+)\big)\leq \lceil \ell/2\rceil$. Note $\phi_{L}(X)$ is aperiodic as $L^+$ is the stabilizer of $X$.  As a result,   applying  Kneser's Theorem to the aperiodic sumset $\phi_{L}(X)$, which is a sumset of $|U_{H\setminus L}|$ cardinality two sets and $|U_L|$ cardinality one sets,  yields
 \be\label{kt-Xbound} |\phi_{L}(X)|\geq |U_{H\setminus L}|+1.\ee
 Combining the above bound with \eqref{bigoff}, we obtain
 \be\label{kt-Ybound} |U_{H}|=|U_{L}|+|U_{H\setminus L}|\leq 2|L|+|\phi_{L}(X)|-3.\ee
In view of \eqref{Ubig}, \eqref{kt-Ybound},  $|\phi_{L}(X)|\leq |H/L|$ (as $X\subset H^+$) and $H<G'$ proper, we have  \be\label{ellway}\ell=|U_{\tau H}|=|U_{H_0}|-|U_{H}|\geq n+|H|+2-2|L|-|\phi_{L}(X)|\geq
n-|H|+1\geq 2|H|+1.\ee
Lemma \ref{lem-easycond} and  the following claim will complete the proof by contradicting that $U$ is an atom.

\subsection*{Claim A} There is a subgroup $K_y=\la K,\tau\alpha^y\ra\leq H_0$ with  $K\leq \la \alpha\ra$  and  a subsequence $V\mid U_{H_0}$ such that   $|V_{K_y\setminus  K}|\geq \frac12 |V_{\tau H}|$, $Z$ is $K^+$-periodic,  $|(U\bdot V^{[-1]})_{K_y}|\geq |K|+1+|G'/K|-|\phi_{K}(Z)|$ and $|(U\bdot V^{[-1]})_{K_y\setminus K}|\geq |G'/K|-|\phi_{K}(Z)|$, where $Z=X_V+\Sigma_{\lfloor \ell_V/2\rfloor}(2V^+_{\tau H})$.

\medskip
To prove the claim, we distinguish three cases.

\medskip
\noindent
CASE 1. \,    $m_X\geq |U_{\tau H}|-|H/L|+2$.

Let $y\in \supp(U^+_{\tau H})$ be an element which modulo $L^+$ has  multiplicity $m_X$ in $\phi_{L}(U_{\tau H}^+)$, and set $L_y=\la L,\tau\alpha^y\ra$.
Lemma \ref{lem-greedy}  applied to the sumset $X$ finds a subsequence $V\mid U_{H}$ with
$V_{H\setminus L}=U_{H\setminus L}$, $|V_{L}|\leq |L|-1$,
$|V|\leq |L|+|\phi_{L}(X)|-2$,  and $X_V=X_{U_H}=X$.
Note \begin{align*}|(U\bdot V^{[-1]})_{L_y}|&\geq m_X+|U_{L}|-|V_{L}|\geq |U_{H_0}|-|U_{H\setminus L}|-|H/L|+2-|V_{L}|\\
&\geq |U_{H_0}|-|\phi_{L}(X)|-|L|-|H/L|+4\geq |U_{H_0}|-|\phi_{L}(X)|-|H|+3\\&\geq n-|\phi_{L}(X)|+2\geq |L|+|G'/L|-|\phi_{L}(X)|+1,\end{align*}
with the first inequality as $V\mid U_H$, with the second inequality by case hypothesis,  the third in view of $|V_{L}|\leq |L|-1$ and \eqref{kt-Xbound},  and the fifth in view of \eqref{Ubig}.
The case hypothesis ensures there are at most $|H/L|-2$ terms of $U_{\tau H}$ lying outside $L_y=\la L,\tau \alpha^y\ra$. Consequently, if  $|(U\bdot V^{[-1]})_{L_y\setminus L}|=|U_{L_y\setminus L}|\leq |G'/L|-1-|\phi_{L}(X)|$,
then $|U_{\tau H}|\leq (|G'/L|-1-|\phi_{L}(X)|)+(|H/L|-2)$, which combined with \eqref{kt-Ybound}  yields $|U_{H_0}|=|U_{H}|+|U_{\tau H}|\leq 2|L|+\frac{n+|H|}{|L|} -6\leq n+|H|-4,$ contrary to \eqref{Ubig}.  Therefore     $|(U\bdot V^{[-1]})_{L_y\setminus L}|=|U_{L_y\setminus L}|\geq |G'/L|-|\phi_{L}(X)|$, meaning Claim A holds with $K_y$ taken to be $L_y$ and  $Z$ taken to be $X$, contradicting that $U$ is an atom.

\medskip
\noindent
CASE 2. \, $|X+\Sigma_{\lceil \ell/2\rceil}(2U_{\tau H}^+)|\geq \min\{|H|,\,(|U'_{\tau H}|-\lceil \ell/2\rceil)|L|+|X|\}$ and $m_X\leq |U_{\tau H}|-|H/L|+1$.

Apply Proposition \ref{thm-partition-thm-equi} to $X+\Sigma_{\lceil\ell/2\rceil}(2U^+_{\tau H})$ taking $L$ to be $L^+$ and using $2(U'_{\tau H})^+\mid 2U^+_{\tau H}$.
First suppose that   $|X+\Sigma_{\lceil \ell/2\rceil}(2U_{\tau H}^+)|\geq (|U'_{\tau H}|-\lceil\ell/2\rceil)|L|+|X|$, so that Proposition \ref{thm-partition-thm-equi}.1 holds.
 Consequently, if $m_X=\mathsf h(2\phi_{L}(U_{\tau H}^+))\leq \lceil\ell/2\rceil$, then $U'_{\tau H}=U_{\tau H}$, in which case
 Proposition \ref{prop-add/mult-DefiningCorrespondance} implies that
 $|\pi(U_{H_0})|\geq
 (|U_{\tau H}|-\lceil \frac{\ell}{2}\rceil)|L|+|X|= \lfloor\ell/2\rfloor \,|L|+|X|\geq \lfloor\ell/2\rfloor\geq  |H|$, with the final inequality by \eqref{ellway}.
 On the other hand, if $m_X=\mathsf h(2\phi_{L}(U_{\tau H}^+))>\lceil\ell/2\rceil$, then there is a unique term with multiplicity  greater than $\lceil\ell/2\rceil$ in $\phi_{L}(2U_{\tau H}^+)$, and
 Proposition \ref{prop-add/mult-DefiningCorrespondance} instead implies   $|\pi(U_{H_0})|\geq
 (|U'_{\tau H}|-\lceil \frac{\ell}{2}\rceil)|L|+|X|\geq (|U'_{\tau H}|-\lceil \frac{\ell}{2}\rceil+1)|L|=(|U_{\tau H}|-m_X+1)|L|\geq|H|$, with the final inequality holding by the upper bound $m_X\leq |U_{\tau H}|-|H/L|+1$  in the hypothesis of CASE 2.
  In either case, we conclude that  $|\pi(U_{H_0})|=|H|$ (note $|\pi(U_{H_0})|\leq |H|$ holds trivially as $\supp(U_{H_0}^+)\subseteq H_0^+= H$, ensuring the inequality in Proposition \ref{thm-partition-thm-equi}.1 must hold with equality), and so Proposition  \ref{thm-partition-thm-equi} ensures  there is a setpartition of cardinality two sets realizing $X+\Sigma_{\lceil \ell/2\rceil}(2U_{\tau H}^+)$ as a sumset.
  If instead $|X+\Sigma_{\lceil \ell/2\rceil}(2U_{\tau H}^+)|< (|U'_{\tau H}|-\lceil \ell/2\rceil)|L|+|X|$ and $|\pi(U_{H_0})|=|H|$, then Proposition \ref{thm-partition-thm-equi}.2 yields these same conclusions.
  Thus in both cases,  Lemma \ref{lem-greedy} applied to the sumset $X=X_{U_H}$ and \cite[Proposition 2.2]{Gr05c} applied to $\phi_{L}(2U^+_{\tau H})$ (via the setpartition realizing $X+\Sigma_{\lceil \ell/2\rceil}(2U_{\tau H}^+)$ considered modulo $L^+$) give us a subsequence $V\mid U_{H_0}$ with $|V|=|V_{H}|+|V_{\tau H}|\leq (|L|+|\phi_{L}(X)|-2)+\max\{2,\,2(|H/L|-|\phi_{L}(X)|)\}\leq  2|H|-2$   such that $|\pi(V)|=|\pi(U_{H_0})|=|H|$.
  Indeed, \cite[Proposition 2.2]{Gr05c} (which is simply the greedy algorithm) ensures we need keep at most  one cardinality two set for each element of $\phi_{L}\big(X+\Sigma_{\lceil \ell/2\rceil}(2U_{\tau H}^+)\big)$ in excess of the original $|\phi_{L}(X)|$ elements  from $\phi_{L}(X)$, and thus $|V_{\tau H}|\leq 2\Big(|\phi_{L}\big(X+\Sigma_{\lceil \ell/2\rceil}(2U_{\tau H}^+)\big)|-|\phi_{L}(X)|\Big)\leq 2(|H/L|-|\phi_{L}(X)|)$. However,  in order to ensure $V_{\tau \la \alpha\ra}$ is nonempty when $|X|=|H|$ (so that we can apply Proposition \ref{prop-add/mult-DefiningCorrespondance} to conclude the resulting sumset has the same cardinality as $|\pi(V)|$), we always include at least $2\leq \ell$ terms from some  cardinality two set, resulting in  $|V_{\tau H}|\leq \max\{2,\, 2(|H/L|-|\phi_{L}(X)|)\}$.
  But now \eqref{Ubig} yields $|(U\bdot V^{[-1]})_{H_0}|\geq n-|H|+1\geq |H|+|G'/H|$, with the latter inequality following as $H$ is a proper, nontrivial subgroup of the odd order group  $G'$ (forcing $|G'|\geq 9$), while \eqref{ellway} then implies $|(U\bdot V^{[-1]})_{\tau H}|\geq \ell-2(|H/L|-|\phi_{L}(X)|+1)
  \geq n+|H|+1-2|L|-2|H/L|
  \geq n-|H|-1\geq |G'/H|-1$.  So Claim A  holds with $K_y=H_0$, contradicting that $U$ is an atom.

\medskip
\noindent
CASE 3. \,  $|X+\Sigma_{\lceil \ell/2\rceil}(2U_{\tau H}^+)|<\min\{|H|,\,(|U'_{\tau H}|-\lceil \ell/2\rceil)|L|+|X|\}$.

In view of the case hypothesis, we can
apply Proposition \ref{thm-partition-thm-equi}.2 to $X+\Sigma_{\lceil\ell/2\rceil}(2U^+_{\tau H})$ taking $L$ to be $L^+$ and using $2(U'_{\tau H})^+\mid 2U^+_{\tau H}$.
  Let $y+K^+$ be the resulting coset with $K/L<H/L$ proper and nontrivial (so $y+K^+$ is the coset $\alpha+H$ from Proposition \ref{thm-partition-thm-equi}.2; note $K/L$ is proper in view of the case hypothesis $|X+\Sigma_{\lceil \ell/2\rceil}(2U_{\tau H}^+)|<|H|$), and set $K_y=\la K,\tau\alpha^y\ra\leq H_0$. In particular, $K$ is nontrivial and $|G'|\geq 3^3=27$.
  In this case, Lemma \ref{lem-greedy} along with \cite[Proposition 2.2]{Gr05c} applied to the sumset given by Proposition \ref{thm-partition-thm-equi}.2(d) yields a subsequence $V\mid U_{H_0}$ with \be\label{dino}|V|=|V_{H}|+|V_{\tau H}|\leq (|L|+|\phi_{L}(X)|-2)+2(|K/L|-1)= |L|+2|K/L|-4+|\phi_{L}(X)|\ee and  $\pi(V)^+$ being $K^+$-periodic.
  Moreover, as the subsequence $V_{\tau H}$ is that partitioned by a sub-setpartition of the one given by Proposition \ref{thm-partition-thm-equi} (cf. \cite[Proposition 2.2]{Gr05c}),  we have  $|V_{K_y\setminus K}|\geq \frac12 |V_{\tau H}|$, for each cardinality two subset must contain at least one term from $y+K^+=(K_y\setminus K)^+$ by Proposition \ref{thm-partition-thm-equi}.2(b).  Proposition \ref{thm-partition-thm-equi}.2(c) ensures that at most $|H/K|-2$ of the terms of $U_{\tau H}$ are not in $K_y$ (as $(K_y\setminus K)^+=y+K^+$). Thus \begin{align*}|(U\bdot V^{[-1]})_{K_y\setminus K}|\geq |U_{\tau H}|-(|H/K|-2)-2(|K/L|-1)\geq n-|H|+5-2|K/L|-|H/K|\\\geq n-|H|+5-2|K|-|H/K|\geq \frac23 |G'|+5-2|K|-\frac13|G'/K|\geq |G'/K|-1,\end{align*} with the second inequality above in view of \eqref{ellway}.
 By construction (cf. Lemma \ref{lem-greedy}), $V_{H}$ contains all terms of $U_{H}$ lying outside $L$, so all terms in
 $U_{H}\bdot V_{H}^{[-1]}$ lie in $L\leq K\leq K_y$. Combining this with \eqref{dino}, \eqref{Ubig} and the already observed fact that there are  at most $|H/K|-2$  terms from $U_{\tau H}$ lying outside $K_y$  implies \begin{align*}|(U\bdot V^{[-1]})_{K_y}|&\geq |U_{H_0}|-|H/K|+2-|V|\geq |U_{H_0}|-|H/K|-|L|-2|K/L|+6-|H/L| \\  &\geq n+|H|-|H/K|-|L|-2|K/L|-|H/L|+5 \\&\geq n+|H|-|H/K|-2|K|-|H|+4 \geq |G'|-\frac13|G'/K|-2|K|+4\\&\geq |K|+|G'/K|,\end{align*} where the final inequality follows as $K$ is nontrivial with $K<H<G'$.
Therefore  Claim A holds with $K_y$ as defined above, contradicting that $U$ is an atom. This completes CASE 3 and the proof.
\end{proof}

\smallskip
\begin{proof}[Proof of Theorem \ref{4.1}]
Let $G$ be a dihedral group of order $2n$ where $n \ge 3$ is odd, say $G=\la \alpha,\,\tau:\; \alpha^n=\tau^2=1_G,\; \tau \alpha=\alpha^{-1}\tau\ra$.

To see $\omega(G)\geq 2n$, let $U=((\tau \alpha)\bdot \tau)^{[n]}\in \mathcal A(G)$. Then $U\bdot U=(\tau\bdot\tau)^{[n]}\bdot (\tau \alpha\bdot \tau \alpha)^{[n]}$ is a factorization into $2n$ length two atoms. Suppose $U$ divides (in $\mathcal B(G)$) a sub-product $S\in\mathcal B(G)$ of these length two atoms. Then $S$ must have an even number of copies of both $\tau \alpha$ and $\tau$, ensuring that $S\bdot U^{[-1]}\in \mathcal B(G)$ contains an odd number of both $\tau \alpha$ and $\tau$ (as $n$ is odd). However, it is readily seen (cf. Lemma \ref{lem-triv-multbound}) that $\tau^{[2]}$, $(\tau\alpha)^{[2]}$ and $U$ are the only atoms with support contained in $\{\tau \alpha,\tau\}$. In particular, $U$ is the only atom with support contained in $\{\tau \alpha,\tau\}$ having an odd number of copies of
 $\tau$ and $\tau \alpha$. Thus $S\bdot U^{[-1]}=U$, ensuring $S$ must be the sub-product of all $2n$ length two atoms, which shows $\omega(G)\geq 2n$.

It remains to show $\omega(G)\leq 2n$. To this end, suppose $U,U_1,\ldots,U_w\in \mathcal A(G)$ are atoms with $U\mid_{\mathcal B(G)} U_1\bdot\ldots\bdot U_w$, i.e., $U\mid U_1\bdot\ldots\bdot U_w$ and $(U_1\bdot\ldots\bdot U_w)\bdot U^{[-1]}\in \mathcal B(G)$. We need to show there exists a subset $J\subset [1,w]$ with $U\mid_{\mathcal B(G)} \prod^{\bullet}_{i\in J}U_i$ and $|J|\leq 2n$.

Since $U\mid U_1\bdot\ldots\bdot U_w$, let $I_\varnothing\subset [1,w]$ be a minimal cardinality subset with $U\mid \prod_{\in I_\varnothing}^{\bullet}U_i$. In view of the minimality of $|I_\varnothing|$, we have $|I_\varnothing|\leq |U|\leq \mathsf D(G)=2n$.
Thus, if $V_\varnothing:=U^{[-1]}\bdot \prod^\bullet_{i\in I_\varnothing}U_i$ is product-one, then the proof is complete taking $J$ to be $I_\varnothing$. Therefore we may assume $$1_G \notin \pi(V_\varnothing).$$
Since both  $U$ and  $\prod_{i\in I_\varnothing}^\bullet U_i$  are product-one sequences, it follows that $2\ell_\varnothing:=|(V_\varnothing)_{\tau \la\alpha\ra}|$ must be even with  $\pi(V_\varnothing)\subset G'$.
Since $U\mid_{\mathcal B(G)} \prod^\bullet_{i\in [1,w]}U_i$, we have $1_G \in \pi(U^{[-1]}\bdot \prod^\bullet_{i\in [1,w]}U_i)=\pi(V_\varnothing\bdot \prod^\bullet_{i\in [1,w]\setminus I_\varnothing}U_i) $. In consequence, if $\supp(V_\varnothing\bdot \prod^\bullet_{i\in [1,w]\setminus I_\varnothing}U_i)\subseteq \la \alpha\ra$, then $1_G \in \pi(V_\varnothing)$ (as $\la \alpha\ra$ is abelian), contrary to assumption. As a result, if $\ell_\varnothing>0$, then set $I_\emptyset=I_\varnothing$, and otherwise set $I_\emptyset=I_\varnothing\cup \{i_\emptyset\}$, where $i_\emptyset\in [1,w]\setminus I_\varnothing$ is an index with $\supp(U_{i_\emptyset})\cap \tau \la\alpha\ra$ nonempty.

Let $I\subset [1,w]$ be an arbitrary subset with $I_\emptyset\subset I$. Set
 $V_\emptyset=U^{[-1]}\bdot \prod_{\in I_\emptyset}^{\bullet}U_i\in \mathcal F(G)$ and
 $V=U^{[-1]}\bdot \prod_{\in I}^{\bullet}U_i\in \mathcal F(G)$.
Since   $U$, \ $\prod_{i\in I}^\bullet U_i$ and  $\prod_{i\in I_0}^\bullet U_i$ are product-one sequences, $2\ell:=|V_{\tau \la\alpha\ra}|$ and $2\ell_\emptyset=|(V_\emptyset)_{\tau\la\alpha\ra}|$ must both be even with  $\pi(V)\subset G'$.
 Let $$X=\{x_1,-x_1\}+\ldots+\{x_{s},-x_{s}\},$$ where $V^+_{\la\alpha\ra}=x_1\bdot\ldots\bdot x_{s}.$ By construction, $\ell\geq \ell_\emptyset>0$. Thus Proposition \ref{prop-add/mult-DefiningCorrespondance} ensures that $$\pi(V)^+=X+\Sigma_{\ell}(2V_{\tau \la\alpha\ra}^+)-\sigma(V_{\tau \la\alpha\ra}^+).$$
Let $H^+=\mathsf H(\pi(V)^+)$ and $L^+=\mathsf H(X)$ (set $L^+=X=\{0\}$ if $s=0$). Note $L\leq H$.
If $\mathsf h(V_{\tau \la\alpha\ra})\leq \ell$, let $\ell'=\ell$. Otherwise, let $\ell'=2\ell-\mathsf h(V_{\tau \la\alpha\ra})$. Note $\ell'\geq 0$ is the maximal integer for which there is a decomposition $V_{\tau \la\alpha\ra}^+=T_1\bdot\ldots\bdot T_{\ell}$ with
$|T_i|=2$  for all $i\in [1,\ell]$ and $|\supp(T_i)|=2$ for all $i\leq \ell'$. Likewise, if $\mathsf h(\phi_{L}(V_{\tau \la\alpha\ra}^+))\leq \ell$, let $\ell_{L}=\ell$. Otherwise, let $\ell_{L}=2\ell-\mathsf h(\phi_{L}(V_{\tau\la\alpha\ra}^+))$.
Since $U\mid \prod_{\in I_\varnothing}^{\bullet}U_i$, we have decompositions $U_i=W_i\bdot W^U_i$, for $i\in I_\varnothing$, with $\prod_{i\in I_\varnothing}^\bullet W_i=V_\varnothing$ and $\prod_{i\in I_\varnothing}^\bullet W^U_i=U$. Note $|W^U_i|>0$ for all $i\in I_\varnothing$ in view of the minimality of $|I_\varnothing|$. Let $W_i^U$ be the trivial sequence with  $W_i=U_i$ for $i\in I\setminus I_\varnothing$.
Partition $I=I_\varnothing^{2}\cup I^\alpha\cup I^{\alpha\tau}\cup I^\tau$,  where $I^{2}_\varnothing\subset I_\varnothing$ consist of all $i\in I_\varnothing$ with $|W_i^U|\geq 2$, where $I^\alpha\subset I$ consists of all $i\in I\setminus I_\varnothing^2$ with $\supp(W_i)\subset \la \alpha\ra$, where $I^\tau\subset I$ consists of all $i\in I\setminus I_\varnothing^2$ with $\supp(W_i)\subset \tau\la \alpha\ra$, and where $I^{\alpha\tau}\subset I$ consists of all $i\in I\setminus I_\varnothing^2$ with $\supp(W_i)\cap  \la \alpha\ra$ and $\supp(W_i)\cap  \tau \la \alpha\ra$ both nonempty. Let $I_\varnothing^1=I_\varnothing\setminus I_\varnothing^2$. Let $I_\varnothing=I_\varnothing^{2}\cup I_\varnothing^\alpha\cup I_\varnothing^{\alpha\tau}\cup I_\varnothing^\tau$ and $I_\emptyset=I_\varnothing^{2}\cup I_\emptyset^\alpha\cup I_\emptyset^{\alpha\tau}\cup I_\emptyset^\tau$ be the analogous partitions for $I_\varnothing$ and $I_\emptyset$.
To simplify notation, we have suppressed the dependency on $I$ of $V$, $\ell$, $X$, $s$, $H$,  $L$, and $\ell'$  from the notation. In the case when $I=I_\emptyset$, we denote these parameters by $V_\emptyset$, $\ell_\emptyset$, $X_\emptyset$, $s_\emptyset$, $H_\emptyset$, $L_\emptyset$ and $\ell'_\emptyset$.

Since $U\mid \prod_{i\in I_\varnothing}^\bullet U_i$ and $1_G \notin \pi(V_\varnothing)$, it follows that $\supp(\prod_{i\in I_\varnothing}^\bullet U_i)\cap \tau \la\alpha\ra$ is nonempty. In particular, if $I_\varnothing\neq I_\emptyset$, then there must be some $U_{i}$ with $\supp(U_{i})\cap \tau\la\alpha\ra$ nonempty and $i\in I_\varnothing$. Moreover, as $U_i$ is product-one, it must then have an even number of terms from $\tau \la \alpha\ra$, all of which are not contained in $V_\varnothing$ as $I_\varnothing\neq I_\emptyset$, whence $i\in I_\varnothing^2$, ensuring $I_\varnothing^2$ is nonempty.
In summary, $I_\varnothing^2\neq \varnothing$  when $I_\varnothing\neq I_\emptyset$.
In particular, if $I_\varnothing \neq I_\emptyset$, then $|I_\emptyset|=|I_\varnothing|+1\leq (|U|-|I_\varnothing^2|)+1\leq |U|\leq 2n$. Thus we can also assume \be\label{V0not1}1_G \notin \pi(V_\emptyset),\ee for otherwise the proof is complete taking $J=I_\emptyset$.

\smallskip

\textbf{Remark:} The proof contains three interludes referred to as Claims A, B and C. The proofs of these claims contain some of the key points in our argument. However, to better see the overall structure of the proof and the role these claims play in it, and also to reduce the  interruption in flow of the main line of arguments, we merely introduce their respective statements during the course of the proof and defer their proofs until the very end. We begin doing so with Claim A.

\subsection*{Claim A} $\ell'_\varnothing\geq \min\{\ell_\varnothing,\,|I_\varnothing^\tau|\}\geq \frac12 |I_\varnothing^\tau|$.

\bigskip

Continuing in the main line of argument, we say the set $I\subset [1,w]$ (containing $I_\emptyset$) is \emph{ample} if
the following hold:
\begin{itemize}
\item[A1.] $|\pi(V)|\geq \left\lfloor\frac12|I_\varnothing^{1}|\right\rfloor+1+|I\setminus I_\varnothing|$.
\item[A2.] $|X|\geq |X_\varnothing|+|I^\alpha\setminus I_\varnothing^\alpha|$.
\end{itemize}

We will say the  set $I$ (containing $I_\emptyset$) is \emph{constrained} or (more specifically) \emph{$H_x$-constrained} if there exists a subgroup $H_x=\la H,\tau \alpha^x\ra\leq G$, as well as a decomposition $V_{\tau \la\alpha\ra}^+=T_1\bdot\ldots\bdot T_{\ell}$ with $|T_i|=2$  for all $i\in [1,\ell]$ such that, letting $A_i=\supp(2T_i)$ for $i\in [1,\ell]$, the following hold:
\begin{itemize}
\item[C1.]  $X+\Sum{i=1}{\ell}A_i$ is $H^+$-periodic.
\item[C2.] $A_i\cap (2x+H^+)\neq \emptyset$ for all $i\in [1,\ell]$.
\item[C3.] There is a $j\in [1,\ell]$ with $|X+\underset{i\neq j}{\Sum{i=1}{\ell}}A_i|=|X+\Sum{i=1}{\ell}A_i|$.
\item[C4.] $|X|\geq |X_\varnothing|+|I^\alpha\setminus I_\varnothing^\alpha|$.
\end{itemize}

Conditions C1 and C2 allow us to apply Proposition \ref{prop-regulate}.1 to conclude $X+\Sum{i=1}{\ell}A_i=X+\Sigma_{\ell}(2V^+_{\tau \la\alpha\ra})=\pi(V)^++\sigma(V^+_{\tau\la\alpha\ra})$, in which case
$H^+=\mathsf H(X+\Sum{i=1}{\ell}A_i)=\mathsf H(\pi(V)^+)$ by definition of $H$.
Note that $\phi_{H}(X+\Sum{i=1}{\ell}A_i)=\phi_{H}(\{x_1,-x_1\})+\ldots+\phi_{H}(\{x_s,-x_s\})+\Sum{i=1}{\ell}\phi_{H}(A_i)$ is a sumset of cardinality at most two sets. Since $n$ is odd, the set $\{x_i,-x_i\}$ considered modulo $H^+$ has cardinality two precisely when $x_i\notin H^+$, while the set $\phi_{H}(A_i)$ has cardinality two (in view of C2) precisely when $T_i$ consists of one term from $x+H^+$ with its other term lying outside $x+H^+$. As a result, $|V_{G\setminus H_x}|$ equals the number of cardinality two summands in the sumset $\phi_{H}(\{x_1,-x_1\})+\ldots+\phi_{H}(\{x_s,-x_s\})+\Sum{i=1}{\ell}\phi_{H}(A_i)$, in which case Kneser's Theorem implies $|X+\Sum{i=1}{\ell}A_i|\geq (|V_{G\setminus H_x}|+1)|H|$.
 In summary, Conditions C1 and C2 imply
\be\label{extra-C} X+\Sum{i=1}{\ell}A_i=X+\Sigma_{\ell}(2V^+_{\tau \la\alpha\ra})=\pi(V)^++\sigma(V^+_{\tau\la\alpha\ra})\quad\und\quad|X+\Sum{i=1}{\ell}A_i|\geq (|V_{G\setminus H_x}|+1)|H|.\ee
Kneser's Theorem (applied modulo $H^+$) ensures that $A_j\subset 2x+H^+$ for any $j\in [1,\ell]$ satisfying C3.  As we trivially have $|X+\Sum{i=1}{\ell}A_i|\leq n$, with equality only possible when $H^+=\Z/n\Z$, \eqref{extra-C} implies \be\label{extra-C-index}|V_{G\setminus  H_x}|\leq |G'/H|-2+\epsilon,\ee where $\epsilon=0$ if $H^+<\Z/n\Z$ is proper, and $\epsilon=1$ if $H^+=\Z/n\Z$.

\medskip

\textbf{Remark:} In order to help the reader digest the overall scope of the proof, we provide some informal remarks here on its structure. These remarks are not essential to the proof, but may help in explaining the general strategy. As we saw in \eqref{V0not1}, having $1_G \in \pi(V_\emptyset)$ would complete the proof. Indeed, having $1_G \in \pi(V)$ when $|I|\leq 2n$ would complete the proof.  The basic idea is to increase the size of $I$ incrementally, starting with $I_\emptyset$ and the corresponding sequence $V_\emptyset$, each time also increasing the size of $|\pi(V)|$. If we can do this, then eventually we obtain $1_G \in G'=\pi(V)$, completing the proof. The difficulty is in showing $|\pi(V)|$ grows quickly enough that we hit the maximal possible value $|\pi(V)|=|G'|=n$ before including too many terms into $I$. This leads us to the definition of ample. When $I$ is ample, we are growing $|\pi(V)|$ fast enough relative to the size of $|I|$, as specified in A1, to be on track to get $1_G \in \pi(V)$ eventually with $|I|\leq 2n$. The bound A1 involves the portion $\left\lfloor\frac12|I_\varnothing^{1}|\right\rfloor+1$, which is how much the terms from $I_\varnothing$ \emph{ought} to contribute for us to be on track, as well as an additional increase in the bound corresponding to an increase of one for each term $i\in I\setminus I_\varnothing$. Of course, it is possible for  A1 to fail initially for $I_\varnothing$ but have later terms $i\in I\setminus I_\varnothing$ increase the size more than they were expected to, offsetting the original deficiency in size, and resulting in A1 holding for the overall sequence $V$.

Note the small discrepancy between $I_\varnothing$ and $I_\emptyset$. We need at least one term from $\tau\la \alpha\ra$ in $V$ for our arguments to work, so if there is none to begin with when considering the initial index set $I_\varnothing$, then we must include an extra term forming $I_\emptyset$  to artificially induce this hypothesis. While this costs us an available $U_i$ with $i\notin I_\emptyset$ that could later be used to increase the size of $|\pi(V)|$, it is balanced by increasing the size of $I_\varnothing^2$. The terms from $I_\varnothing^2$ in $I$ do not need to contribute to the size of $\pi(V)$ (as described in A1). Instead, for each increase in the size of $I_\varnothing^2$ we get an increase in the number of available terms in $[1,w]\setminus I_\varnothing$ (corresponding to the size of $I_\varnothing$ decreasing), as we saw when deriving \eqref{V0not1}. This is sufficient contribution that we don't require more from them in A1.

Our arguments for increasing the size of $|\pi(V)|$ work on two levels, using Kneser's Theorem to increase the size of $|X|$ (which is the contribution we get using element from $\la \alpha\ra$) and using the Partition Theorem derived results to increase the size of $|\Sigma_\ell(2V^+_{\tau\la\alpha\ra})|$ (which is the contribution we get using elements from $\tau\la \alpha\ra$. For the latter to be effective, we must be able to control the maximum multiplicity of a term from $2V^+_{\tau\la \alpha\ra}$, enough to get a sufficiently large $\ell'$ value, as our increase for the size of $|\Sigma_\ell(2V^+_{\tau\la\alpha\ra})|$ is tied directly to the size of $\ell'$. This explains the importance of Claim A, as it ensures our $\ell'$ starts off at a reasonable level.
We use Lemma \ref{lem-triv-multbound} to control the maximum multiplicity, but this will only give us effective bounds when restricted to the subsequence of $V_\varnothing$ indexed by $I_\varnothing^1$.
As a brief example illustrating the potential problem, simply imagine we have $i\in I_\varnothing$ with $x^{[2]}\bdot y^{[2]}$ being the subsequence corresponding to terms from $\tau\la \alpha\ra$ contained in $U_i$. If $i\in I_\varnothing ^1$, then w.l.o.g. $x^{[2]}\bdot y^{[1]}$ is the contribution of this index $i$ to the larger sequence $2V^+_{\tau\la \alpha\ra}$, which contains at least one pair of distinct elements, thus increasing the size of $\ell'$.  On the other hand, if $i\in I^2_\varnothing$, then the contribution might instead be $x^{[2]}$, contributing no  pair of distinct elements, with no resulting increase to the size of $\ell'$ (potentially).

The condition A2/C4 in the definition of ample is to ensure we don't wastefully include terms $i$ in $I$ that add no contribution to the size of $\pi(V)$, at least not terms $i$ with $U_i$ having all terms from $\la \alpha\ra$. What we don't want is an unexpected increase in the size of $|\pi(V)|$ resulting from terms $\tau\la\alpha\ra$ to offset terms from $I^\alpha$ that provide no increase at all.  While the unexpected increase might temporarily offset the wasteful inclusion of terms into $I^\alpha$, it is possible at a latter stage, once $I$ has been increased in size further, that this margin for error will evaporate and these added terms will now drag down the increase in size for $|\pi(V)|$.

Of course, we may fail to achieve the necessary growth described in A1. When this happens, we utilize structural characterizations for sequences and sumsets failing to yield the required growth to give us \emph{local} structural control over the terms in $V$. The conditions described in the notion of "constrained" contain much of the structure we will need, as well as that which is described in Claim B (which helps us in the case A1 doesn't hold with strict inequality). In CASE 1, we assume we begin achieving sufficient growth for $|\pi(V)|$, but that at some point this growth stalls. This forces equality in A1, in which case either we get the structure of $V$ from $I$ being constrained (CASE 1.2) or from Claim B (CASE 1.1). In either case, we then use the structure of $V$, along with the inability to add any further index to $I$ and increase the size of $|\pi(V)|$, to gain control over the \emph{global} structure of \emph{all} terms not contained in $U$, not just those in $V$, which allows us to derive the desired conclusion. The other possibility is contained in CASE 2, where it is assumed every interval $I$ fails to have the required growth. This forces us into the structural description given by being constrained for our initial interval $I_\emptyset$, as well as for all extensions of it. The main complication is to make sure the structural conditions we get for $I$ being constrained align with those we get once to increase the size of $I$ by adding a new index.
If they do not, then the  increment from $I$ to the next interval will completely change the means we  using to measure how quickly $|\pi(V)|$ grows, throwing into question whether we are consistently increasing the size of $|\pi(V)|$ fast enough. To ensure the structural conditions align,  we will introduce the conditions given in D1--D4 later in the proof.
We also need the alternative way to measure the size of $|\pi(V)|$ to ensure our incremental increase in the size of $I$ stops before $I$ grows too large. For this, we utilize that the structural conditions obtained from $I$ being constrained force all but a small number of exceptional terms of $\tau\la \alpha\ra$ to be from the same coset of the stabilizer. If there are no more such exceptional terms among all the $U_i$ with $i\notin I$, then Proposition \ref{prop-regulate} will allow us to shift the global condition $1_G \in \pi(U^{[-1]}\bdot U_1\bdot\ldots\bdot U_w)$  into the needed local one $1_G \in \pi(V)$.
If not, we will be able to increase the size of $\pi(V)$ by including a new index into $I$, and this increase will be controlled (one of the key parts of the argument) via the conditions D1--D4 to ensure the process stops soon enough.
\bigskip

With this briefest  outlines of the overall strategy of the proof, we continue by next  introducing the second claim whose proof is deferred till later, and continue by dividing the main proof into three subcases: CASE 1.1, CASE 1.2 and CASE  2.

\subsection*{Claim B} If $|\pi(V)|\leq \left\lfloor\frac12|I_\varnothing^{1}|\right\rfloor+1+|I\setminus I_\varnothing|$ and C4 holds but  $I$ is not constrained, then $\ell=\ell'$, $|I_\varnothing^1|$ is even, $|\pi(V)|= \frac12|I_\varnothing^{1}|+1+|I\setminus I_\varnothing|=\ell+1$, $L$ is trivial, $|I^{\alpha\tau}|=|I^\alpha|=0$, $X=\{0\}$, and $|U_i|=2$ for every $i\in I^\tau$. Moreover,  $\supp(V_{\tau \la\alpha\ra}^+)\subset \{x-d,x,x+d\}$ with $\vp_x(V_{\tau \la\alpha\ra}^+)=\mathsf h(V^+_{\tau \la\alpha\ra})=\ell\leq \ord(d)-1$, for some $x,d\in \Z/n\Z$.

\bigskip

\noindent
CASE 1. \,  There exists an ample subset $I\subset [1,w]$.

We may w.l.o.g. assume $|I|$ is maximal among all ample subsets. By definition of $I_\varnothing^2$, we have $2n\geq |U|\geq 2|I_\varnothing^2|+|I_\varnothing^1|$. Hence, by A1, we have
$n\geq |\pi(V)|\geq \frac12|I_\varnothing^{1}|+\frac12+|I\setminus I_\varnothing|=|I|+\frac12-\frac12|I_\varnothing^1|-|I_\varnothing^2|\geq |I|+\frac12-n$, implying $|I|\leq 2n-1$. If $|\pi(V)|=n$, then $\pi(V)=G'$ follows. In particular, $1\in \pi(V)$, ensuring that $U\mid_{\mathcal B(G)} \prod_{i\in I}^\bullet U_i$ with $|I|\leq 2n-1$, and the proof is complete. Therefore we may assume $|\pi(V)|\leq n-1$, in which case the above estimates improve to $|I|\leq 2n-2$. We must have $I\subset [1,w]$ proper; otherwise $U\mid_{\mathcal B(G)}\prod_{i\in [1,w]}^\bullet U_i=\prod_{i\in I}^\bullet U_i$ with $|I|\leq 2n-2$, and the proof is again complete.

Let $j\in [1,w]\setminus I$ be arbitrary. The  maximality of $|I|$ ensures that $I_j=I\cup \{j\}$ is not ample, meaning either A1 or A2 fails. Let $V_j$, $\ell_j$, $X_j$  and $I_j^\alpha$ be the respective quantities  $V$, $\ell$, $X$ and $I^\alpha$ for the set $I_j$.
Suppose there is some $j\in [1,w]\setminus I$ and $g\in \supp(U_j)\cap \la\alpha\ra$ with $g\notin H$.
Since $H^+=\mathsf H(X+\Sigma_\ell(2V_{\tau \la\alpha\ra}^+))$, it follows that $\phi_{H}(X)$ and $\phi_{H}(X+\Sigma_\ell(2V_{\tau \la\alpha\ra}^+))$ are both aperiodic.
Thus Kneser's Theorem implies that $|X_j|\geq |\{g^+,-g^+\}+X|>|X|$ and $|\pi(V_j)|\geq |\{g^+,-g^+\}+X+\Sigma_\ell(2V_{\tau \la\alpha\ra}^+)|>|X+\Sigma_\ell(2V_{\tau \la\alpha\ra}^+)|=|\pi(V)|$, so that A1 and A2 holding for $I$ ensures they hold for $I_j$, contradicting the maximality of $|I|$.
So we instead conclude that  \be\label{H-stab}\supp\big({\prod}_{i\in [1,w]\setminus I}^\bullet U_i\big)\cap \la \alpha\ra\subset H.\ee
As a result, if $\supp(U_i)\subset \la \alpha\ra$ for all $i\in [1,w]\setminus I$, then  $\pi(U^{[-1]}\bdot \prod^\bullet_{i\in [1,w]}U_i)^+=\pi(V\bdot \prod^\bullet_{i\in [1,w]\setminus I}U_i)^+=\pi(V)^+$ follows from Proposition \ref{prop-add/mult-DefiningCorrespondance}.2.
Thus, since $1_G \in  \pi(U^{[-1]}\bdot \prod^\bullet_{i\in [1,w]}U_i)$ in view of the hypothesis that $U\mid_{\mathcal B(G)} \prod^\bullet_{i\in [1,w]}U_i$, we conclude that $1_G \in \pi(V)$, so that $U\mid_{\mathcal B(G)} \prod^\bullet_{i\in I}U_i$ with $|I|\leq 2n-2$, and then the proof is complete taking $J=I$.
Therefore let $J_\tau\subset [1,w]\setminus I$ be the nonempty set of all $i\in [1,w]\setminus I$ with $\supp(U_i)\cap \tau \la\alpha\ra\neq \emptyset$.
Note that $I_j^\alpha=I^\alpha$ for all $j\in J_\tau$, so that A2 holds for any $I_j$ with $j\in J_\tau$ since it holds for $I$. This means A1 fails for every $I_j$ with $j\in J_\tau$, which in view of A1 holding for $I$ implies
\be\label{erlaub}|\pi(V)|=\left\lfloor\frac12|I_\varnothing^{1}|\right\rfloor+1+|I\setminus I_\varnothing|.\ee

\medskip
\noindent
CASE 1.1. \, $I$ is not constrained.

In this case,  \eqref{erlaub} and A2 holding for $I$ allow us to apply Claim B yielding $\ell=\ell'$, $|I_\varnothing^1|$ is even, $L$ is trivial, $|I^{\alpha\tau}|=|I^\alpha|=0$, $X=\{0\}$, $|U_i|=2$ for every $i\in I^\tau$, $|\pi(V)|=\ell+1$ and $\supp(V_{\tau \la\alpha\ra}^+)\subset \{x-d,x,x+d\}$ for some $x,d\in \Z/n\Z$ with $\ell\leq \ord(d) -1$ and $\vp_{2x}(2V_{\tau \la\alpha\ra}^+)=\mathsf h(2V_{\tau\la\alpha\ra}^+)=\ell$. Thus we have a decomposition  $V^+_{\tau \la\alpha\ra}=T_1\bdot\ldots\bdot T_\ell$ with $|A_i|=|T_i|=2$ for all $i$, where each  $A_i=\supp(2T_i)$ is an arithmetic progression with difference $2d$ containing $2x$, for all $i\in [1,\ell]$, in which case Lemma \ref{lem-subsums=sumset} (with $H$ taken to be trivial) ensures $\Sigma_{\ell}(2V^+_{\tau \la\alpha\ra})=\Sum{i=1}{\ell}A_i\subset 2\ell x+\la d\ra$ is an arithmetic progression of length $\ell+1$. In particular, either $\ell<\ord(d)-1$ and $H$ is trivial, or $\ell=\ord(d)-1$ and $H^+=\la d\ra$.

Consider an arbitrary index $j\in J_{\tau}$, in which case $\ell_j>\ell$. Since $V_j=V\bdot U_j$, we have  $\Sigma_{\ell_j-\ell}(2(U_j)_{\tau\la\alpha\ra}^+)+\Sigma_\ell(2V_{\tau \la\alpha\ra}^+)\subset \Sigma_{\ell_j}(2(V_j)^+_{\tau \la\alpha\ra})$. Consequently, since A1 fails for $I_j$ (as $j\in J_\tau$), it then  follows from  \eqref{erlaub} that
$|\pi(V)|=|\pi(V_j)|=\lfloor \frac12|I_\varnothing^1|\rfloor +1+|I\setminus I_\varnothing|$, and hence
$|\Sigma_\ell(2V_{\tau \la\alpha\ra}^+)|=|\pi(V)|=|\pi(V_j)|=|\Sigma_{\ell_j}(2(V_j)^+_{\tau \la\alpha\ra})|$ and
\be\label{wilt1}\Sigma_{\ell_j-\ell}(2(U_j)_{\tau\la\alpha\ra}^+)+\Sigma_\ell(2V_{\tau \la\alpha\ra}^+)= \Sigma_{\ell_j}(2(V_j)^+_{\tau \la\alpha\ra})=\beta+\Sigma_\ell(2V_{\tau\la\alpha\ra}^+),\ee for any $\beta\in \Sigma_{\ell_j-\ell}(2(U_j)_{\tau\la\alpha\ra}^+)$. In particular, since $|(U_j)_{\tau\la\alpha\ra}|=2(\ell_j-\ell)\geq 2$, we conclude that all terms of $2(U_j)_{\tau\la\alpha\ra}^+$ are congruent to each other modulo the stabilizer $H^+=\mathsf H(\Sigma_{\ell}(2V_{\tau\la\alpha\ra}^+))$.

Suppose $H^+$ is nontrivial. Then $H^+=\la d\ra$ with $\Sigma_\ell(2V_{\tau\la\alpha\ra}^+)=\Sum{i=1}{\ell}A_i=2\ell x+H^+$, in which case $\Sigma_{\ell_j}(2(V_j)^+_{\tau \la\alpha\ra})=\beta+\Sigma_\ell(2V_{\tau\la\alpha\ra}^+)=\beta+2\ell x+H^+$ is also an $H^+$-coset. However, as $\ell_j<|(V_j)_{\tau \la\alpha\ra}|$, this is only possible if all terms of $(V_j)_{\tau\la\alpha\ra}^+$ lie in the same $H^+$-coset, ensuring that $\supp((V_j)^+_{\tau \la\alpha\ra})\subset x+ H^+=x+\la d\ra$. This must be true for any $j\in J_\tau$, so $\supp\Big((\prod_{i\in [1,w]\setminus I}^\bullet U_i)^+_{\tau \la \alpha\ra}\Big)\subset x+H^+$. Combined with \eqref{H-stab}, we conclude
$\supp\big({\prod}_{i\in [1,w]\setminus I}^\bullet U_i\big)\subset H_x$, and now Proposition \ref{prop-regulate}.2 implies    that $\pi(V)$ is a translate of $\pi\Big(U^{[-1]}\bdot \prod_{i\in [1,w]}^\bullet U_i\Big)=\pi(V\bdot \prod^\bullet_{i\in [1,w]\setminus I}U_i)$. However, since $\prod_{i\in [1,w]\setminus I}^\bullet U_i$ is product-one,  we have $1_G \in \pi\Big(U^{[-1]}\bdot \prod_{i\in [1,w]}^\bullet U_i\Big)=\pi(V)$.  Thus  $U\mid_{\mathcal B(G)}\prod_{i\in I}^\bullet U_i$ with $|I|\leq 2n-2$, and the proof is complete taking $J=I$. So we now instead assume $H$ is trivial and $\ell<\ord(d)-1$.

Combining $H$ trivial  with \eqref{H-stab} implies $\supp(U_i)\subset \tau\la \alpha\ra$ for all $i\in [1,w]\setminus I$ (as we can assume no $U_i$ is the atom consisting of a single term equal to $1_G$). We showed above that all terms of $U_j=(U_j)_{\tau\la\alpha\ra}$ are equal (as $H$ is trivial), for any $j\in J_\tau=[1,w]\setminus I$. Thus Lemma \ref{lem-triv-multbound} ensures each $U_j=g_j^{[2]}$ for some $g_j\in \tau\la \alpha\ra$.  Suppose, for some $j\in [1,w]\setminus I$, that $\supp(U_j^+)=\{y\}$ with $2y\notin A_{j'}$ for some $j'\in [1,\ell]$, say w.l.o.g. $2y\notin A_\ell$.
Observe that $(\ell_j-\ell)2y+\Sum{i=1}{\ell-1}A_i+(A_\ell\cup \{2y\})\subset \Sigma_{\ell_j}(2(V_j)_{\tau\la\alpha\ra}^+)$. Since $\Sum{i=1}{\ell}A_i$ is  aperiodic (as $H^+$ is trivial and $X=\{0\}$) and $2y\notin A_\ell$, Kneser's Theorem and \cite[Lemma 2.6]{Gr21a} imply  $|\Sum{i=1}{\ell-1}A_i+(A_\ell\cup \{2y\})|>\Sum{i=1}{\ell}|A_i|-\ell+1=\ell+1=|\Sum{i=1}{\ell}A_i|$. Thus $|\Sigma_{\ell_j}(2(V_j)^+_{\tau \la\alpha\ra})|>|\Sum{i=1}{\ell}A_i|=|\Sigma_\ell(2V_{\tau\la\alpha\ra}^+)|$, contrary to \eqref{wilt1}. So we are left to conclude that, for any $j\in [1,w]\setminus I$,  all terms of $U_j$ are equal to some multiplicity $\ell$ term in $V_{\tau\la\alpha\ra}$.
If it is always the same multiplicity $\ell$ term for each $U_j$ with $j\in [1,w]\setminus I$, then Proposition \ref{prop-regulate} (taking $H$ to be trivial) implies $\pi(V)=\pi(V\bdot\prod^\bullet_{i\in[1,w]\setminus I}U_i)=\pi\Big(U^{[-1]}\bdot \prod_{i\in [1,w]}^\bullet U_i\Big)$, and the proof is complete as before.

It remains to consider the case when there are two multiplicity $\ell$ terms in $V_{\tau\la\alpha\ra}^+$, so w.l.o.g. $A_i=\{x,x+d\}$ for all $i\in [1,\ell]$, and $(V_j)^+_{\tau \la\alpha\ra}=x^{[2]}$ and $(V_{j'})^+_{\tau \la\alpha\ra}=(x+d)^{[2]}$ for some $j,j'\in [1,w]\setminus I$. Set $A_{\ell+1}=A_{\ell+2}=\{2x,2x+d\}$. Define $J=I\cup \{j,j'\}$ and let $V_J$, $\ell_J$, $I_J^\alpha$ and $\ell'_{J}$ be the corresponding quantities $V$, $\ell$, $I^\alpha$ and $\ell'$ for the set $J$. Then $\ell'_J=\ell'+2=\ell+2=\ell_J$ and $I^\alpha=I_J^\alpha$, ensuring that A2 holds for $J$ since it held for $I$.
Observe (in view of \eqref{H-stab} and Proposition \ref{prop-add/mult-DefiningCorrespondance}) that $$\pi(V_J)^++\sigma(V_{\tau\la\alpha\ra}^+)=\Sum{i=1}{\ell+2}A_i=
{\underbrace{\{2x,2x+2d\}+\ldots+\{2x,2x+2d\}}}_{\ell+2}.$$ If $|X+\Sum{i=1}{\ell+2}A_i|\geq |X+\Sum{i=1}{\ell}A_i|+2$, then A1 holding for $I$ will imply it holds for $J$, in which case $|J|$ contradicts the maximality of $|I|$.
Therefore $|X+\Sum{i=1}{\ell+2}A_i|\leq |X+\Sum{i=1}{\ell}A_i|+1$, which is only possible if $2d\in H^+_J:=\mathsf H(\Sum{i=1}{\ell+1}A_i)=\mathsf H(\Sum{i=1}{\ell+2}A_i)=\mathsf H(\Sigma_{\ell+2}(V_J))=\la d\ra$. As a result, since $\supp((U_j)^+_{\tau \la\alpha\ra})\subset \{x,x+d\}\subset x+H^+_J$ for all $j\in J_\tau$,  we conclude via Proposition \ref{prop-regulate} that  $1_G \in \pi(U^{[-1]}\bdot\prod_{i\in [1,w]}^\bullet U_i)=\pi(V_J\bdot\prod^\bullet_{i\in [1,w]\setminus J}U_i)=\pi(V_J)$. Thus $U\mid_{\mathcal B(G)} \prod_{i\in J}^\bullet U_i$ with $|J|=|I|+2\leq 2n$, completing  CASE 1.1.

\medskip
\noindent
CASE 1.2. \, $I$ is $H_x$-constrained.

Let $2x+H^+$ and $\mathscr A=A_1\bdot\ldots\bdot A_\ell$ be the coset and setpartition showing $I$ is constrained, and let w.l.o.g. $j=\ell$ be the index from C3. Consider an arbitrary index $k\in J_{\tau}$. Then A1 fails for $I_k$, which in view of \eqref{erlaub} implies that $|\pi(V)|=|\pi(V_k)|$. Indeed, since $|\pi(V)|=|X+\Sigma_\ell(2V^+_{\tau\la\alpha\ra})|\leq
|X+\Sigma_{\ell_k}(2(V_k)^+_{\tau\la\alpha\ra})|\leq |X_k+\Sigma_{\ell_k}(2(V_k)^+_{\tau\la\alpha\ra})|=|\pi(V_k)|$, we obtain that
 $|X+\Sigma_\ell(2V^+_{\tau \la\alpha\ra})|=|X+\Sigma_{\ell_k}(2(V_k)^+_{\tau \la\alpha\ra})|$.
Moreover, arguing as we did when establishing \eqref{wilt1}, we conclude that
\be\label{wilt2}X+\Sigma_{\ell_k-\ell}(2(U_k)_{\tau\la\alpha\ra}^+)+\Sigma_\ell(2V_{\tau \la\alpha\ra}^+)= X+\Sigma_{\ell_k}(2(V_k)^+_{\tau \la\alpha\ra})=\beta+X+\Sigma_\ell(2V_{\tau\la\alpha\ra}^+),\ee for any $\beta\in \Sigma_{\ell_k-\ell}(2(U_k)_{\tau\la\alpha\ra}^+)$, and that  all terms of $(U_k)_{\tau\la\alpha\ra}^+$ are congruent to each other modulo the stabilizer $H^+=\mathsf H(X+\Sigma_{\ell}(2V_{\tau\la\alpha\ra}^+))$. We claim that they are, in fact, all congruent to $x$ modulo $H^+$.
If this fails, then there is some $z\in \supp((U_k)_{\tau \la\alpha\ra}^+)$ with $z\notin x+H^+$,
whence \eqref{wilt2}  yields \be\label{chalk}X+\Sigma_{\ell_k}(2(V_k)^+_{\tau \la\alpha\ra})=(\ell_k-\ell)2z+X+\Sigma_\ell(2V_{\tau\la\alpha\ra}^+).\ee
Recall that $j=\ell$ is the index given by C3 and define a new setpartition $\mathscr B=B_1\bdot\ldots\bdot B_\ell$ by setting $B_\ell=\{2y,2z\}$, where $2y\in A_\ell=A_j\subset 2x+H^+$ is any element, and setting $B_i=A_i$ for $i<\ell$. In view of C3 and \eqref{extra-C}, we have $(\ell_k-\ell)2z+X+\Sigma_\ell(2V_{\tau \la\alpha\ra}^+)=(\ell_k-\ell)2z+X+\Sum{i=1}{\ell}A_i\subset (\ell_k-\ell)2z+X+ \Sum{i=1}{\ell}B_i\subset X+ \Sigma_{\ell_k}(2(V_k)_{\tau \la\alpha\ra}^+)$, whence $X+\Sum{i=1}{\ell}A_i=X+ \Sum{i=1}{\ell}B_i$ follows in view of \eqref{chalk}. In particular,
since $X+\Sum{i=1}{\ell-1}A_i$ is a translate of $X+\Sum{i=1}{\ell}A_i$ by C3, and thus has stabilizer $H^+$, Kneser's Theorem implies that all terms of $B_\ell=\{2y,2z\}$ are congruent modulo $H^+$, contradicting the assumption $z\notin x+H^+=y+H^+$.
So we conclude that $\supp((U_k)_{\tau \la\alpha\ra}^+)\subset x+H^+$, as claimed. However, as $j\in J_\tau$ was arbitrary, combining this with \eqref{H-stab} and Proposition \ref{prop-regulate} (as $V$ is $H_x$-constrained) once more yields  $1_G\in \pi(U^{[-1]}\bdot \prod_{i\in [1,w]}^\bullet U_i)=\pi(V\bdot\prod^\bullet_{i\in [1,w]\setminus I}U_i)=\pi(V)$, showing that $U\mid_{\mathcal B(G)} \prod_{i\in I}^\bullet U_i$ with $|I|\leq 2n-2$. Thus the proof is complete taking $J=I$, which completes CASE 1.1, and thus also  CASE 1.

\medskip
\noindent
CASE 2. \,  There is no  ample subset $I\subset [1,w]$.

As a particular instance of the case hypothesis, $I_\emptyset$ is not ample. Since A2/C4
 holds trivially for $I_\emptyset$ (as $i_\emptyset\in I^{\tau}\cup I^{\alpha\tau}$
 when $I_\varnothing \neq I_\emptyset$), this means A1 must fail:
\be\label{A1fails} |X_\emptyset+\Sigma_{\ell_\emptyset}(2(V_\emptyset)_{\tau \la\alpha\ra}^+)|=|\pi(V_\emptyset)|\leq \frac12 |I_\varnothing^1|+|I_\emptyset\setminus I_\varnothing|.\ee
Thus Claim B ensures that $I_\emptyset$ is $(H_\emptyset)_{x_0}$-constrained for some $x_0\in \Z/n\Z$, and we use the abbreviation $\tilde H_\emptyset=(H_\emptyset)_{x_0}$.

Suppose $H_\emptyset$ is trivial. Then $X_\emptyset+\Sigma_{\ell_\emptyset}(2(V_\emptyset)_{\tau \la\alpha\ra}^+)$ is aperiodic. Moreover, from the definitions involved,  $X_\emptyset$ is a sumset of  $|(V_\emptyset)_{\la\alpha\ra}|\geq |I_\emptyset^\alpha|+|I_\emptyset^{\alpha\tau}|$ cardinality two sets, while in view of \eqref{extra-C}, C2 and the definition of $\ell'_\emptyset$,
it follows that $\Sigma_{\ell_\emptyset}(2(V_\emptyset)_{\tau \la\alpha\ra}^+)$ is a sumset of $\ell'_\emptyset$ cardinality two sets (as well as several cardinality one sets).
 Hence it follows from Kneser's Theorem that \be\label{kt-wow}|X_\emptyset+\Sigma_{\ell_\emptyset}(2(V_\emptyset)_{\tau \la\alpha\ra}^+)|\geq |(V_\emptyset)_{\la\alpha\ra}|+1+\ell'_\emptyset\geq |I_\emptyset^\alpha|+|I_\emptyset^{\alpha\tau}|+1+\ell'_\emptyset.\ee
 If $I_\varnothing\neq I_\emptyset$, then $|I_\emptyset\setminus I_\varnothing|=1$, $\ell'_\emptyset\geq 0$,  $(V_\varnothing)_{\la\alpha\ra}=V_\varnothing$, and $|(V_\emptyset)_{\la\alpha\ra}|\geq |(V_\varnothing)_{\la\alpha\ra}|=|V_\varnothing|\geq \max\{1,|I_\varnothing^1|\}$, in which case \eqref{kt-wow} contradicts \eqref{A1fails}. On the other hand, if $I_\varnothing =I_\emptyset$, then $|I_\emptyset\setminus I_\varnothing|=0$, and \eqref{kt-wow} implies
 $|X_\emptyset+\Sigma_{\ell_\emptyset}(2(V_\emptyset)_{\tau \la\alpha\ra}^+)|\geq |I_\varnothing^\alpha|+|I_\varnothing^{\alpha\tau}|+1+\ell'_\varnothing\geq |I_\varnothing^\alpha|+|I_\varnothing^{\alpha\tau}|+1+\frac12|I_\varnothing^\tau|\geq \frac12|I_\varnothing^1|+1$, with the second inequality in view of Claim A. However, this also contradicts \eqref{A1fails}. So we instead conclude that $H_\emptyset$ is nontrivial. We must also have $H_\emptyset$ proper, else $1_G \in \pi(V_\emptyset)$, contradicting \eqref{V0not1}. Since $n=|G'|$ is odd, this forces $3\leq |H_\emptyset|\leq \frac{n}{3}$.

Let $I^e_\emptyset\subset I_\emptyset$ consist of all indices $i\in I_\emptyset$ such that $W_i$ contains some term from $|G\setminus \tilde H_\emptyset|$. Then $\supp(W_i)\subset  \tilde H_\emptyset$ for all $i\in I_\emptyset\setminus I_\emptyset^e$, while $U_i\bdot W_i^{[-1]}=W_i^U$ is a single term if we additionally have $i\in I_\varnothing^1$. It follows that the remaining term from $W_i^U$ in the product-one sequence $U_i$ must also be from $\tilde H_\emptyset$ for $i\in I_\varnothing^1\setminus I_\emptyset^e$. As a result, the atom $U$ contains at least $|I^1_\varnothing\setminus I_\emptyset^e|$ terms from the subgroup $\tilde H_\emptyset$, in which case Proposition \ref{prop-subgroupbound} ensures that \be\label{star-0}|I^1_\varnothing\setminus I^e_\emptyset|\leq n+|H_\emptyset|-2.\ee
In view of \eqref{extra-C-index}, we have $|(V_\emptyset)_{G\setminus \tilde H_\emptyset}|\leq |G'/H_\emptyset|-2$, in which case \be\label{sun-0}|I_\emptyset^e|\leq |(V_\emptyset)_{G\setminus \tilde H_\emptyset}|\leq |G'/H_\emptyset|-2.\ee
Thus $|I_\varnothing^1|\leq n+|H_\emptyset|-2+|(V_\emptyset)_{G\setminus
\tilde H_\emptyset}|$. Averaged with the inequality
$2|I_\varnothing^2|+|I^1_\varnothing|\leq |U|\leq 2n$, we obtain
\be\label{I_0bound}|I_\emptyset|\leq |I_\varnothing|+1\leq \frac12\Big(3n-2+|H_\emptyset|+|(V_\emptyset)_{G\setminus \tilde H_\emptyset}|\Big)+1<2n-\Big(|G'/H_\emptyset|-1-|(V_\emptyset)_{G\setminus  \tilde H_\emptyset}|\Big),\ee
with the final inequality making use of $3\leq |H_\emptyset|\leq \frac{n}{3}$.

Let $I\subset [1,w]$ be a subset containing $I_\emptyset$ with $|I|$ maximal  subject to A2 holding,
 \be\label{max-cond} H_\emptyset\leq H\quad\und\quad |\pi(V)|\geq |\pi(V_\emptyset)|+|I\setminus I_\emptyset||H_\emptyset|.\ee
 Thus our  case hypothesis ensures that A1 fails, allowing us to apply Claim B to conclude $I$ is $H_x$-constrained.
 Let $2x+H^+$ and
$\mathscr A=A_1\bdot\ldots\bdot A_{\ell}$ be the coset and setpartition exhibiting that $I$ is constrained, so $X+\Sum{i=1}{\ell}A_i=X+\Sigma_{\ell}(2V_{\tau \la\alpha\ra}^+)$ by \eqref{extra-C}.
In view of the second condition in \eqref{max-cond} and \eqref{extra-C} (applied to $I_\emptyset$), we have
 $n\geq |\pi(V)|\geq (|(V_\emptyset)_{G\setminus \tilde H_\emptyset}|+1)|H_\emptyset|+|I\setminus I_\emptyset||H_\emptyset|$, and thus
 $|I\setminus I_\emptyset|\leq |G'/H_\emptyset|-1-|(V_\emptyset)_{G\setminus\tilde  H_\emptyset}|$,
 with equality only possible if $H^+=\Z/n\Z$. Thus \eqref{I_0bound} implies $|I|=|I_\emptyset|+|I\setminus I_\emptyset|\leq 2n-1$. Consequently, if $1_G \in \pi(V)$, then taking $J=I$ completes the proof as $|I|\leq 2n-1$. Therefore we may assume $1_G \notin \pi(V)$. In particular, $H$ is proper, in which case the previous estimate improves by one: $|I|\leq 2n-2$.

If $\supp(U_k)\subset  H_x$ for all $i\in [1,w]\setminus I$, then Proposition \ref{prop-regulate}  again ensures $1_G \in \pi(U^{[-1]}\bdot \prod_{i\in [1,w]}^\bullet U_i)=\pi(V\bdot \prod^\bullet_{i\in [1,w]\setminus I}U_i)=\pi(V)$ (as $V$ is $H_x$-constrained), contrary to assumption. Therefore there must be some $k\in [1,w]\setminus I$ with $\supp(U_k)\nsubseteq H_x$. Let $I_k=I\cup \{k\}$, and let $V_k$, $\ell_k$, $X_k$,  $H_k$,  and $I_k^\alpha$ be the respective quantities $V$, $\ell$, $X$,  $H$, and $I^\alpha$ for $I_k$.
If $\supp((U_k)_{\tau \la\alpha\ra}^+)\subset x+H^+$, then \eqref{extra-C} and  Lemma \ref{lem-subsums=sumset} imply that $X+\Sigma_{\ell}(2V_{\tau \la\alpha\ra}^+)$ is a translate of $X+\Sigma_{\ell_k}(2(V_k)_{\tau \la\alpha\ra}^+)$. In particular, $H_\emptyset\leq H\leq H_k$ (the first inclusion follows from \eqref{max-cond}).  Moreover, there must be be some $g\in \supp((U_k)_{\la\alpha\ra})$ with $g\notin H$, and now  Kneser's Theorem ensures that $X_k+\Sigma_{\ell_k}(2(V_k)_{\tau \la\alpha\ra}^+)$ is  strictly larger in size than $X+\Sigma_{\ell}(2V_{\tau \la\alpha\ra}^+)$. Since both these sets are $H_\emptyset$-periodic in view of $H_\emptyset\leq H\leq H_k$, it follows from Proposition \ref{prop-add/mult-DefiningCorrespondance} that
$|\pi(V_k)|=|X_k+\Sigma_{\ell_k}(2(V_k)_{\tau \la\alpha\ra}^+)|\geq |X+\Sigma_{\ell}(2V_{\tau \la\alpha\ra}^+)|+|H_\emptyset|=|\pi(V)|+|H_\emptyset|\geq |\pi(V_\emptyset)|+|I_k\setminus I_\emptyset|\, |H_\emptyset|$, with the final inequality from \eqref{max-cond}.
Since $g\notin H$, we also have $g\notin L\leq H$, so that Kneser's Theorem implies $|X_k|>|X|$, ensuring A2 holds for $I_k$ (as it holds for $I$). It follows that $I_k$ satisfies \eqref{max-cond}, contradicting the maximality of $|I|$. Therefore we instead conclude that \be\label{oncetwice}\supp((U_k)_{\tau \la\alpha\ra}^+)\nsubseteq x+H^+.\ee In particular, $(U_k)^+_{\tau \la\alpha\ra}$ is not trivial, and hence $\supp(U_k)\not\subset \la \alpha\ra$. Thus $I^\alpha=I_k^\alpha$, ensuring that A2 holds for $I_k$ (as it holds for $I$).

Let $S\mid (V_k)_{\tau \la\alpha\ra}$ be a subsequence with $V_{\tau \la\alpha\ra}\mid S$, so $S=V_{\tau\la\alpha\ra}\bdot T'$ for some $T'\mid (U_k)_{\tau\la\alpha\ra}$, for which $|S|=2r$ is maximal subject to there existing a decomposition $S^+=S_1\bdot\ldots\bdot S_r$ with $|S_i|=2$  for all $i\in [1,r]$ and the following holding, where  $B_i=\supp(2S_i)$ for $i\in [1,r]$:
\begin{itemize}
\item[D1.]  $H\leq H_S$, where $H^+_S=\mathsf H(X+\Sum{i=1}{r}B_i)$.
\item[D2.] $B_i\cap (2x+H^+_S)\neq \emptyset$ for $i\in [1,r]$.
\item[D3.] There is a $j\in [1,r]$ with $|X+\underset{i\neq j}{\Sum{i=1}{r}}B_i|=|X+\Sum{i=1}{r}B_i|$.
\item[D4.] $\phi_{H_S}(A_i)=\phi_{H_S}(B_i)$ for all $i\in I_S$, where $I_S\subset [1,\ell]$ is the subset of all $i\in [1,\ell]$ with $|\phi_{H_S}(A_i)|=2$.
\end{itemize}
Note $S=V_{\tau \la\alpha\ra}$ satisfies the above conditions with $A_i=B_i$ for all $i$ in view of $I$ being $H_x$-constrained, so $S$ exists.

Our final claim, needed only for Case 2, is the following.

\subsection*{Claim C} $S=(V_{k})_{\tau \la\alpha\ra}$

\bigskip

In view of Claim C, we have $S=(V_k)_{\tau \la\alpha\ra}$. In particular, $r=\ell_k$. In view of D1 and D2, we can apply Proposition \ref{prop-regulate}.1 (with $H$ taken to be $H_S^+$) to conclude \be\label{dayo}X+\Sum{i=1}{\ell_k}B_i=X+\Sigma_{\ell_k}(2(V_k)_{\tau\la\alpha\ra}^+).\ee
In particular, $H_S\leq H_k$.
In view of \eqref{dayo}, D1 and D4, we see that $X+\Sum{i=1}{\ell_k}B_i=X+\Sigma_{\ell_k}(2(V_k)_{\tau\la\alpha\ra}^+)$, and thus also $X_k+\Sigma_{\ell_k}(2(V_k)_{\tau\la\alpha\ra}^+)$,  is $H^+$-periodic and  contains a translate of the $H^+$-periodic set  $X+\Sum{i=1}{\ell}A_i=X+\Sigma_{\ell}(2V_{\tau\la\alpha\ra}^+)$.
If this translate is a proper subset, then $$|\pi(V_k)|=|X_k+\Sigma_{\ell_k}(2(V_k)_{\tau\la\alpha\ra}^+)|\geq |X+\Sigma_{\ell}(2V_{\tau\la\alpha\ra}^+)|+|H|=|\pi(V)|+|H|.$$ Thus \eqref{max-cond} holds for $I_k=I\cup \{k\}$ as it held for $I$, with $H_\emptyset\leq H\leq H_S\leq H_k$ following from D1. We already noted above D1--D4 that A2 holds for $I_k$, so $|I_k|$ contradicts the maximality of $|I|$ in such case.
Therefore we instead conclude that  $X_k+\Sum{i=1}{\ell_k}B_i$  is equal to a translate of $X+\Sum{i=1}{\ell}A_i$. Consequently, since a translate of $X+\Sigma_\ell(2V_{\tau\la\alpha\ra}^+)=X+\Sum{i=1}{\ell}A_i$ is trivially contained in $X+\Sigma_{\ell_k}(2(V_k)_{\tau\la\alpha\ra}^+)=X+\Sum{i=1}{\ell_k}B_i\subset X_k+\Sum{i=1}{\ell_k}B_i$ (the equalities follows from \eqref{extra-C} and \eqref{dayo}), it follows that $X+\Sum{i=1}{\ell_k}B_i$  is also equal to a translate of $X+\Sum{i=1}{\ell}A_i$,
 whence $H_S=H$.

Since $S=(V_k)_{\tau \la\alpha\ra}$, we have $S=V_{\tau \la\alpha\ra}\bdot (U_k)_{\tau\la\alpha\ra}$. Since $H=H_S$ is the stabilizer of both $X+\Sum{i=1}{\ell_k}B_i$ and $X+\Sum{i=1}{\ell}A_i$, which are simply translates of each other, Kneser's Theorem combined with D4 ensures this is only possible if the cardinality two sets among $\phi_{H}(B_1),\ldots,\phi_{H}(B_{\ell_k})$ are the same as the cardinality two sets among $\phi_{H}(A_1),\ldots,\phi_{H}(A_{\ell})$.
Combined with D2, we conclude that $\phi_{H}(B_1),\ldots,\phi_{H}(B_{\ell_k})$ consists of the cardinality two sets from $\phi_{H}(A_1),\ldots,\phi_{H}(A_{\ell})$ with all other sets equal to $\{\phi_{H}(2x)\}$. Recall that $A_1\bdot\ldots\bdot A_\ell$ partitions the terms from $2V_{\tau \la\alpha\ra}^+$ by its definition, meaning the terms in sets $B_i$ with $|\phi_H(B_i)|\geq 2$  form a subsequence of $2V_{\tau \la\alpha\ra}^+$. Thus   $\supp((U_k)^+_{\tau \la\alpha\ra})=\supp((S\bdot V_{\tau\la\alpha\ra}^{[-1]})^+)\subset x+H^+$, contradicting \eqref{oncetwice}, which completes CASE 2.

\medskip

It remains only to provide the proofs of the three claims introduced above, which we now do.

\begin{proof}[Proof of Claim A] If $I_\emptyset\neq I_\varnothing$, then $\ell_\varnothing=\ell'_\varnothing=|I_\varnothing^\tau|=0$, and the claim is true. Therefore we now assume $I_\emptyset=I_\varnothing$.
Let $j\in I_\varnothing^\tau$ be arbitrary. Then $|W_j^U|=1$ and $\supp(W_j)\subset \tau \la \alpha\ra$.
Let us consider the various possibilities that can occur for $W_j$. If $|W_j|=1$, then $U_j$ is a length two atom, forcing $U_j=w_j^{[2]}$ for some $w_j\in \tau\la\alpha\ra$. However, in such case, we must have $\vp_{w_j}(V_\varnothing)=1$, for if $w_j\in \supp(W_{j'})$ for some $j'\in I_\varnothing\setminus \{j\}$, then, since $W_j^U=\{w_j\}$, this means $U\mid \prod_{i\in I_\varnothing\setminus \{j\}}^\bullet U_i$, contradicting the minimality of $|I_\varnothing|$.  If $|W_j|=2$, then $|U_j|=3$. Consequently, since the number of terms from $\tau\la \alpha\ra$ in a product-one sequence must be even, we conclude that $U_j=g_j\bdot h_j\bdot \alpha^z$ for some $g_j=\tau\alpha^x,h_j=\tau\alpha^y\in \tau\la\alpha\ra$ and $z\in \Z/n\Z$. Since $U_j$ is an atom, we cannot have $\alpha^z=1_G$ while either $x+z=y$ or $y+z=x$. Thus $g_j\neq h_j$. If $|W_j|\geq 3$, then Lemma \ref{lem-triv-multbound} ensures that there are distinct $g_j,\,h_j\in \supp(W_j)$.
Partition $I_\varnothing^\tau=J_1\cup J_2$, with $J_1\subset I_\varnothing^\tau$ consisting of all $j\in I_\varnothing^\tau$ with $|W_j|=1$, and $J_2\subset I_\varnothing^\tau$ consisting of all $j\in I_\varnothing^\tau$ with $|W_j|\geq 2$.
By the above work, $\prod_{j\in J_1}^\bullet w_j\bdot \prod_{j\in J_2}^\bullet (g_j\bdot h_j)\mid (V_\varnothing)_{\tau \la\alpha\ra}$ with $g_j\neq h_j$ for all $j\in J_2$, and the $w_j$ for $j\in J_1$ all distinct.
Suppose $\ell'_\varnothing<\ell_\varnothing$.
Then there is a unique $g\in \tau\la \alpha\ra$ with $\vp_g(V_\varnothing)\geq \ell_\varnothing+1=\ell_\emptyset+1\geq 2$ (the equality follows from the assumption $I_\emptyset=I_\varnothing$). Since $\vp_{w_j}(V_\varnothing)=1$ for all $j\in J_1$, we have $g\neq w_j$ for all $j\in J_1$. Since $g_j\neq h_j$, each $j\in J_2$ has $g\neq g_j$ or $g\neq h_j$.
Thus, swapping the roles of each $g_j$ and $h_j$ as need by, we may  w.l.o.g. assume $\prod_{j\in J_1}^\bullet w_j\bdot \prod_{j\in J_2}^\bullet g_j\mid (V_\varnothing)_{\tau \la\alpha\ra}$ is a sequence of $|J_1|+|J_2|=|I_\varnothing^\tau|$ terms all distinct from $g$.
Since $\vp_g(V_\varnothing)\geq \ell_\varnothing+1>\frac12 |(V_\varnothing)_{\tau \la\alpha\ra}|$, each of these terms can be paired up with a distinct term equal to $g$, showing $\ell'_\varnothing\geq |I_\varnothing^\tau|$.
So $\ell'_\varnothing\geq \min\{\ell_\varnothing,|I_\varnothing^\tau|\}$. Since each $W_i$ with $i\in I_\varnothing^\tau$ contains at least one term from $\tau\la \alpha\ra$, we have $2\ell_\varnothing=|(V_\varnothing)_{\tau \la\alpha\ra}|\geq |I_\varnothing^\tau|$, and Claim A follows.
\end{proof}

\begin{proof}[Proof of Claim B]
By definition of $\ell_{L}$ and $\ell'$, there is a decomposition  $V_{\tau \la\alpha\ra}^+=T_1\bdot\ldots\bdot T_{\ell}$ with $|T_i|=2$  for all $i\in [1,\ell]$,   $|A_i|=2$ for all $i\leq \ell'$, $|\phi_{L}(A_i)|=2$ for all $i\leq \ell_{L}$, and $|\phi_{L}(A_i)|=1$ for all $i> \ell_{L}$, where  $A_i=\supp(2T_i)$ for $i\in [1,\ell]$.

Suppose $\ell_{L}<\ell$, i.e., $\mathsf h(\phi_{L}(V_{\tau \la\alpha\ra}^+))\geq \ell+1$, and let $x\in\Z/n\Z$ be an element with $\phi_{L}(x)$ a maximum multiplicity term in $\phi_{L}(V_{\tau \la\alpha\ra}^+)$. Since $\ell_{L}<\ell$, we have $A_i\cap (2x+L^+)\neq \emptyset$ and $|A_i\setminus (2x+L^+)|\leq 1$ for all $i$. Hence, since $L^+=\mathsf H(X)$,  Proposition \ref{prop-regulate}.1 implies that $X+\Sum{i=1}{\ell}A_i=X+\Sigma_{\ell}(2V_{\tau \la\alpha\ra}^+)=\pi(V)^++\sigma(V^+_{\tau\la\alpha\ra})$, which is $H^+$-periodic by definition. Thus C1 holds. Additionally, since
$L\leq H$, it follows that  C2 holds (in view of  $A_i\cap (2x+L^+)\neq \emptyset$ for all $i$), while C3 holds for any $j>\ell_{L}$ as these sets are subsets of the same $L^+$-coset with $\mathsf H(X)=L^+$. As C4 holds by hypothesis, we conclude that $I$ is constrained, which is contrary to hypothesis. So we instead assume $\ell_{L}=\ell$, i.e., $\mathsf h(\phi_{L}(V_{\tau \la\alpha\ra}^+))\leq \ell$, which also forces $\ell'=\ell$.

Suppose $|X+\Sigma_{\ell}(2V_{\tau \la\alpha\ra}^+)|< \ell|L|+|X|=(|V_{\tau\la \alpha\ra}|-\ell)|L|+ |X|$. Then, in view of $\ell_{L}=\ell$ (which is equivalent to $\mathsf h(\phi_{L}(V_{\tau \la\alpha\ra}^+))\leq \ell$), we can apply Proposition \ref{thm-partition-thm-equi}.2 to $X+\Sigma_{\ell}(2V_{\tau \la\alpha\ra}^+)$ taking $L$ to be $L^+$ and using $2V^+_{\tau\la \alpha\ra}\mid 2V^+_{\tau\la \alpha\ra}$. But now C1--C2 all hold for the resulting setpartition $B_1\bdot\ldots\bdot B_\ell$  given by Proposition \ref{thm-partition-thm-equi}.2, and  C4 holds by hypothesis. Moreover, if C3 fails, then $|X+\Sum{i=1}{j}B_i|\geq |X+\Sum{i=1}{j-1}B_i|+|L|$ for all $j\in [1,\ell]$,
which implies $|X+\Sigma_{\ell}(2V_{\tau \la\alpha\ra}^+)|\geq \ell|L|+|X|$, contrary to assumption. Thus C3 also holds, meaning  $I$ is constrained, which is contrary to hypothesis. So we now instead assume \be\label{XsumBig}|X+\Sigma_{\ell}(2V_{\tau \la\alpha\ra}^+)|\geq \ell|L|+|X|\geq |X|+\ell+|L|-1.\ee

Let $e_\alpha\geq 0$ be the number of indices $i\in I_\varnothing^\alpha$ for which $W_i$ contains some term lying outside $L_\varnothing$. Kneser's Theorem implies $|X_\varnothing|\geq (e_\alpha+1)|L_\varnothing|$. Combined with C4, we find that \be\label{beet1}|X|\geq  (e_\alpha+1)|L_\varnothing|+|I^\alpha|-|I_\varnothing^\alpha|.\ee
For each of the $|I_\varnothing^\alpha|-e_\alpha$ indices $i\in I_\varnothing^\alpha$ not counted by $e_\alpha$, we have $\supp(W_i)\subset L_\varnothing$.
Thus, since $U_i$ is an atom with $W_i^U=U_i\bdot W_i^{[-1]}$ a single term, it follows that the unique term from $W_i^U$ must also lie in $L_\varnothing$.
It follows that $|U_{L_\varnothing}|\geq |I_\varnothing^\alpha|-e_\alpha$. Hence Lemma \ref{lem-HinAlpha} applied to $U$ implies $|I_\varnothing^\alpha|\leq 2|L_\varnothing|-2+e_\alpha.$ Combined with \eqref{beet1}, we obtain
\be\label{beet2} |X|\geq |I^\alpha|+(e_\alpha-1)(|L_\varnothing|-1)+1\geq |I^\alpha|-|L_\varnothing|+2,
 \ee
 with equality only possible if $|I_\varnothing^\alpha|=2|L_\varnothing|-2+e_\alpha$.
 Since each $W_i$ with $i\in I_\varnothing^1\setminus I_\varnothing^\alpha$ contains at least one term from $\tau \la \alpha\ra$, and since each $W_i=U_i$ with $i\in I\setminus (I^\alpha\cup I_\varnothing)=(I\setminus I_\varnothing)\setminus (I^\alpha\setminus I_\varnothing^\alpha)$ contains at least two terms from $\tau \la \alpha\ra$, we have \be\label{ell-bow}\ell\geq (|I\setminus I_\varnothing|-|I^\alpha|+|I_\varnothing^\alpha|)+
 \left\lceil\frac12(|I_\varnothing^1|-|I_\varnothing^\alpha|)\right\rceil=|I\setminus I_\varnothing|-|I^\alpha|+\left\lceil\frac12|I_\varnothing^1|+\frac12|I_\varnothing^\alpha|\right\rceil,\ee with equality only possible if $|(U_i)_{\tau \la\alpha\ra}|=2$ for all $i\in I^{\alpha\tau}$ and $|U_i|=2$ for all $i\in I^\tau$ (since each $|(U_i)_{\tau \la\alpha\ra}|$ must be even).
Combining \eqref{XsumBig}, \eqref{beet2} and \eqref{ell-bow}, we obtain
\be\label{estimate}|X+\Sigma_{\ell}(2V_{\tau \la\alpha\ra}^+)|\geq \left\lceil\frac12 |I_\varnothing^1|+\frac12|I_\varnothing^\alpha|\right\rceil+|I\setminus I_\varnothing|+1+|L|-|L_\varnothing|\geq
\left\lceil\frac12 |I_\varnothing^1|\right\rceil+|I\setminus I_\varnothing|+1.\ee  Since $|X+\Sigma_{\ell}(2V_{\tau \la\alpha\ra}^+)|\leq \left\lfloor\frac12 |I_\varnothing^1|\right\rfloor+|I\setminus I_\varnothing|+1$ holds by hypothesis, we are left to conclude equality holds in \eqref{estimate}  as well as in the estimates \eqref{XsumBig}, \eqref{beet2} and \eqref{ell-bow} used to derive \eqref{estimate}, and that $|I_\varnothing^1|$ is even, $|I_\varnothing^\alpha|=0$ and $L=L_\varnothing$ (lest the second inequality in \eqref{estimate} be strict). Equality in \eqref{beet2} implies $|I_\varnothing^\alpha|=2|L_\varnothing|-2+e_\alpha$, which combined with $|I_\varnothing^\alpha|=0$ forces $|L_\varnothing|=1$ and $e_\alpha=0$.
Since $L=L_\varnothing$ is trivial, Kneser's Theorem implies $|X|\geq |V_{\la\alpha\ra}|+1\geq
2|I^\alpha\setminus I^\alpha_\varnothing|+|I^\alpha_\varnothing|+|I^{\alpha\tau}|+1
 =2|I^\alpha|-|I_\varnothing^\alpha|+|I^{\alpha\tau}|+1=2|I^\alpha|+|I^{\alpha\tau}|+1$.
  As a result, since equality holds in \eqref{beet2}, we are left to conclude  $|V_{\la\alpha\ra}|=|I^\alpha|=|I^{\alpha\tau}|=0$.
  Thus $\supp(V)\subset \tau\la \alpha\ra$ and $X=\{0\}$. As equality holds in \eqref{ell-bow}, we have $|U_i|=2$ for all $i\in I^\tau$. It remains to show $\supp(V^+_{\tau \la\alpha\ra})\subset\{x-d,x,x+d\}$ with $\vp_x(V^+_{\tau \la\alpha\ra})=\mathsf h(V^+_{\tau \la\alpha\ra})=\ell\leq \ord(d)-1$, for some $x,d\in \Z/n\Z$.

Since equality holds in \eqref{estimate} and \eqref{ell-bow} with $X=\{0\}$, $|I_\varnothing^\alpha|\leq |I^\alpha|=0$ and $|I_\varnothing^1|$ even, we have
$|\pi(V)|=|\Sigma_{\ell}(2V_{\tau \la\alpha\ra}^+)|=\frac12 |I_\varnothing^1|+|I\setminus I_\varnothing|+1=\ell+1=\ell'+1$, allowing us to apply Proposition \ref{thm-special-dihedral-ample} (with $m=n=\ell$) to $\Sigma_{\ell}(2V_{\tau }^+\la\alpha\ra)$. If Proposition \ref{thm-special-dihedral-ample}.5 holds, then $I$ is constrained in view of C4 holding by hypothesis, which is contrary to hypothesis.
If Proposition \ref{thm-special-dihedral-ample}.4 holds, then $\supp(V^+_{\tau \la\alpha\ra})\subset\{x-d,x,x+d\}$ with $\vp_x(V^+_{\tau \la\alpha\ra})=\mathsf h(V^+_{\tau \la\alpha\ra})=\ell$, for some $x,d\in \Z/n\Z$. Thus the claim is complete unless $\ell\geq \ord(d)$. However, in this case, each $A_i$ is an arithmetic progression with difference $2d$, so that $|\Sum{i=1}{j}A_i|=\min\{\ord(d),j+1\}$ for all $j\in [1,\ell]$. Thus $\ell\geq \ord(d)$ implies that $|\Sum{i=1}{\ell-1}A_i|=|\Sum{i=1}{\ell}A_i|=\ord(d)$, whence C3 holds as well as C1 with $H^+=\la d\ra$. Since there is some term $2x\in 2V^+_{\tau \la\alpha\ra}$ with multiplicity $\ell$, we obtain  $2x\in A_i$ for all $i$, whence C2 holds. Hence $I$ is constrained as C4 holds by hypothesis, a contradiction. If Proposition \ref{thm-special-dihedral-ample}.3 holds, then each $A_i=\{2x,2x+2d\}$ for some $x,x+d\in \Z/n\Z$. If $\ell\leq \ord(d)-1$, the claim is complete. Otherwise, arguing as in the previous case, we conclude that $I$ is constrained, contrary to hypothesis.
If Proposition \ref{thm-special-dihedral-ample}.2 holds, then $\ell=2$ and $\supp(V^+_{\tau \la\alpha\ra})= x+\la d\ra$ for some $x,d\in\Z/n\Z$ with
$\ord(d)=3$. In this case, the pigeonhole principle ensures there is some $y\in x+\la d\ra$ with $\vp_y(V^+_{\tau \la\alpha\ra})=2$, so the claim follows as $\la d\ra$ is trivially an arithmetic progression with difference $d$ and length $\ord(d)=3$. Finally, we note that Proposition \ref{thm-special-dihedral-ample}.1 cannot hold since this requires $\Z/n\Z$ to contain a subgroup isomorphic to $(\Z/2\Z)^2$. As this exhausts all possibilities, the proof of Claim B is complete.
\end{proof}

\begin{proof}[Proof of Claim C] Assume by contradiction that $T:=S^{[-1]}\bdot (V_k)_{\tau \la\alpha\ra}=(T')^{[-1]}\bdot (U_k)_{\tau\la\alpha\ra}$ is nontrivial. Let $H^+_S=\mathsf H(X+\Sum{i=1}{r}B_i)$. Since $|(V_k)_{\tau \la\alpha\ra}|=2\ell_k$ and $|S|=2r$ are both even, it follows that $|T|$ is even, so $|T|\geq 2$. If there is some $y\in \supp(T^+)\cap (x+H^+_S)$, then setting $S_{r+1}=y\bdot z$ and $B_{r+1}=\supp(2S_{r+1})$, where $z$ is any other term from $T$,  we find that D1--D4 hold for $S\bdot y\bdot z$, contradicting the maximality of $|S|$. Therefore we instead conclude that $\supp(T^+)$ is disjoint from $x+H^+_S$. As a result, there is a two-term subsequence $z_1\bdot z_2\mid T^+$ with $z_1,\,z_2\notin x+H^+_S$. Let $j\in [1,r]$ be an index given by D3.
Let $y\in \supp(S_j)$ be any element,  and define a decomposition $S^+\bdot z_1\bdot z_2=S'_1\bdot\ldots\bdot S'_{r+1}$ and sets $B'_i=\supp(2S'_i)$ as follows: $S'_j=S_j\bdot y^{[-1]}\bdot z_1$, \ $S'_{r+1}=z_2\bdot y$, and $S'_i=S_i$ for $i\neq j,r+1$.
In view of D1 and D3 holding for the original decomposition, it follows that $H\leq H_S\leq H_{S'}$, where $H^+_{S'}=\mathsf H(X+\Sum{i=1}{r+1}B'_i)$, so   D1 holds for the new decomposition.
Since $|B_i|\leq 2$ for all $i$, Kneser's Theorem and D3 imply that $|\phi_{H_S}(B_j)|=1$, which combined with D2 ensures $\supp(S_j)\subset x+H^+_S$. Thus D4 holds for the new decomposition as it held for the original decomposition (in view of $H_S\leq H_{S'}$), and  both terms from $S_j$ lie in $x+H^+_S\subset x+H^+_{S'}$, ensuring that D2 also holds for the new decomposition.
 In order not to contradict the maximality of $|S|$, we are left to conclude that D3 fails for the new decomposition.
 As a result, $|X+\Sum{i=1}{m}B'_i|>|X+\Sum{i=1}{m-1}B'_i|$ for all $m$, ensuring that \be\label{labory}|X+\Sum{i=1}{r+1}B'_i|\geq |X|+r+1\geq |X|+\ell+1\geq |X|+\frac12|I_\varnothing^\tau|+|I^\tau\setminus I^\tau_\varnothing|+\frac12|I_\varnothing^{\alpha\tau}|+|I^{\alpha\tau}\setminus I_\varnothing^{\alpha\tau}|+1.\ee
Let $e_\alpha\geq 0$ be the number of indices $i\in I_\varnothing^\alpha$ for which $W_i$ contains some term lying outside $L_\varnothing$. Since A2/C4 holds for $I$,
arguing as in Claim B when establishing \eqref{beet2}, we conclude that
$|X|\geq |I^\alpha|+(e_\alpha-1)(|L_\varnothing|-1)+1\geq |I^\alpha|-|L_\varnothing|+2$.
In view of A2/C4 and $X_\varnothing$ being $L_\varnothing$-periodic, we trivially have $|X|\geq |X_\varnothing|+|I^\alpha\setminus I_\varnothing^\alpha|\geq |L_\varnothing|+|I^\alpha\setminus I_\varnothing^\alpha|$, which averaged with the previous bound implies $|X|\geq \frac12|I_\varnothing^\alpha|+|I^\alpha\setminus I_\varnothing^\alpha|+1$. Combined with \eqref{labory}, we find $$|X+\Sum{i=1}{r+1}B'_i|\geq \frac12 |I_\varnothing^1|+|I\setminus I_\varnothing|+2=\frac12 |I_\varnothing^1|+|I_k\setminus I_\varnothing|+1.$$ Since a translate of $X+\Sum{i=1}{r+1}B'_i$ lies contained in $X+\Sigma_{\ell_k}(2(V_k)_{\tau \la\alpha\ra}^+)$, it follows from Proposition \ref{prop-add/mult-DefiningCorrespondance} that $|\pi(V_k)|\geq |X+\Sigma_{\ell_k}(2(V_k)_{\tau \la\alpha\ra}^+)|\geq \frac12 |I_\varnothing^1|+|I_k\setminus I_\varnothing|+1$, ensuring that A1 holds for $I_k$. However, since A2/C4 holds for $I$ with $I^\alpha=I_k^\alpha$, it follows that A2 holds for $I_k$, implying that $I_k$ is ample, contrary to case hypothesis. This completes the proof of  Claim C.
\end{proof}
Thus the proof of Theorem \ref{4.1} is complete.
\end{proof}

\section{On the set of distances and the set of catenary degrees} \label{5}

In this section, we study the set of distances and the set of catenary degrees. Our main result is Theorem \ref{5.1},
  which substantially uses Theorem \ref{4.1}. We recall the definition of catenary degrees and summarize some basic properties of distances and catenary degrees.

Let $H$ be an atomic monoid. For an element $a \in H$, let $\mathsf c (a)$ be the smallest $N \in \N_0 \cup \{\infty\}$ with the following property:
\begin{itemize}
\item[] If $z, z' \in \mathsf Z (a)$ are two factorizations of $a$, then there exist factorizations $z=z_0, z_1, \ldots, z_k=z' \in \mathsf Z (a)$ such that $\mathsf d (z_{i-1}, z_i) \le N$ for each $i \in [1,k]$.
\end{itemize}
Then $\mathsf c (a)=0$ if and only if $|\mathsf Z (a)|=1$ (i.e., $a$ has unique factorization) and if $|\mathsf Z (a)|>1$, then $2 \le \mathsf c (a) \le \sup \mathsf L (a)$.
Then
\[
\Ca (H) = \{\mathsf c (a) \colon a \in H \ \text{with} \ \mathsf c (a) > 0  \} \subset \N_0
\]
denotes the {\it set of (positive) catenary degrees},
 and its supremum $\mathsf c (H) = \sup \Ca (H)$ is called the {\it catenary degree} of $H$ (we use the convention that $\sup \emptyset = 0$). It is easy to see that (\cite[Proposition 3.6]{Ge-Ka10a})
\begin{equation} \label{basic-inequality}
2+ \max \Delta (H) \le \mathsf c (H) \le \omega (H) \,.
\end{equation}
Each of the inequalities can be strict,
 and the structure of the sets $\Delta (H)$ and $\Ca (H)$ can be quite arbitrary. We mention a couple of results. For every finite set $\Delta \subset \N$ with $\min \Delta = \gcd \Delta$ (recall  property \eqref{set-of-distances}) there is a finitely generated Krull monoid $H$ with $\Delta (H) = \Delta$ (\cite{Ge-Sc17a}). For every finite set $C \subset \mathbb N_{\ge 2}$, there is a finitely generated Krull monoid $H_1$ and, if $\max C \ge 3$,  a numerical monoid $H_2$ such that $\Ca (H_1) = \Ca (H_2) = C$ (\cite{ON-Pe18a, Fa-Ge19a}). On the other hand, sets of distances and sets of catenary degrees are intervals for transfer Krull monoids over finite groups and for classes of seminormal weakly Krull monoids (\cite{Ge-Zh19a, Ge-Zh16c}).

The main result (Theorem \ref{5.1}) of the present section states that the set of distances and the set of catenary degrees of $\mathcal B (D_{2n})$ are intervals. By Theorem \ref{3.3}, $\mathcal B (D_{2n})$ is neither transfer Krull nor weakly Krull nor seminormal nor does it have the property studied in \cite[Theorem 5.5]{Oh20a} enforcing that  sets of distances are intervals.

\smallskip
\begin{theorem} \label{5.1}~
Let $G$ be a dihedral group of order $2n$, where $n \ge 3$ is odd. Then  $\Delta (D_{2n}) = [1, 2n-2]$ and $\mathsf {Ca}(G) = [2, 2n]$.
\end{theorem}

We start with a simple lemma.

\smallskip
\begin{lemma} \label{5.2}~
Let $n \in \N_{\ge 3}$ be an odd, and $G = \langle \alpha, \tau :\; \alpha^{n} = \tau^{2} = 1_G \text{ and } \tau\alpha = \alpha^{-1}\tau \rangle$.
 Then
\[
 \begin{aligned}
     \mathcal A \big( \{ \alpha, \tau, \alpha\tau \} \big)\   = & \Big\{ \alpha^{[j]} \bdot \tau^{[n-j]} \bdot (\alpha\tau)^{[n-j]} \colon j \in [0,n] \Big\} \,
                                                                     \bigcup \, \\ & \Big\{ \alpha^{[2j-1]} \bdot \tau \bdot \alpha\tau\colon j\in [1,n-1]\Big\} \bigcup \Big\{ \alpha^{[2j]} \bdot \tau^{[2]}, \quad \alpha^{[2j]} \bdot (\alpha\tau)^{[2]}\colon j\in [0 ,n-1] \Big\} \,.
   \end{aligned}
\]
\end{lemma}

\begin{proof} First, it is easy to check that the product-one sequences on the right hand side are indeed atoms.
 Let $S = \alpha^{[k_1]} \bdot \tau^{[k_2]} \bdot (\alpha\tau)^{[k_3]} \in \mathcal A \big( \{ \alpha, \tau, \alpha\tau \} \big)$, where $k_1,k_2,k_3\in \N_0$.
It is easily checked that $S$ must have one of the listed forms if $k_2+k_3\leq 2$. For $k_2+k_3> 2$, Lemma \ref{lem-triv-multbound} implies  $k_2=k_3$, say $k_2=k_3=n-j$ with $j\in [0,n-2]$.
Then w.l.o.g. $S=\alpha^{[k_1-x]}\bdot\alpha\tau\bdot \alpha^{[x]}\bdot \tau\bdot (\alpha\tau\bdot \tau)^{[n-j-1]}$ with $1_G=\alpha^{k_1-x}\cdot\alpha\tau\cdot \alpha^{x}\cdot \tau\cdot (\alpha\tau\cdot \tau)^{n-j-1}$, $0\leq x\leq \min\{n-1,k_1\}$ and $k_1-x\leq n-1$. Since $1_G=\alpha^{k_1-x}\cdot\alpha\tau\cdot \alpha^{x}\cdot \tau\cdot (\alpha\tau\cdot \tau)^{n-j-1}=\alpha^{k_1-2x-j}$, we must have $k_1-2x-j\equiv 0\mod n$.
If $x\geq 1$, then $\alpha^{[k_1-x]}\bdot\alpha\tau\bdot \alpha^{[x-1]}\bdot \tau\bdot (\alpha\tau\bdot \tau)^{[n-j-2]}$ and $\alpha\bdot \tau \bdot \alpha\tau $ are both product-one, contradicting that $S$ is an atom. If $x=0$, then $k_1-2x-j\equiv 0\mod n$ forces $k_1\equiv j\mod n$. However, as $0\leq k_1=k_1-x\leq n-1$ and $j\in [0,n-2]$, this implies $k_1=j$, and now $S$ has the desired form.
\end{proof}

\smallskip
\begin{proof}[Proof of Theorem \ref{5.1}]
Let $n \in \N_{\ge 3}$ be odd and let $G = \langle \alpha, \tau \colon\alpha^{n} = \tau^{2} = 1_G \text{ and } \tau\alpha = \alpha^{-1}\tau \rangle$ be a dihedral group of order $2n$.
Clearly, we have that $[1, n-2] = \Delta (C_n) \subset \Delta (G)$ (\cite[Theorem 6.7.1]{Ge-HK06a})  and $[2, n]\subset \mathsf {Ca}(G)$. We assert  that
\begin{equation} \label{inclusion}
[n-2, 2n-2] \subset \Delta (G) \quad \text{and} \quad [n, 2n]\subset \mathsf {Ca}(G) \,.
\end{equation}
Then Equation \eqref{basic-inequality} and Theorem \ref{4.1} imply that
\[
2n \le 2 + \max \Delta (G) \le \mathsf c (G) \le \omega (G) = 2n \,.
\]
Thus it remains to verify the inclusions \eqref{inclusion}.

Let $U = \tau^{[n]} \bdot (\alpha\tau)^{[n]} \in \mathcal A (G)$. For every $k \in [0,n]$, we let $U_k=\alpha^{[k]} \bdot \tau^{[n-k]} \bdot (\alpha\tau)^{[n-k]}$.
We claim that $\mathsf L(U\bdot U_k)=\{2, 2n-k\}$ for all $k\in [0,n]$, and the assertions then follow by definition.

Let $k\in [0,n]$. Since $U\bdot U_k=(\alpha\bdot \tau\bdot \alpha\tau)^{[k]}\bdot (\tau^{[2]})^{[n-k]}\bdot ((\alpha\tau)^{[2]})^{[n-k]}$, we obtain that $\{2,2n-k\}\subset \mathsf L(U\bdot U_k)$.
Suppose
\[
U \bdot U_k \, = \, \big(\tau^{[n]} \bdot (\alpha\tau)^{[n]}\big) \bdot \big(\alpha^{[k]} \bdot \tau^{[n-k]} \bdot (\alpha\tau)^{[n-k]}\big) \, = \, V_1 \bdot \ldots \bdot  V_{\ell} \,,
\]
where $\ell\ge 3$ and  $V_1, \ldots, V_{\ell} \in \mathcal A (G)$.
If there exists $i \in [1,\ell]$ such that $V_i = \alpha^{[j]} \bdot \tau^{[n-j]} \bdot (\alpha\tau)^{[n-j]}$ for some $j \in [1,k]$, then the remaining sequence is $\alpha^{[k-j]} \bdot \tau^{[n-k+j]} \bdot (\alpha\tau)^{[n-k+j]}$, which is an atom, and hence  $\ell = 2$, a contradiction.
Thus we may assume by Lemma \ref{5.2} that, for every $i \in [1,\ell]$, $\mathsf v_{\tau}(V_i)+\mathsf v_{\alpha\tau}(V_i)=2$. Therefore $\ell=\frac{n+n+n-k+n-k}{2}=2n-k$, and hence $\mathsf L(U\bdot U_k)=\{2, 2n-k\}$.
\end{proof}

\section{On the structure of sets of lengths} \label{6}

For an atomic monoid $H$, unions of sets of lengths $\mathcal U_k (H)$, where $k \in \N$, and sets of elasticities $\mathcal R (H) = \{\rho (L) \colon L \in \mathcal R (H)\}$ are well-studied invariants. Under very mild conditions, unions of sets of lengths are almost arithmetical progressions (e.g., \cite{Tr19a}). For monoids of product-one sequences, both invariants, unions and sets of elasticities, are as simple as possible,  and this is not difficult to obtain. Recall that $\rho_k (H) = \sup \mathcal U_k (H)$ and set $\lambda_k (H) = \min \mathcal U_k (H)$. If $G$ is a finite group with $|G| \le 2$, then $\mathcal B (G)$ is half-factorial, whence
\[
\mathcal L (G) = \bigl\{ \{k\} \colon k \in \N_0 \bigr\} \,.
\]
Thus, whenever convenient, we will assume that $|G| \ge 3$.

\begin{proposition} \label{6.0}
Let $G$ be a finite group with $|G| \ge 3$.
\begin{enumerate}
\item For every $k \in \N$, $\mathcal U_k (G)$ is an interval, $\rho_{2k} (G) = k \mathsf D (G)$, $k \mathsf D (G)+1 \le \rho_{2k+1} (G) \le (2k+1) \mathsf D (G)/2$, and if $|G|> 1$, then $\rho (G) = \mathsf D (G)/2$. If $G$ is dihedral of order $2n$ for some odd $n \ge 3$, then, for every $k \ge 2$ and every $\ell \in \N_0$, $\rho_k (G) = kn$ and
    \[
    \lambda_{2 \ell n + j } (G) = \begin{cases}
                                  2 \ell + j & \ \text{for} \ j \in [0,1] \,, \\
                                  2 \ell + 2 & \ \text{for} \ j \ge 2 \ \text{and} \ \ell = 0 \,, \\
                                  2 \ell + 1 & \ \text{for} \ j \in [2,n] \ \text{and} \ l \ge 1 \,, \\
                                  2 \ell + 2 & \ \text{for} \ j \in [n+1, 2n-1] \ \text{and} \ \ell \ge 1 \,,
                                  \end{cases}
    \]
    provided that $2 \ell n + j \ge 1$.

\item  $\{ \rho (L) \colon L \in \mathcal L (G)\} = \{ q \in \Q \colon 1 \le q \le \mathsf D (G)/2 \}$.
\end{enumerate}
\end{proposition}

\begin{proof}
1. See \cite[Theorem 5.5 and Proposition 5.6]{Oh20a} and \cite[Theorem 5.4]{Oh-Zh20a}.

2. By 1., we have  $\rho (G)= \mathsf D (G)/2$.
By definition of the elasticity $\rho (G)$ (see Equation \eqref{elasticity}), we have
\[
\{ \rho (L) \colon L \in \mathcal L (G)\} \subset \{ q \in \Q \colon 1 \le q \le \rho (G) \} \,.
\]
In order to show that equality holds, we use \cite[Theorem 1.2]{Zh19a}. By that result, it is sufficient to verify that
\[
\inf \{ \overline{\rho} (A) \colon 1 \ne A \in \mathcal B (G) \} = 1 \,,
\quad \text{where} \quad
\overline{\rho } (A) = \lim_{n \to \infty} \rho (\mathsf L (A^{[n]})) \,.
\]
If $g \in G$ and $A = g^{[\ord (g)]}$, then $\mathsf L (A^{[n]})= \{n\}$ for every $n \in \N$ whence $\overline{\rho } (A)=1$ and $\inf \{ \overline{\rho} (B) \colon 1 \ne B \in \mathcal B (G) \} = 1$.
\end{proof}

In this section, we study the structure of sets of lengths over dihedral groups $G$. In order to do so, we consider two distinguished subsets of $\Delta (G)$, namely $\Delta^* (G)$ and $\Delta_{\rho}^* (G)$, which play a crucial role in all structural descriptions of sets of lengths. We start with the definitions of generalizations of arithmetic
progressions.

Let $d \in \N$,  $\ell,\, M \in \N_0$, and  $\{0,d\} \subset
\mathcal D \subset [0,d]$. A subset $L \subset \Z$ is called an

\begin{itemize}
\item {\it almost arithmetical multiprogression}  ({\rm AAMP})  with  {\it difference}  $d$, \ {\it period} \ $\mathcal D$,
       {\it length}  $\ell$, and  {\it bound}  $M$,  if
      \[
      L = y + (L' \cup L^* \cup L'') \, \subset \, y + \mathcal D + d \Z \quad \text{where}
      \]
      \begin{itemize}
      \item  $\min L^* = 0$, $L^* = (\mathcal D + d \mathbb Z) \cap [0, \max L^*]$,
       and  $\ell$ \ is maximal such that \ $\ell d \in L^*$,
      \item  $L' \subset [-M, -1]$, \ $L'' \subset \max L^* + [1,M]$,  and
      \item $y \in \Z$.
      \end{itemize}

\item {\it almost arithmetical progression}  ({\rm AAP})  with
       {\it difference} $d$,  {\it bound} $M$,  and  {\it length}
      $\ell$,  if it is an {\rm AAMP} with difference $d$, period  $\{0,d\}$, bound  $M$, and length  $\ell$.
\end{itemize}

Let $H$ be an atomic monoid. Following \cite[Definition 4.3.12]{Ge-HK06a}, we define
\begin{itemize}
\item $\Delta_1 (H)$ to be the set of all $d \in \N$ having the following property:
      \begin{itemize}
      \item[] For every $k \in \N$, there is some $L_k \in \mathcal L (H)$ that is an AAP with difference $d$ and length at least $k$.
      \end{itemize}

\item  $\Delta^* (H)$ to be the set of all $\min \Delta (S)$ for some divisor-closed submonoid $S \subset H$ with $\Delta (S) \ne \emptyset$.
\end{itemize}
If $H$ is finitely generated, then by \cite[Corollary 4.3.16]{Ge-HK06a}
\begin{equation} \label{Delta-1-Delta^*}
\Delta^* (H) \subset \Delta_1 (H) \subset \{ d_1 \in \Delta (H) \colon d_1 \ \text{divides some} \ d \in \Delta^* (H) \} \subset \Delta (H) \,.
\end{equation}
The  significance of the sets $\Delta^* (H)$ and $\Delta_1 (H)$ stem from the following result (\cite[Theorem 4.4.11]{Ge-HK06a}).

\begin{lemma} \label{6.1}
Let $H$ be a finitely generated monoid. Then there is a constant $M \in \N_0$ such that every $L \in \mathcal L (H)$ is an \text{\rm AAMP} with difference $d \in \Delta^* (H)$ and bound $M$.
\end{lemma}

Let $G$ be a finite group.
By \cite[Lemma 3.3]{Oh20a}, a submonoid $S \subset \mathcal B (G)$ is divisor-closed if and only if $S = \mathcal B (G_0)$ for some subset $G_0 \subset G$. As usual, we set $\Delta^* (G) := \Delta^* ( \mathcal B (G))$. Thus it follows that
\[
\Delta^* (G) = \{ \min \Delta (G_0) \colon G_0 \subset G \ \text{such that} \ \Delta (G_0) \ne \emptyset \} \,.
\]
If $G_1 \subset G_0$ with $\Delta (G_1) \ne \emptyset$, then $\Delta (G_1) \subset \Delta (G_0)$ whence
\[
\min \Delta (G_0) = \gcd \Delta (G_0) \mid \gcd \Delta (G_1) = \min \Delta (G_1) \,.
\]
Thus there exists a minimal non-half-factorial subset $G_0 \subset G$ with $\max \Delta^* (G) = \min \Delta (G_0)$. The set of minimal distances $\Delta^* (G)$ has found much attention in the literature. If $G$ is finite abelian with $|G| >2$, then (by \cite{Ge-Zh16a})
\begin{equation} \label{max-delta*}
\max \Delta^* (G) = \max \{\exp (G)-2, \mathsf r (G)-1 \} \,,
\end{equation}
and $[1, \mathsf r (G)-1] \subset \Delta^* (G)$ (here $\mathsf r (G)$ denotes the rank of $G$). In contrast to $\Delta (G)$, the set $\Delta^* (G)$ is not an interval in general, but there is a characterization when this is the case (\cite[Theorem 1.1]{Zh18a}).
Cross numbers are a crucial tool in the study of half-factorial and minimal non-half-factorial sets. We start with a simple lemma whose proof  runs along the same lines as the proof in the abelian case.

\medskip
\begin{lemma} \label{6.2}
Let $G$ be a finite group and $G_0 \subset G$ a subset. Then the following statements are equivalent{\rm \,:}
\begin{enumerate}
\item[(a)] $G_0$ is half-factorial.

\item[(b)] $\mathsf k (U)=1$ for every $U \in \mathcal A (G_0)$.

\item[(c)] $\mathsf L (A) = \{ \mathsf k (A)\}$ for every $A \in \mathcal B (G_0)$.
\end{enumerate}
\end{lemma}

\begin{proof}
(a) $\Longrightarrow$ (b) Let $U = g_1 \bdot \ldots \bdot g_{\ell} \in \mathcal A (G_0)$ and suppose that $\ord (g_i)=m_i$ for every $i \in [1,\ell]$. For every $i \in [1, \ell]$, we have $U_i = g_i^{[m_i]} \in \mathcal A (G_0)$. We set $m = \lcm (m_1, \ldots, m_{\ell}) $ and $m = m_im_i'$ for every $i \in [1, \ell]$. Then $U^{[m]} = U_1^{[m_1']} \bdot \ldots \bdot U_{\ell}^{[m_{\ell}']}$ and $\mathsf L (U^{[m]})=\{m\}$
imply that
\[
m = m_1' + \ldots + m_{\ell}' = m \mathsf k (U), \quad \text{whence} \quad \mathsf k (U)=1 \,.
\]

(b) $\Longrightarrow$ (c) If $A = U_1 \bdot \ldots \bdot U_m \in \mathcal B (G_0)$, where $m \in \N$ and $U_1, \ldots, U_m \in \mathcal A (G_0)$, then $m = \mathsf k (U_1) + \ldots + \mathsf k (U_m) = \mathsf k (A)$, whence $\mathsf L (A) = \{\mathsf k (A)\}$.

(c) $\Longrightarrow$ (a) Obvious.
\end{proof}

A subset $G_0 \subset G$ is called an LCN-set (large cross number set) if $\mathsf k (U) \ge 1$ for each $U \in \mathcal A (G_0)$. We define
\[
\mathsf m (G) = \max \{ \min \Delta (G_0) \colon G_0 \subset G \ \text{is a non-half-factorial LCN-set} \} \,.
\]

Let $G$ be a finite cyclic group and let $g \in G$ with
 $\ord(g) = |G| \ge 2$. For every product-one
sequence $S = g^{[n_1]} \bdot \ldots \bdot g^{[n_{\ell]}} \in \mathcal F(G)$, \ where $\ell \in \N$ and $n_1, \ldots, n_{\ell}
\in [1, |G|]$, \ we define its $g$-norm
\[
\| S \|_g = \frac{n_1+ \ldots + n_{\ell}}{|G|} \in \N \,.
\]

\begin{lemma} \label{6.3}
Let $G$ be a finite group, $G_0 \subset G$ a non-half-factorial subset, and $e \in \N$ such that $\ord (g) \mid e$ for all $g \in G_0$.
\begin{enumerate}
\item \[
      \min \Delta (G_0) \mid \gcd \bigl\{ e (\mathsf k (U) -1) \colon U \in \mathcal A (G_0) \bigr\} \,.
       \]
       In particular, if there is some $U \in \mathcal A (G_0)$ with $\mathsf k (U) < 1$, then $\min \Delta (G_0) \le e-2$.

\item $\min \Delta (G_0) \le \max \{e-2, \mathsf m (G) \}$.

\item If $\langle G_0 \rangle = \langle g \rangle$ for some $g \in G_0$, then $\min \Delta (G_0) = \gcd  \bigl\{ \|V\|_g -1 \,\colon\, V \in \mathcal A (G_0) \bigr\}$.

\item Let $g \in G$ with $\ord (g)=n > 3$ and let  $a \in [2,n-1]$.
      If  $[a_0,\ldots, a_m]$ is the continued fraction expansion of $n/a$ with odd length (i.e. $m$ is even), then  $\min \Delta(\{g, ag\})=\gcd(a_1, a_3,\ldots, a_{m-1})$.
\end{enumerate}
\end{lemma}

\begin{proof}
1. We set $d = \min \Delta (G_0)=\gcd \Delta(G_0)$
 and choose some $U \in \mathcal A (G_0)$. It is sufficient to show that $d \mid  e (\mathsf k (U) -1)$. We set $U = g_1 \bdot \ldots \bdot g_{\ell}$. Then $U_i = g_i^{[\ord (g_i)]} \in \mathcal A (G_0)$ for all $i \in [1, \ell]$ and
\[
U^{[e]} = {\prod}^\bullet_{i\in [1,\ell]} U_i^{[e/\ord(g_i)]} \quad \text{implies that} \quad e \mathsf k (U) = \sum_{i=1}^{\ell} \frac{e}{\ord (g_i)} \in \mathsf L (U^{[e]}) \,.
\]
Since $e \in \mathsf L (U^{[e]})$, we infer that $d$ divides $e \mathsf k (U) - e$.

If $\mathsf k (U)< 1$, then $e \mathsf k (U) \in [2, e-1]$ whence $e - e \mathsf k (U) \in [1, e-2]$ and thus $d \le e-2$.

2. This follows immediately from 1.

3. follows from \cite[Lemma 6.8.5]{Ge-HK06a} and 4. from \cite[Theorem 2.1]{Ch-Ch-Sm07b}.
\end{proof}

\begin{lemma}\label{lem-atom-transfer}
Let $G$ be a dihedral group of order $2n$ where $n\ge 3$ is odd, say  $G=\langle\alpha,\tau\colon \alpha^n=\tau^2=1_G \text{ and }\alpha\tau=\tau\alpha^{-1}\rangle$. For $U\in \mathcal F(\la \alpha\ra)$, define $$\phi(U)={\prod}^\bullet_{g\in \la \alpha\ra} (g\tau\bdot \tau)^{[\vp_g(U)]}.$$
Then $U$ is an atom if and only if $\phi(U)$ is an atom.
\end{lemma}

\begin{proof} If $U$ is not an atom, then there is a factorization $U=U_1\bdot U_2$ with $U_1,\,U_2\in \mathcal B(G)$ nontrivial product-one sequences. But then $\phi(U)=\phi(U_1)\bdot \phi(U_2)$ is also a factorization of $\phi(U)$ into nontrivial product-one sequences, showing that $\phi(U)$ is not an atom.

 For the other direction, assume $U$ is an atom and let $\vp_{\alpha^{x}}(U)=s_x$ for $x\in [1,n]$. Then $$\Sum{x=1}{n}s_xx\equiv 0\mod n \quad\und \quad \Sum{x=1}{n}s'_xx\not\equiv 0\mod n$$  for any $s'_x\in [0,s_x]$ with $0<\Sum{x=1}{n}s'_x<\Sum{i=1}{n}s_x$. If we take an arbitrary ordering of the terms of $\phi(U)$, then its product equals $\alpha^z$ with $z\equiv \Sum{x=1}{n}\big((s_x-s'_x)x-s'_xx\big)\mod n$, where $s'_x\in [0,s_x]$ is the number of terms equal to $\alpha^x\tau$ occurring as the $j$-th term in the ordering with $j$ even (making $s_x-s'_x$ the number of terms occurring as the $j$-th term in the ordering with $j$ odd). Thus, in any ordering whose product is one, we must have $0\equiv \Sum{x=1}{n}s_xx-2\Sum{x=1}{n}s'_xx\equiv -2\Sum{x=1}{n}s'_xx\mod n$. Since $n$ is odd, this forces $\Sum{x=1}{n}s'_xx\equiv 0 \mod n$, which is only possible if $s'_x=0$ for all $x$, or if $s'_x=s_x$ for all $x$ (as $U$ is an atom). We are left to conclude that, in any ordering of the terms of $\phi(U)$ having product-one, either all terms equal to $g\tau$ with $g\in \supp(U)$ occur at odd places in the ordering, or all occur at even places. From this conclusion, it is now rather immediate that $\phi(U)$ is an atom, completing the proof.
\end{proof}

\begin{proposition} \label{6.5}
Let $G$ be a dihedral group of order $2n$ where $n\ge 3$ is odd, say  $G=\langle\alpha,\tau\colon \alpha^n=\tau^2=1_G \text{ and }\alpha\tau=\tau\alpha^{-1}\rangle$, and  let $G_0=\{\tau,\alpha\tau, \alpha^i\tau\}$, where $i\in [2,n-1]$.
Then $\min\Delta^*(G_0)$ divides $\gcd \left (\min \Delta(\{\alpha,\alpha^{i}\}), \min \Delta(\{\alpha,\alpha^{1-i}\}), \min \Delta(\{\alpha^i,\alpha^{i-1}\})\right)$.	
\end{proposition}

\begin{proof}We set $d=\min\Delta(G_0)$. Clearly,  it is sufficient to show $d\t \min \Delta(\{\alpha,\alpha^{i}\})$, $d\t \min \Delta(\{\alpha,\alpha^{1-i}\})$ and $d\t  \min \Delta(\{\alpha^i,\alpha^{i-1}\})$. Let $d_1=\min \Delta(\{\alpha,\alpha^{i}\})$ and let $U_1,\ldots, U_k, V_1,\ldots, V_{k+d_1}\in \mathcal A(\{\alpha,\alpha^{i}\})$ be atoms with  $U_1\bdot\ldots\bdot U_k=V_1\bdot\ldots\bdot V_{k+d_1}$. Letting $\phi$ be as in Lemma \ref{lem-atom-transfer}, we have
	$$\phi(U_1)\bdot \ldots\bdot \phi(U_k)=\phi(V_1)\bdot \ldots\bdot \phi(V_{k+d_1})\,,$$ with each of the $\phi(U_j)$ and $\phi(V_j)$ atoms over $\{\tau,\alpha\tau,\alpha^i\tau\}$ (by Lemma \ref{lem-atom-transfer}). Since $d=\gcd \Delta(\{\tau,\alpha\tau,\alpha^i\tau\})$, this shows $d\t d_1$.
Using the generating sets $\{\alpha^{-1}, \alpha\tau\}$ and $\{\alpha^{-1}, \alpha^i\tau\}$ in place of $\{\alpha, \tau\}$, we obtain that $\min\Delta(G_0)=\min\Delta(\{\alpha, \alpha\tau, \alpha^i\tau\})=\min\Delta(\{\alpha\tau, \tau, \alpha^{1-i}\tau\})=\min\Delta(\{\alpha^i\tau, \alpha^{i-1}\tau, \tau\})$, and now repeating the above arguments likewise shows $d\t \min \Delta(\{\alpha, \alpha^{1-i}\}))$ and $d\t \min\Delta(\{\alpha^i, \alpha^{i-1}\})$, completing the  proof.
\end{proof}

\begin{lemma}\label{lemma6.7}
	Let $G$ be a dihedral group of order $2n$ where $n\ge 3$ is odd, say $G=\langle\alpha,\tau\colon \alpha^n=\tau^2=1_G \text{ and }\alpha\tau=\tau\alpha^{-1}\rangle$,  and let $G_0=\{\tau, \alpha\tau, \alpha^i\tau\}$ with $i\in [2,n-1]$.
	Every atom $A\in \mathcal A(G_0)$  has one of  the following forms.
	\begin{enumerate}
		\item[Type I:]$\tau^{[2]}, (\alpha\tau)^{[2]}, (\alpha^i\tau)^{[2]}$.
		
		\item[Type II:]  $(\tau\bdot \alpha\tau)^{[n]}, (\tau\bdot \alpha^i\tau)^{[n/\gcd(i,n)]}, (\alpha\tau\bdot\alpha^i\tau)^{[n/\gcd(i-1,n)]}$
		
		\item[Type III:] $(\alpha\tau\bdot \tau)^{[x]}\bdot(\alpha^i\tau\bdot\tau)^{[y]}$, where $x,y\in [1,n-1]$ such that the sequence $\alpha^{[x]}\bdot (\alpha^{i})^{[y]}$ over $\langle\alpha\rangle$ is an atom.
		
		\item[Type IV:] $ (\alpha\tau\bdot \tau)^{[x]}\bdot(\alpha\tau\bdot\alpha^i\tau)^{[y]}$,  where $x,y\in [1,n-1]$ such that the sequence $\alpha^{[x]}\bdot (\alpha^{1-i})^{[y]}$ over $\langle\alpha\rangle$ is an atom.
		
		\item[Type V:] $ (\alpha^i\tau\bdot \alpha\tau)^{[x]}\bdot(\alpha^i\tau\bdot\tau)^{[y]}$,  where $x,y\in [1,n-1]$ such that the sequence $(\alpha^{i})^{[y]}\bdot (\alpha^{i-1})^{[x]}$ over $\langle\alpha\rangle$ is an atom.
	\end{enumerate}
\end{lemma}

\begin{proof}
	Let $A\in \mathcal A(G_0)$ be an atom. Then $|A|$ is even. If $|\supp(A)|=1$, then  $A\in \{\tau^{[2]}, (\alpha\tau)^{[2]}, (\alpha^i\tau)^{[2]}\}$ and hence is of Type I. If $|\supp(A)|=2$, then  $A\in \{(\tau\bdot \alpha\tau)^{[n]}, (\tau\bdot \alpha^i\tau)^{[n/\gcd(i,n)]}, (\alpha\tau\bdot\alpha^i\tau)^{[n/\gcd(i-1,n)]}\}$ and hence is of  Type II.
	
	Suppose $\supp(A)=G_0$ and $A=g_1\bdot g_2\bdot \ldots \bdot g_{2s}$ such that $g_1\cdot \ldots \cdot g_{2s}=1_G$ and $|\{g_1,g_2,\ldots, g_{2s-1}\}|\le |\{g_2,g_4,\ldots, g_{2s}\}|$, where $s\in \N$. Assume to the contrary that $\{g_1, g_3, \ldots, g_{2s-1}\}\cap \{g_2,g_4,\ldots, g_{2s}\}\neq \emptyset$. Then there are an odd $i\in [1,2s]$ and an even $j\in [1,2s]$ such that $g_i=g_j$. If $i=j-1$, then $g_i\cdot g_{i+1}=g_1\cdot\ldots\cdot g_{i-1}\cdot g_{i+2}\cdot\ldots\cdot g_{2s}=1_G$, a contradiction to the fact that $A$ is an atom. If $i\neq j-1$, then $g_i\cdot g_{i+1}\cdot g_{j-1}\cdot g_j=g_i\cdot g_j\cdot g_{j-1}\cdot g_{i+1}$  and hence $$g_i\cdot g_j=g_{j-1}\cdot g_{i+1}\cdot g_1\cdot \ldots\cdot  g_{\min\{i-1,j-2\}}\cdot g_{\min\{i+2,j+1\}}\cdot\ldots\cdot g_{\max\{i-1,j-2\}}\cdot g_{\max\{i+2, j+1\}}\cdot\ldots\cdot g_{2s}=1_G\,,$$
	a contradiction to the fact that $A$ is an atom. Thus $\{g_1, g_3, \ldots, g_{2s-1}\}\cap \{g_2,g_4,\ldots, g_{2s}\}= \emptyset$ which implies that $|\{g_1,g_3,\ldots, g_{2s-1}\}|=1$ and $|\{g_2,g_4,\ldots, g_{2s}\}|=2$
	
	If $g_1=\tau$, then $A$ can be written as  the form $(\alpha\tau\bdot \tau)^{[x]}\bdot(\alpha^i\tau\bdot\tau)^{[y]}$ for some $x,y\in [1,n-1]$.
	If $g_1=\alpha\tau$, then $A$ can be written as the form $(\alpha\tau\bdot \tau)^{[x]}\bdot(\alpha\tau\bdot\alpha^i\tau)^{[y]}$ for some $x,y\in [1,n-1]$.
	If $g_1=\alpha^i\tau$, then $A$ can be written as the form $(\alpha^i\tau\bdot \alpha\tau)^{[x]}\bdot(\alpha^i\tau\bdot\tau)^{[y]}$ for some $x,y\in [1,n-1]$.
	Now we distinguish three cases.
	
	\medskip
	\noindent{CASE 1. } $g_1=\tau$.
	\smallskip
	
We assert that $A$ must be of Type III. It is sufficient to  show that $(\alpha\tau\bdot \tau)^{[x]}\bdot(\alpha^i\tau\bdot\tau)^{[y]}$ is an atom if and only if $\alpha^{[x]}\bdot (\alpha^{i})^{[y]}$ is an atom, where $x,y\in [1,n-1]$. In fact, if $(\alpha\tau\bdot \tau)^{[x]}\bdot(\alpha^i\tau\bdot\tau)^{[y]}$ is an atom, then it is obvious that $\alpha^{[x]}\bdot (\alpha^{i})^{[y]}$ is an atom. Suppose $\alpha^{[x]}\bdot (\alpha^{i})^{[y]}$ is an atom. Then $y\le \frac{n}{\gcd(i,n)}-1$. Assume to the contrary that $W=(\alpha\tau\bdot \tau)^{[x]}\bdot(\alpha^i\tau\bdot\tau)^{[y]}$ is not an atom.
	Then $$W=U_1\bdot\ldots U_k\bdot V_1\bdot\ldots V_{\ell}\,,$$
	where $U_1,\ldots, U_k$ are atoms of length $\ge 3$ and $V_1,\ldots,V_{\ell}$ are atoms of length $2$. We may assume that the factorization above is the factorization of $A$ with $\ell$ is maximal. Since $x\in [1,n-1]$ and $y\in [1, n/\gcd(i,n)-1]$, we obtain that $U_i\not\in \{(\alpha\tau)^{[n]}\bdot \tau^{[n]}, (\alpha^i\tau)^{[n/\gcd(i,n)]}\bdot \tau^{[n/\gcd(i,n)]}\}$ for every $i\in [1, k]$. If there exits $i\in [1,k]$, say $i=1$, such that $U_1=(\alpha\tau\bdot \tau)^{[x_1]}\bdot(\alpha^i\tau\bdot\tau)^{[y_1]}$, where $x_1,y_1\in [1,n-1]$, then $(\alpha\tau\bdot \tau)^{[x-x_1]}\bdot(\alpha^i\tau\bdot\tau)^{[y-y_1]}$ is also a product-one sequence, a contradiction to the fact that $\alpha^{[x]}\bdot (\alpha^{i})^{[y]}$ is an atom.
	If there exist distinct $i,j\in [1,k]$ such that $U_i=(\alpha\tau\bdot \tau)^{[x_1]}\bdot(\alpha\tau\bdot\alpha^i\tau)^{[y_1]}$ and
	$U_j=(\alpha^i\tau\bdot \alpha\tau)^{[x_2]}\bdot(\alpha^i\tau\bdot\tau)^{[y_2]}$, where $x_1,x_2,y_1,y_2\in [1,n-1]$, then
	$$U_i\bdot U_j=(\alpha^i\tau\bdot \alpha\tau)^{[x_2+y_1+\min\{x_1,y_2\}]}\bdot (\tau\bdot \alpha\tau)^{[x_1-\min\{x_1,y_2\}]}\bdot (\alpha^i\tau\bdot \tau)^{[y_2-\min\{x_1,y_2\}]}\bdot (\tau^{[2]})^{[\min\{x_1,y_2\}]}\,,$$
	where $(\alpha^i\tau\bdot \alpha\tau)^{[x_2+y_1+\min\{x_1,y_2\}]}\bdot (\tau\bdot \alpha\tau)^{[x_1-\min\{x_1,y_2\}]}\bdot (\alpha^i\tau\bdot \tau)^{[y_2-\min\{x_1,y_2\}]}$
	is  a product-one sequence, a contradiction to the maximal choice of $\ell$.
	Put all together, we obtain that
	$U_1\bdot\ldots\bdot U_k=(\alpha\tau\bdot \tau)^{[x_3]}\bdot(\alpha\tau\bdot\alpha^i\tau)^{[y_3]}$ or $(\alpha^i\tau\bdot \alpha\tau)^{[x_3]}\bdot(\alpha^i\tau\bdot\tau)^{[y_3]}$, where $x_3,y_3\in \N$.

	If $U_1\bdot\ldots\bdot U_k=(\alpha\tau\bdot \tau)^{[x_3]}\bdot(\alpha\tau\bdot\alpha^i\tau)^{[y_3]}$, then
	$$|\{j\in [1,\ell]\colon V_j=\tau^{[2]}\}|=\frac{x+y-x_3}{2}=\frac{\mathsf v_{\alpha\tau}(W)-x_3+\mathsf v_{\alpha^i\tau}(W) }{2}= \frac{y_3+y_3}{2}=y_3\,.$$
	After renumbering if necessary, we may assume that $V_1=V_2=\ldots=V_{y_3}=\tau^{[2]}$. Then
	$$U_1\bdot\ldots\bdot U_k\bdot V_1\bdot\ldots\bdot V_{y_3}=(\alpha\tau\bdot \tau)^{[x_3]}\bdot(\alpha\tau\bdot \tau\bdot\tau\bdot\alpha^i\tau)^{[y_3]}=(\alpha\tau\bdot \tau)^{[x_3+y_3]}\bdot (\tau\bdot \alpha^i\tau)^{[y_3]}$$
	and $(\alpha\tau\cdot \tau)^{x_3+y_3}(\tau\cdot \alpha^i\tau)^{y_3}=1_G$. Then $x=x_3+y_3$ and $y=y_3$. Since $x+iy\equiv x-iy\equiv 0\pmod n$, we obtain that $n\t x$, a contradiction to $x\in [1,n-1]$.

	If $U_1\bdot\ldots\bdot U_k=(\alpha^i\tau\bdot \alpha\tau)^{[x_3]}\bdot(\alpha^i\tau\bdot\tau)^{[y_3]}$, then
	$$|\{j\in [1,\ell]\colon V_j=\tau^{[2]}\}|= \frac{x+y-y_3}{2}=\frac{\mathsf v_{\alpha\tau}(W)-y_3+\mathsf v_{\alpha^i\tau}(W) }{2}= \frac{x_3+x_3}{2}=x_3\,.$$
	After renumbering if necessary, we may assume that $V_1=V_2=\ldots=V_{x_3}=\tau^{[2]}$. Then
	$$U_1\bdot\ldots\bdot U_k\bdot V_1\bdot\ldots\bdot V_{x_3}=(\alpha^i\tau\bdot\tau\bdot\tau\bdot \alpha\tau)^{[x_3]}\bdot(\alpha^i\tau\bdot\tau)^{[y_3]}=(\alpha^i\tau\bdot \tau)^{[x_3+y_3]}\bdot (\tau\bdot \alpha\tau)^{[x_3]}$$
	and $(\alpha^i\tau\cdot \tau)^{x_3+y_3}(\tau\cdot \alpha\tau)^{x_3}=1_G$. Then $x=x_3$ and $y=x_3+y_3$. Since $x+iy\equiv -x+iy\equiv 0\pmod n$, we obtain that $n\t x$, a contradiction to $x\in [1,n-1]$.

\medskip
\noindent{CASE 2. } $g_1=\alpha\tau$.
\smallskip

 We assert that $A$ must be of Type IV. It is sufficient to show that $(\alpha\tau\bdot \tau)^{[x]}\bdot(\alpha\tau\bdot\alpha^i\tau)^{[y]}$ is an atom if and only if $\alpha^{[x]}\bdot (\alpha^{1-i})^{[y]}$ is an atom, where $x,y\in [1,n-1]$. In fact, if $(\alpha\tau\bdot \tau)^{[x]}\bdot(\alpha\tau\bdot\alpha^i\tau)^{[y]}$ is an atom, then it is obvious that $\alpha^{[x]}\bdot (\alpha^{1-i})^{[y]}$ is an atom. Suppose $\alpha^{[x]}\bdot (\alpha^{1-i})^{[y]}$ is an atom. Then $y\le n/\gcd(i-1,n)-1$. Assume to the contrary that $W=(\alpha\tau\bdot \tau)^{[x]}\bdot(\alpha\tau\bdot\alpha^i\tau)^{[y]}$ is not an atom.
	Then $$W=U_1\bdot\ldots U_k\bdot V_1\bdot\ldots V_{\ell}\,,$$
	where $U_1,\ldots, U_k$ are atoms of length $\ge 3$ and $V_1,\ldots,V_{\ell}$ are atoms of length $2$. We may assume that the factorization above is the factorization of $A$ with $\ell$ is maximal. Since $x\in [1,n-1]$ and $y\in [1, n/\gcd(i-1,n)-1]$, we obtain that $U_i\not\in \{(\alpha\tau)^{[n]}\bdot \tau^{[n]}, (\alpha\tau)^{[n/\gcd(i-1,n)]}\bdot (\alpha^i\tau)^{[n/\gcd(i-1,n)]}\}$ for every $i\in [1, k]$. If there exits $i\in [1,k]$, say $i=1$, such that $U_1=(\alpha\tau\bdot \tau)^{[x_1]}\bdot(\alpha\tau\bdot\alpha^i\tau)^{[y_1]}$, where $x_1,y_1\in [1,n-1]$, then $(\alpha\tau\bdot \tau)^{[x-x_1]}\bdot(\alpha\tau\bdot\alpha^i\tau)^{[y-y_1]}$ is also a product-one sequence, a contradiction to $\alpha^{[x]}\bdot (\alpha^{1-i})^{[y]}$ is an atom.
	If there exit distinct $i,j\in [1,k]$ such that $U_i=(\alpha\tau\bdot \tau)^{[x_1]}\bdot(\alpha^i\tau\bdot\tau)^{[y_1]}$ and
	$U_j=(\alpha^i\tau\bdot \alpha\tau)^{[x_2]}\bdot(\alpha^i\tau\bdot\tau)^{[y_2]}$, where $x_1,x_2,y_1,y_2\in [1,n-1]$,
	then
	$$U_i\bdot U_j=(\alpha^i\tau\bdot \tau)^{[y_1+y_2+\min\{x_1,x_2\}]}\bdot (\alpha\tau\bdot \tau)^{[x_1-\min\{x_1,x_2\}]}\bdot (\alpha^i\tau\bdot \alpha\tau)^{[x_2-\min\{x_1,x_2\}]}\bdot ((\alpha\tau)^{[2]})^{[\min\{x_1,x_2\}]}\,,$$
	where $(\alpha^i\tau\bdot \tau)^{[y_1+y_2+\min\{x_1,x_2\}]}\bdot (\alpha\tau\bdot \tau)^{[x_1-\min\{x_1,x_2\}]}\bdot (\alpha^i\tau\bdot \alpha\tau)^{[x_2-\min\{x_1,x_2\}]}$
	is  a product-one sequence, a contradiction to the maximal choice of $\ell$.
	Put all together, we obtain that $U_1\bdot\ldots\bdot U_k=(\alpha\tau\bdot \tau)^{[x_3]}\bdot(\alpha^i\tau\bdot\tau)^{[y_3]}$ or $(\alpha^i\tau\bdot \alpha\tau)^{[y_3]}\bdot(\alpha^i\tau\bdot\tau)^{[x_3]}$, where $x_3,y_3\in \N$.

	If $U_1\bdot\ldots\bdot U_k=(\alpha\tau\bdot \tau)^{[x_3]}\bdot(\alpha^i\tau\bdot\tau)^{[y_3]}$, then
	$$|\{j\in [1,\ell]\colon V_j=(\alpha\tau)^{[2]}\}|= \frac{x+y-x_3}{2}=\frac{\mathsf v_{\tau}(W)-x_3+\mathsf v_{\alpha^i\tau}(W) }{2}= \frac{y_3+y_3}{2}=y_3\,.$$
	After renumbering if necessary, we may assume that $V_1=V_2=\ldots=V_{y_3}=(\alpha\tau)^{[2]}$. Then
	$$U_1\bdot\ldots\bdot U_k\bdot V_1\bdot\ldots\bdot V_{y_3}=(\alpha\tau\bdot  \tau)^{[x_3]}\bdot(\alpha^i\tau\bdot \alpha\tau\bdot \alpha\tau\bdot \tau)^{[y_3]}=(\alpha\tau\bdot \tau)^{[x_3+y_3]}\bdot (\alpha^i\tau\bdot\alpha\tau )^{[y_3]}$$
	and $(\alpha\tau\cdot \tau)^{x_3+y_3}(\alpha^i\tau\cdot \alpha\tau)^{y_3}=1_G$. Then $x=x_3+y_3$ and $y=y_3$. Since $x+(1-i)y\equiv x+(i-1)y\equiv 0\pmod n$, we obtain that $n\t x$, a contradiction to $x\in [1,n-1]$.
	
	If $U_1\bdot\ldots\bdot U_k=(\alpha^i\tau\bdot \alpha\tau)^{[y_3]}\bdot(\alpha^i\tau\bdot\tau)^{[x_3]}$, then
	$$|\{j\in [1,\ell]\colon V_j=(\alpha\tau)^{[2]}\}|= \frac{x+y-y_3}{2}=\frac{\mathsf v_{\tau}(W)+\mathsf v_{\alpha^i\tau}(W)-y_3 }{2}= \frac{x_3+x_3}{2}=x_3\,.$$
	After renumbering if necessary, we may assume that $V_1=V_2=\ldots=V_{x_3}=(\alpha\tau)^{[2]}$. Then
	$$U_1\bdot\ldots\bdot U_k\bdot V_1\bdot\ldots\bdot V_{x_3}=(\alpha^i\tau\bdot \alpha\tau)^{[y_3]}\bdot(\alpha^i\tau\bdot\alpha\tau\bdot\alpha\tau\bdot\tau)^{[x_3]}=(\alpha^i\tau\bdot \alpha\tau)^{[x_3+y_3]}\bdot (\alpha\tau\bdot \tau)^{[x_3]}$$
	and $(\alpha^i\tau\cdot \alpha\tau)^{x_3+y_3}(\alpha\tau\cdot \tau)^{x_3}=1_G$. Then $x=x_3$ and $y=x_3+y_3$. Since $x+(1-i)y\equiv x+(i-1)y\equiv 0\pmod n$, we obtain that $n\t x$, a contradiction to $x\in [1,n-1]$.

	\medskip
	\noindent{CASE 3. } $g_1=\alpha^i\tau$.
	\smallskip
	
	We assert that $A$ must be of Type V. It is sufficient to show that $(\alpha^i\tau\bdot \alpha\tau)^{[x]}\bdot(\alpha^i\tau\bdot\tau)^{[y]}$ is an atom if and only if $(\alpha^{i-1})^{[x]}\bdot (\alpha^{i})^{[y]}$ is an atom, where $x,y\in [1,n-1]$. In fact, if $(\alpha^i\tau\bdot \alpha\tau)^{[x]}\bdot(\alpha^i\tau\bdot\tau)^{[y]}$ is an atom, then it is obvious that $(\alpha^{i-1})^{[x]}\bdot (\alpha^{i})^{[y]}$ is an atom.
	Suppose $(\alpha^{i-1})^{[x]}\bdot (\alpha^{i})^{[y]}$ is an atom. Then $x\le n/\gcd(i-1,n)-1$ and $y\le n/\gcd(i,n)-1$. Assume to the contrary that $W=(\alpha^i\tau\bdot \alpha\tau)^{[x]}\bdot(\alpha^i\tau\bdot\tau)^{[y]}$ is not an atom.
	Then $$W=U_1\bdot\ldots U_k\bdot V_1\bdot\ldots V_{\ell}\,,$$
	where $U_1,\ldots, U_k$ are atoms of length $\ge 3$ and $V_1,\ldots,V_{\ell}$ are atoms of length $2$. We may assume that the factorization above is the factorization of $A$ with $\ell$ is maximal. Since  $x\in [1, n/\gcd(i-1,n)-1]$ and $y\in [1, n/\gcd(i,n)-1]$, we obtain that $U_i\not\in \{(\alpha\tau)^{[n/\gcd(i-1,n)]}\bdot (\alpha^i\tau)^{[n/\gcd(i-1,n)]}, (\alpha^i\tau)^{[n/\gcd(i,n)]}\bdot \tau^{[n/\gcd(i,n)]}\}$ for every $i\in [1, k]$. If there exits $i\in [1,k]$, say $i=1$, such that $U_1=(\alpha^i\tau\bdot \alpha\tau)^{[x_1]}\bdot(\alpha^i\tau\bdot\tau)^{[y_1]}$, where $x_1,y_1\in [1,n-1]$, then $(\alpha^i\tau\bdot \alpha\tau)^{[x-x_1]}\bdot(\alpha^i\tau\bdot\tau)^{[y-y_1]}$ is also a product-one sequence, a contradiction to $(\alpha^{i-1})^{[x]}\bdot (\alpha^{i})^{[y]}$ is an atom.
	If there exit distinct $i,j\in [1,k]$ such that $U_i=(\alpha\tau\bdot \tau)^{[x_1]}\bdot(\alpha^i\tau\bdot\tau)^{[y_1]}$ and
	$U_j=(\alpha\tau\bdot \tau)^{[x_2]}\bdot(\alpha\tau\bdot\alpha^i\tau)^{[y_2]}$, where $x_1,x_2,y_1,y_2\in [1,n-1]$, then
	$$U_i\bdot U_j=(\alpha\tau\bdot \tau)^{[x_1+x_2+\min\{y_1,y_2\}]}\bdot (\alpha^i\tau\bdot \tau)^{[y_1-\min\{y_1,y_2\}]}\bdot (\alpha\tau\bdot \alpha^i\tau)^{[y_2-\min\{y_1,y_2\}]}\bdot ((\alpha^i\tau)^{[2]})^{[\min\{y_1,y_2\}]}\,,$$
	where $(\alpha\tau\bdot \tau)^{[x_1+x_2+\min\{y_1,y_2\}]}\bdot (\alpha^i\tau\bdot \tau)^{[y_1-\min\{y_1,y_2\}]}\bdot (\alpha\tau\bdot \alpha^i\tau)^{[y_2-\min\{y_1,y_2\}]}$
	is  a product-one sequence, a contradiction to the maximal choice of $\ell$.
	Put all together, we obtain that $U_1\bdot\ldots\bdot U_k=(\alpha\tau\bdot \tau)^{[x_3]}\bdot(\alpha^i\tau\bdot\tau)^{[y_3]}$ or $(\alpha\tau\bdot \tau)^{[x_3]}\bdot(\alpha\tau\bdot\alpha^i\tau)^{[y_3]}$, where $x_3,y_3\in \N$.

	If $U_1\bdot\ldots\bdot U_k=(\alpha\tau\bdot \tau)^{[x_3]}\bdot(\alpha^i\tau\bdot\tau)^{[y_3]}$, then
	$$|\{j\in [1,\ell]\colon V_j=(\alpha^i\tau)^{[2]}\}|= \frac{x+y-y_3}{2}=\frac{\mathsf v_{\alpha\tau}(W)+\mathsf v_{\tau}(W)-y_3 }{2}= \frac{x_3+x_3}{2}=x_3\,.$$
	After renumbering if necessary, we may assume that $V_1=V_2=\ldots=V_{x_3}=(\alpha^i\tau)^{[2]}$. Then
	$$U_1\bdot\ldots\bdot U_k\bdot V_1\bdot\ldots\bdot V_{x_3}=(\alpha\tau\bdot \alpha^i\tau\bdot \alpha^i\tau\bdot \tau)^{[x_3]}\bdot(\alpha^i\tau\bdot \tau)^{[y_3]}=(\alpha^i\tau\bdot \tau)^{[x_3+y_3]}\bdot (\alpha\tau\bdot\alpha^i\tau )^{[x_3]}$$
	and $(\alpha^i\tau\cdot \tau)^{x_3+y_3}(\alpha\tau\cdot \alpha^i\tau)^{x_3}=1_G$. Then $x=x_3$ and $y=x_3+y_3$. Since $(i-1)x+iy\equiv (1-i)x+iy\equiv 0\pmod n$, we obtain that $n\t iy$, a contradiction to $y\in [1,n/\gcd(i,n)-1]$.

	If $U_1\bdot\ldots\bdot U_k=(\alpha\tau\bdot \tau)^{[x_3]}\bdot(\alpha\tau\bdot\alpha^i\tau)^{[y_3]}$, then
	$$|\{j\in [1,\ell]\colon V_j=(\alpha^i\tau)^{[2]}\}|= \frac{x+y-y_3}{2}=\frac{\mathsf v_{\alpha\tau}(W)-y_3+\mathsf v_{\tau}(W) }{2}= \frac{x_3+x_3}{2}=x_3\,.$$
	After renumbering if necessary, we may assume that $V_1=V_2=\ldots=V_{x_3}=(\alpha^i\tau)^{[2]}$. Then
	$$U_1\bdot\ldots\bdot U_k\bdot V_1\bdot\ldots\bdot V_{x_3}=(\alpha\tau\bdot\alpha^i\tau\bdot \alpha^i\tau\bdot \tau)^{[x_3]}\bdot(\alpha\tau\bdot\alpha^i\tau)^{[y_3]}=(\alpha\tau\bdot \alpha^i\tau)^{[x_3+y_3]}\bdot (\alpha^i\tau\bdot \tau)^{[x_3]}$$
	and $(\alpha\tau\cdot \alpha^i\tau)^{x_3+y_3}(\alpha^i\tau\cdot \tau)^{x_3}=1_G$. Then $y=x_3$ and $x=x_3+y_3$. Since $(i-1)x+iy\equiv (1-i)x+iy\equiv 0\pmod n$, we obtain that $n\t iy$, a contradiction to $y\in [1,n/\gcd(i,n)-1]$.
\end{proof}

\medskip
\begin{theorem} \label{6.6}
Let $G$ be a finite group with $|G| \ge 3$.
\begin{enumerate}
\item $1 \in \Delta^* (G)$ and $\{ \ord (g)-2 \colon g \in G \ \text{with} \ \ord (g) \ge 3 \} \subset \Delta^* (G)$.

\item $\max\Delta^*(G) \le \mathsf D(G)-2$.
		
\item $\max\Delta^*(G)= |G|-2$  if and only if $G$ is cyclic or a dihedral group of order $2n$ for some odd $n \ge 3$.
		
\item If $G$ is a dihedral group of order $2n$ for some odd $n \ge 3$, then $\{1,2,n-2,2n-2\}\subset \Delta^*(G)$ and $\max\Delta^*(G)\setminus\{2n-2\}=\max\{2,n-2\}$.

\item If $\max\Delta^*(G)= \mathsf D(G)-2$, then $G$ is cyclic or there is a subgroup $G_1\subset G$ such that $\mathsf D(G_1)=\mathsf D(G)$ and $G_1$ is generated by elements of order two.
\end{enumerate}
\end{theorem}

\begin{proof}
1. Suppose there is  $g \in G$ with $\ord (g) =n \ge 3$.  Since $1 \in \Delta^* (C_n) $ by \cite[Proposition 6.8.2]{Ge-HK06a}, it follows that $1 \in \Delta^* (G)$. Since  $\Delta (\{g, g^{-1}\}) = \{\ord (g) - 2\}$, we infer that $\ord (g) - 2 = \min \Delta (\{g, g^{-1}\}) \in \Delta^* (G)$. Suppose that all elements of $G$ have order two. Then $G$ is an elementary $2$-group and since $|G| \ge 3$, $G$ has a subgroup isomorphic to $C_2 \oplus C_2$. Then, again by \cite[Proposition 6.8.2]{Ge-HK06a}, we obtain that $1 \in \Delta^* (C_2 \oplus C_2) \subset \Delta^* (G)$.

2.
Let $G_0\subset G$ be a non half-factorial subset.
Suppose there exists an atom $A\in \mathcal A(G_0)$ such that $\mathsf k(A)<1$. We assume that $\mathsf k(A)$ is minimal.
	Let $A=g_1\bdot\ldots\bdot g_{\ell}$, where $\ell\in \N_{\ge 2}$ and $g_1,\ldots, g_{\ell}\in G_0$. Since $A$ is product-one, we can index the terms of $A$ such that $g_1\cdot\ldots\cdot g_\ell=1_G$ (and cyclically permuting such an ordering allows   $g_1\in \supp(A)$ to be arbitrary), meaning $(g_2\cdot\ldots\cdot g_\ell)=g_1^{-1}$, which ensures that $(g_2\cdot\ldots\cdot g_\ell)^{[\ord(g_1)]}$ is product-one. Hence $g_1^{[\ord(g_1)]}$ divides $A^{[\ord(g_1)]}$, so
$A^{[\ord(g_1)]}=U_1\bdot U_2\bdot\ldots\bdot U_{\ell_0}$ for some $U_i\in \mathcal A(G_0)$ with $U_1=g_1^{[\ord(g_1)]}$. But then $\mathsf k(U_1)=1$ and $\mathsf k(U_i)\geq \mathsf k(A)$ (in view of the minimality of $\mathsf k(A)$), whence $\ord(g_1)\mathsf k(A)=\mathsf k(A^{[\ord(g_1)]})\geq 1+(\ell_0-1)\mathsf k(A)>\ell_0\mathsf k(A)$, implying $\ell_0<\ord(g_1)$. It follows that    there exists $\ell_0\in \N$ with $2\le\ell_0<\ord(g_1)$ such that $\{\ord(g_1), \ell_0\}\subset \mathsf L(A^{[\ord(g_1)]})$, which implies that
	\begin{equation}\label{eq1}
	\min\Delta(G_0)\le \ord(g_1)-\ell_0\le \ord(g_1)-2\le \mathsf D(G)-2\,.
	\end{equation}
	
	Suppose $\mathsf k(A)\ge 1$ for all $A\in \mathcal A(G_0)$. Since $G_0$ is not half-factorial, Lemma \ref{6.2} implies there exists $A\in \mathcal A(G_0)$ with $\mathsf k(A)>1$.   Let $A=g_1\bdot\ldots\bdot g_{\ell}$, where $\ell\in \N_{\ge 2}$ and $g_1,\ldots, g_{\ell}\in G_0$, and let $B=g_1^{[\ord(g_1)]}\bdot\ldots\bdot g_{\ell}^{[\ord(g_{\ell})]}$. Then $B\in \mathcal B(G_0)$ and $A$ divides $B$ in $\mathcal B(G_0)$, so $B=A\bdot U_2\bdot\ldots\bdot U_{\ell_0}$ for some $U_i\in \mathcal A(G_0)$. But now $\ell=\mathsf k(B)=\mathsf k(A)+\mathsf k(U_2)+\ldots+\mathsf k(U_{\ell_0})>\ell_0$. Therefore there exists $\ell_0 \in [2, \ell - 1]$ such that  $\{\ell, \ell_0\}\subset \mathsf L(B)$, which implies that
	\begin{equation}\label{eq2}
	\min\Delta(G_0)\le \ell-\ell_0\le |A|-2\le \mathsf D(G)-2\,.
	\end{equation}
	Since $G_0$ is arbitrary, we obtain  $\max\Delta^*(G)\le \mathsf D(G)-2$.

3.(a). If $G$ is a cyclic group, then $\max\Delta^*(G)=|G|-2$ by \eqref{max-delta*}. Let $G$ be a dihedral group of order $2n$ where $n\ge 3$ is odd, say $G=\langle\alpha, \tau\colon \alpha^n=\tau^2=1_G \text{ and }\alpha\tau=\tau\alpha^{-1}\rangle$, and set $G_0=\{\alpha\tau, \tau\}$. Then $\min\Delta(G_0)=2n-2=|G|-2$ which, together with 2., implies that $\max\Delta^* (G)=|G|-2$.
	
3.(b). Suppose $\max\Delta^*(G)=|G|-2$.  Then Item 2. implies that $|G|\le \mathsf D(G)$ whence the  assertion follows from Proposition \ref{2.3}.
	
4. Let $G$ be a dihedral group of order $2n$, where $n \ge 3$ is odd, say say $G = \langle \alpha, \tau \colon\alpha^{n} = \tau^{2} = 1_G \text{ and } \tau\alpha = \alpha^{-1}\tau \rangle$.

4.(a). Items 1  implies that $\{1,n-2\}\subset  \Delta^*(C_n)\subset \Delta^*(G)$ and Item 3 implies that $2n-2 = |G|-2 \in \Delta^* (G)$. We assert that $2=\min\Delta(\{\alpha,\tau\})$. Note that $\mathcal A(\{\tau, \alpha\})=\{\alpha^{[n]}\}\cup \{\alpha^{[2i]}\bdot\tau^{[2]}\colon i\in [0,n-1]\}$. Since $(\alpha^{[n]})^{[2]}\bdot (\tau^{[2]})^{[2]}=(\alpha^{[2n-2]}\bdot \tau^{[2]})\bdot (\alpha^{[2]}\bdot \tau^{[2]})$, we obtain that $\min\Delta(\{\tau,\alpha\})\le 4-2=2$.
	Suppose $U_1,\ldots, U_k,V_1,\ldots,V_\ell\in \mathcal A(\{\tau,\alpha\})$, where $k,\ell\in \N$ with $k<\ell$, such that $U_1\bdot\ldots\bdot U_k=V_1\bdot\ldots\bdot V_{\ell}$ and $\{U_1, \ldots, U_k\}\cap \{V_1,\ldots, V_{\ell}\}=\emptyset$. If $\alpha^{[n]}\not\in \{U_1,\ldots, U_k,V_1,\ldots,V_\ell\}$, then $k=\ell=\frac{\mathsf v_{\tau}(U_1\bdot\ldots\bdot U_k)}{2}$, a contradiction. Thus $\alpha^{[n]}\in \{U_1,\ldots, U_k,V_1,\ldots,V_\ell\}$.
Since $\ell > k$, we obtain $\alpha^{[n]} \in \{V_1, \ldots, V_{\ell}\}$ and
\[
k = \mathsf v_{\tau}(U_1 \bdot \ldots \bdot U_k)/2 \quad \text{and} \quad
\ell = |\{j \in [1, \ell] \colon V_j = \alpha^{[n]} \}|+ \mathsf v_{\tau} (V_1 \bdot \ldots \bdot V_{\ell})/2 \,.
\]
Since $\mathsf v_{\alpha} (U_1 \bdot \ldots \bdot U_k)$ is even, we infer that $\ell - k = |\{j \in [1, \ell] \colon  V_j = \alpha^{[n]} \}|$ is even whence $\min \Delta ( \{\alpha, \tau\}) \ge 2$.

4.(b) By 4.(a), it remains to verify $\max\Delta^*(G)\setminus\{2n-2\}\le \max\{2,n-2\}$. Let $G_0\subset G\setminus \{1_G\}$ with $|G_0|\ge 2$.  If $G_0\subset \langle \alpha\rangle$, then $\min\Delta(G_0)\le \max \Delta^* (C_n) = n-2$ by Item 3. Suppose there exists $i\in [0,n-1]$ such that $\alpha^i\tau\in G_0$. If there exists $j\in [1,n-1]$ such that $\alpha^j\in G_0$, then
\[
((\alpha^j)^{[n/\gcd(j,n)]})^{[2]}\bdot ((\alpha^i\tau)^{[2]})^{[2]}=\big((\alpha^j)^{[2n/\gcd(j,n)-2]}\bdot(\alpha^i\tau)^{[2]}\big)\bdot \big((\alpha^j)^{[2]}\bdot (\alpha^i\tau)^{[2]}\big)
\]
implies that $\min\Delta(G_0)\le 2$. Suppose $G_0\cap \langle \alpha\rangle=\emptyset$ and hence there
 exist distinct $i,j\in [0,n-1]$ such that $\{\alpha^i\tau,\alpha^j\tau\}\subset G_0$. If $\gcd(i-j,n)>1$, then $\min\Delta(G_0)\le \min\Delta(\{\alpha^i\tau, \alpha^j\tau\})\le \frac{2n}{\gcd(i-j,n)}-2<n-2$. If $\gcd(i-j,n)=1$, then choosing a different basis if necessary, we may assume that $\tau, \alpha\tau\in G_0$. If $G_0=\{\tau,\alpha\tau\}$, then $\min\Delta(G_0)=2n-2$. Suppose there exists $r\in [2,n-1]$ such that $\alpha^r\tau\in G_0$. By Proposition \ref{6.5} and Lemma \ref{6.3}.3, we obtain that $\min\Delta(G_0)\le \|(\alpha^r)^{[n]}\|_{\alpha}-1=r-1\le n-2$.

5. Let $G_0\subset G$ be a non half-factorial subset such that $\min\Delta(G_0)=\mathsf D(G)-2$. If there exists an atom $A\in \mathcal A(G_0)$ such that $\mathsf k(A)<1$, then \eqref{eq1} implies that there is $g\in G_0$ with $\ord(g)=\mathsf D(G)$ and hence $G$ is cyclic.
	
	Suppose $\mathsf k(A)\ge 1$ for all $A\in \mathcal A(G_0)$. Then \eqref{eq2} implies that there exists $A=g_1\bdot\ldots\bdot g_{\mathsf D(G)}\in \mathcal A(G_0)$ such that
$B=g_1^{[\ord(g_1)-1]}\bdot\ldots\bdot g_{\mathsf D(G)}^{[\ord(g_{\mathsf D(G)})-1]}$ is an atom. Hence $\ord(g_i)=2$ for all $i\in [1, \mathsf D(G)]$, else $|B|>\mathsf D(G)$. Then $G_1=\langle g_1,\ldots, g_{\mathsf D(G)}\rangle$ is a subgroup satisfying the assertion.
\end{proof}

If $G$ is a dihedral group of order $2n$, then $G$ has a cyclic subgroup of order $n$ whence $\Delta^* (C_n) \subset \Delta^* (G)$.
The set $\Delta^* (G)$ for finite cyclic groups is studied in detail in  \cite{Pl-Sc20a}. If $G$ is finite cyclic, then, by Theorem \ref{6.6}, we have $\max \Delta^* (G) = |G|-2$. The second largest value of $\Delta ^* (G) $ equals $\lfloor |G|/2 \rfloor - 1$.

\medskip
Next we look at the structure of (long) sets of lengths having maximal elasticity. To do so, we define two further subsets of the set of distances. Let $H$ be an atomic monoid. Following (\cite[Definition 2.1 and Lemma 2.2]{Ge-Zh18a}, we define
\begin{itemize}
\item  $\Delta_{\rho} (H)$ to be the set of all $d \in \N$ having the following property:
       \begin{itemize}
       \item[] For every $k \in \N$, there is some $L_k \in \mathcal L (H)$ that is an AAP with difference $d$ and length at least $k$.
       \end{itemize}

\item $\Delta_{\rho}^* (H) = \{\min  \Delta ( \LK a \RK ) \colon a \in H \ \text{with} \ \rho (\mathsf L (a)) = \rho (H) \}$.
\end{itemize}
If $H$ is finitely generated and $\Delta (H) \ne \emptyset$, then by \cite[Lemma 2.4]{Ge-Zh18a}
\begin{equation} \label{Delta-rho-Delta-rho^*}
\Delta_{\rho}^* (H) \subset \Delta_{\rho} (H) \subset \{d_1 \in \Delta (H) \colon d_1 \ \text{divides some} \ d \in \Delta_{\rho}^* (H) \} \,.
\end{equation}
Let $G$ be a finite group. Every divisor-closed submonoid of $\mathcal B (G)$ is generated by one element and all divisor-closed submonoids $S \subset \mathcal B (G)$ are of the form $S = \mathcal B (G_0)$ for a subset $G_0 \subset G$. Consistent with our conventions, we set $\Delta_{\rho} (G) := \Delta_{\rho} ( \mathcal B (G))$ and $\Delta_{\rho}^*  (G) := \Delta_{\rho}^* ( \mathcal B (G))$, and we have (by using \cite[Lemma 2.2.3]{Ge-Zh18a} and that $\rho (G) = \mathsf D (G)/2$)
\[
\Delta_{\rho}^* (G) = \{ \min \Delta (G_0) \colon G_0 \subset G \ \text{with} \ \rho (G_0) = \mathsf D (G)/2 \}  \subset \Delta^* (G) \,.
\]

Before we formulate our main result on $\Delta_{\rho}^* (G)$ for dihedral groups $G$, we briefly summarize what is known on $\Delta_{\rho}^* (G)$ for abelian groups. Let $G$ be a finite abelian group with $|G| \ge 3$. If $G$ is an elementary $2$-group of rank $r$, then $\Delta_{\rho}^* (G) = \Delta_{\rho}(G) = \{1, r-1\}$. If $G$ is neither cyclic nor an elementary $2$-group, then the standing conjecture states that $\Delta_{\rho}(G) = \{1\}$. If $G$ is cyclic with $|G| > 10$, then $\{1, |G|-2\} \subset \Delta_{\rho}^* (G) = \Delta_{\rho} (G)$ and the precise form of $\Delta_{\rho} (G)$ depends on number theoretic properties of the group order $|G|$ (\cite{Ge-Zh18a}). Thus it is no surprise that the similar phenomena occur for dihedral groups.

\begin{theorem} \label{6.7}
Let $G$ be a dihedral group of order $2n$, where $n\ge 3$ is odd,  say  $G=\langle\alpha,\tau\colon \alpha^n=\tau^2=1_G \text{ and }\alpha\tau=\tau\alpha^{-1}\rangle$.
\begin{enumerate}
\item  $\{1, 2n-2\}\subset \Delta_{\rho}^*(G)$ and $\max\Delta^*_{\rho}(G)\setminus\{2n-2\}\le \max\{1,\frac{n-1}{4}\}$.

\item If there exists  $i\in [2,n-1]$ with $\gcd(i,n)=1$ such that
\[
\gcd \left( \min \Delta(\{\alpha,\alpha^i\}), \ \min \Delta(\{\alpha, \alpha^{1-i}\}), \ \min\Delta(\{\alpha^i, \alpha^{i-1}\}) \right) \quad \text{is even, \qquad $(*)$}
\]
then $\Delta_{\rho}^*(G)\supsetneq\{1, 2n-2\}$, and $(*)$ holds, for example,  if $n=m^2-m+1$ for some odd $m \ge 3$.
\end{enumerate}
\end{theorem}

\begin{proof}
1.(i)	Let  $U_1=\alpha^{[2n-2]} \bdot \tau^{[2]}$, $U_2=(\alpha\tau)^{[n]}\bdot \tau^{[n]}$.
	Then $U_1,U_2\in \mathcal A(G)$ with $|U_1|=|U_2|=2n=\mathsf D(G)$.
	Since $\rho( \mathsf L (U_1\bdot U_1^{-1}))=\rho( \mathsf L (U_2^{[2]}))=\rho(G)$, we obtain
	$\min \Delta(\supp(U_1\bdot U_1^{-1}))=\min \Delta(\{\alpha, \alpha^{-1}, \tau\})\in \Delta^*_{\rho}(G)$ and $\min \Delta(\supp(U_2^{[2]}))=\min \Delta(\{\alpha\tau,  \tau\})=2n-2\in \Delta^*_{\rho}(G)$
	
	Since $\alpha^{[n]}\bdot (\alpha^{-1})^{[n]}=(\alpha\bdot \alpha^{-1})^{[n]}$ and $\alpha^{[n]}\bdot \alpha^{[n]}\bdot \tau^{[2]}\bdot \tau^{[2]}=(\alpha^{[2n-2]}\bdot \tau^{[2]})\bdot (\alpha^{[2]}\bdot \tau^{[2]})$, we obtain that
	$\min \Delta(\{\alpha,\alpha^{-1}, \tau\})\t \gcd(n-2, 2)=1$ which implies that $\min \Delta(\{\alpha,\alpha^{-1}, \tau\})=1\in \Delta^*_{\rho}(G)$.
\smallskip
	
1.(ii) Let $d=\max\Delta^*_{\rho}(G)\setminus\{2n-2\}$. Then there exist $G_0\subset G$ and $W\in \mathcal B(G_0)$ with $d=\min\Delta(G_0)$ and $G_0=\supp(W)$ such that $\rho( \mathsf L (W))=\frac{\mathsf D(G)}{2}=n$.  Therefore there are atoms $U_1,\ldots, U_k$  of length $2n$ and atoms $V_1,\ldots, V_{kn}$ of length $2$  such that $W=U_1\bdot\ldots\bdot U_k=V_1\bdot\ldots\bdot V_{kn}$. Thus if $a\in \supp(W)$, then $a^{-1}\in \supp(W)$.
	
	If there exists $i\in[1,k]$ such that $U_i=(\alpha^{k_1})^{[2n-2]}\bdot (\alpha^{k_2}\tau)^{[2]}$, where $\gcd(k_1,n)=1$, 	then we have $\{\alpha^{k_1}, \alpha^{-k_1}, \alpha^{k_2}\tau\}\subset G_0$ and hence $d\t \min\Delta(\{\alpha^{k_1}, \alpha^{-k_1}, \alpha^{k_2}\tau\})=\min\Delta(\{\alpha, \alpha^{-1}, \tau\})=1$.
	By Proposition \ref{2.4}, we may assume that for all $i\in [1,k]$, there are $k_i,t_i\in [0,n-1]$ with $\gcd(k_i-t_i,n)=1$ such that $U_i=(\alpha^{k_i}\tau)^{[n]}\bdot(\alpha^{t_i}\tau)^{[n]}$. Changing a different basis if necessary, we may assume that $U_1=\tau^{[n]}\bdot (\alpha\tau)^{[n]}$. If $U_i=U_1$ for all $i\in [2,k]$, then $G_0=\{\tau,\alpha\tau\}$ and $d=2n-2$, a contradiction. Otherwise there is $r\in [2,n-1]$ such that $\{\tau,\alpha\tau, \alpha^r\tau\}\subset G_0$. By Proposition \ref{6.5} and Lemma \ref{6.3}.3, we obtain that $d\t \gcd(\|(\alpha^r)^{[n]}\|_{\alpha}-1, \|(\alpha^{n+1-r})^{[n]}\|_{\alpha}-1)=\gcd(r-1, n-r)$ and hence $d<n-1$. If $d=\frac{n-1}{2}\ge 2$, then $r-1=n-r=\frac{n-1}{2}$. Since the continued fraction of $n/r$ is $[1;1,r-1]$, it follows by Lemma \ref{6.3}.4 that $\min\Delta(\{\alpha, \alpha^r\})=1$ and hence $d=1$, a contradiction. Suppose $d=\frac{n-1}{3}\ge 2$. Then $r-1=2(n-r)=2\frac{n-1}{3}$ or $n-r=2(r-1)=2\frac{n-1}{3}$.

If $r-1=2(n-r)$, then $2n=3r-1$ and hence the continued fraction of $n/(n+1-r)$ is $[2;1,\frac{r-3}{4}]$ or $[2;1,\frac{r-5}{4},1,1]$. It follows by Lemma \ref{6.3}.4 that $\min\Delta(\{\alpha, \alpha^{1-r}\})=1$ and hence $d=1$, a contradiction.
	
If $n-r=2(r-1)$, then $n=3r-2$ and hence the continued fraction of $n/r$ is $[2;1,\frac{r-3}{2},1,1]$. It follows by Lemma \ref{6.3}.4 that $\min\Delta(\{\alpha, \alpha^{r}\})=1$ and hence $d=1$, a contradiction.

Therefore, we obtain that $d\le \frac{n-1}{4}$.

\medskip
2.(i) Let $i\in [2,i-1]$ with $\gcd(i,n)=1$ such that
\[
\gcd \left( \min \Delta(\{\alpha,\alpha^i\}), \ \min \Delta(\{\alpha, \alpha^{1-i}\}), \ \min\Delta(\{\alpha^i, \alpha^{i-1}\}) \right)
\]
is even. Since $U=(\alpha\tau\bdot \tau)^{[n]}$ and $V=(\alpha^i\tau\bdot \tau)^{[n]}$ are atoms of length $\mathsf D(G)$, then $\rho( \mathsf L (U^{[2]}\bdot V^{[2]}))=\rho(G)$ and hence $d=\min\Delta(\{\tau,\alpha\tau,\alpha^i\tau\})\in \Delta^*_{\rho}(G)$. By 1., it suffices to show $d\neq 1$ and $d\neq 2n-2$.

Let $W=(\alpha\tau\bdot \tau^{[n-i]})\bdot \alpha^i\tau\bdot \tau$. Then $W$ is an atom of Type III in Lemma \ref{lemma6.7}. Since $W^{[n]}=((\alpha\tau\bdot\tau)^{[n]})^{[n-i]}\bdot (\alpha^i\tau\bdot \tau)^{[n]}$, we obtain that $d\t n-(n-i+1)=i-1$ which implies that $i<2n-2$.	
	
Assume to the contrary that $d=1$. Then there are atoms $U_1,\ldots, U_k, V_1,\ldots, V_{k+1}$ over $\{\tau,\alpha\tau, \alpha^i\tau\}$ such that
$$U_1\bdot\ldots\bdot U_k=V_1\bdot\ldots\bdot V_{k+1}\,.$$	
	
We assert that if $A\in \mathcal A(\{\tau,\alpha\tau, \alpha^i\tau\})$, then $|A|\equiv 2\pmod 4$. Suppose this holds. Then $2k\equiv 2(k+2)\pmod 4$, a contradiction. Thus we only need to show the assertion. Note that $(\alpha^i)^{[n]}$ is an atom. Since $\min\Delta(\{\alpha,\alpha^i\})$ is even, it follows by Lemma \ref{6.3}.3 that  $i=\|(\alpha^i)^{[n]}\|_{\alpha}$ is odd. If $A$ is of Type I or Type II in Lemma \ref{lemma6.7}, the assertion follows by $n$ is odd. If $A$ is of Type III in Lemma \ref{lemma6.7}, say $A=(\alpha\tau\bdot \tau)^{[x]}\bdot (\alpha^i\tau\bdot\tau)^{[y]}$, then $\alpha^{[x]}\bdot (\alpha^{1-i})^{[y]}$ is an atom and by Lemma \ref{6.3}.3 that $\|\alpha^{[x]}\bdot (\alpha^{1-i})^{[y]}\|_{\alpha}=\frac{x+iy}{n}$ is odd. Therefore $x+y\equiv 1\pmod 2$ and $|A|=2x+2y\equiv 2\pmod 4$.  If $A$ is of Type IV in Lemma \ref{lemma6.7}, say $A=(\alpha\tau\bdot \tau)^{[x]}\bdot (\alpha\tau\bdot\alpha^i\tau)^{[y]}$, then $\alpha^{[x]}\bdot (\alpha^{1-i})^{[y]}$ is an atom. Since $\min\Delta(\{1,1-i\})$ is even, it follows by Lemma \ref{6.3}.3 that  $\|\alpha^{[x]}\bdot (\alpha^{1-i})^{[y]}\|_{\alpha}=\frac{x+(n+1-i)y}{n}$ is odd. Therefore $x+y\equiv 1\pmod 2$ and $|A|=2x+2y\equiv 2\pmod 4$. If $A$ is of Type V in Lemma \ref{lemma6.7}, say $A=(\alpha^i\tau\bdot \alpha\tau)^{[x]}\bdot (\alpha^i\tau\bdot\tau)^{[y]}$, then $(\alpha^{i-1})^{[x]}\bdot (\alpha^{1-i})^{[y]}$ is an atom. Let $j\in [1,n-1]$ such that $ij\equiv 1\pmod n$. Then $j$ is odd,  $\min \Delta(\{\alpha^i,\alpha^{i-1}\})=\min \Delta(\{\alpha, \alpha^{n+1-j}\})$ is even, and $\alpha^{[x]}\bdot(\alpha^{1-j})^{[y]}$, $(\alpha^{1-j})^{[n/\gcd (1-j,n)]}$ are both atoms. Thus Lemma \ref{6.3}.3 implies that  $\|\alpha^{[x]}\bdot(\alpha^{n+1-j})^{[y]}\|_{\alpha}=\frac{x+(n+1-j)y}{n}$ and $\|(\alpha^{1-j})^{[n/\gcd (1-j,n)]} \|_{\alpha} = n(n+1-j)/\gcd(1-j,n)$ are odd.
 Therefore $x+y\equiv 1\pmod 2$ and $|A|=2x+2y\equiv 2\pmod 4$.
\smallskip

2.(ii) Let $n=x^2-x+1$ for some odd $x\in \N$.  Then $\gcd(n,x)=1$. We need to show $(*)$ holds.
Since  $\alpha^{x-1}=(\alpha^x)^{x}$, we have $\min \Delta(\{\alpha^x,\alpha^{x-1}\})=\min \Delta(\{\alpha,\alpha^x\})$. Since the continued fraction of $n/x$ with odd length is $[x-1;x-1,1]$ and the continued fraction of $n/(n+1-x)$ with odd length is $[1;x-1,x-1]$, it follows by Lemma \ref{6.3}.4 that $\gcd(\min \Delta(\{\alpha,\alpha^x\}), \min \Delta(\{\alpha,\alpha^{1-x}\}))=x-1$ is even. 
\end{proof}

\begin{remark} \label{6.8}
Let $G$ be a dihedral group of order $2n$, where $n\ge 3$ is odd,  say  $G=\langle\alpha,\tau\colon \alpha^n=\tau^2=1_G \text{ and }\alpha\tau=\tau\alpha^{-1}\rangle$.
If, for all $i\in [2,n-1]$, we have
\begin{equation}\label{eq} \gcd(\min\Delta(\{\alpha, \alpha^i\}),\min\Delta(\{\alpha, \alpha^{1-i}\}),\min\Delta(\{\alpha^{i-1}, \alpha^i\}))=1\,,
\end{equation}  then a similar proof as that of Theorem \ref{6.7}.1 shows that $\max\Delta^*_{\rho}(G)\setminus\{2n-2\}=1$ and hence $\Delta^*_{\rho}(G)=\{1,2n-2\}$.
By Lemma \ref{6.3}, we can use  continued fraction expansions to check  Condition \ref{eq} and (within a few minutes of computer calculations) one gets the list of all  $n\in [5, 10000]$ with  $\Delta^*_{\rho}(G)=\{1,2n-2\}$.
\end{remark}

\begin{corollary} \label{6.9}
Let $G$ be a dihedral group of order $2n$ where $n \ge 3$ is odd. Then
\[
2+\max \Delta_{\rho}^* (G) = 2 + \max \Delta^* (G) = 2 +\max \Delta (G) = \mathsf c (G)=\omega (G)=2n= \mathsf D (G) = |G| \,.
\]
\end{corollary}

\begin{proof}
We have $\Delta_{\rho}^* (G) \subset \Delta^* (G) \subset \Delta (G)$. Proposition \ref{6.7} implies that $2n \le 2 + \max \Delta_{\rho}^* (G)$. Thus the assertion follows from Theorems \ref{4.1} and \ref{5.1}.
\end{proof}

Consider a class $\mathcal C$ of atomic monoids or domains (say orders in algebraic number fields, Krull monoids, or monoids of product-one sequences). Arithmetical investigations of objects from $\mathcal C$ are always done with respect to the following aims and questions.
\begin{itemize}
\item Are the arithmetical invariants of two objects $H_1$ and $H_2$ in $\mathcal C$ characteristic for $H_1$ and $H_2$? To pick a prominent question of this type, let $G_1$ and $G_2$ be two finite abelian groups, say with $|G_1| > 4$. The standing conjecture states that $\mathcal L (G_1) = \mathcal L (G_2)$ implies that $G_1$ and $G_2$ are isomorphic (see \cite{Ge-Zh20a} for an overwiew).

\item To what extent is the arithmetic of an object $H$ in $\mathcal C$ distinct from the arithmetic of objects of a further class $\mathcal C'$?
\end{itemize}

In our final result (Corollary \ref{6.11}) we demonstrate that our results on the arithmetic of $\mathcal B (D_{2n})$, where $n \ge 3$ is odd, are strong enough to settle questions of the above type. We start with a lemma.

\medskip
\begin{lemma} \label{6.10}
Let $G_1$ and $G_2$ be finite groups such that $\mathcal L (G_1) = \mathcal L (G_2)$.
\begin{enumerate}
\item $\Delta (G_1) = \Delta (G_2)$, $\mathcal U_k (G_1) = \mathcal U_k (G_2)$ for all $k \in \N$.

\item For every $k \in \N$ we have $\rho_k (G_1) = \rho_k (G_2)$ and $\mathsf D (G_1) = \mathsf D (G_2)$.

\item $\max \Delta^* (G_1) = \max \Delta^* (G_2)$ and $\max \Delta_{\rho}^* (G_1) =     \max \Delta_{\rho}^* (G_2)$.
\end{enumerate}
\end{lemma}

\begin{proof}
1. Since $\mathcal L (G_1) = \mathcal L (G_2)$, this follows from Equations \eqref{def-set-of-distances} and \eqref{def-unions}.

2. This follows from 1. and from Proposition \ref{6.0}.1.

3. Since $\mathcal L (G_1) = \mathcal L (G_2)$, it follows that $\Delta_1 (G_1) = \Delta_2 (G_2)$ and $\Delta_{\rho} (G_1) = \Delta_{\rho} (G_2)$. Thus  \eqref{Delta-1-Delta^*} and  \eqref{Delta-rho-Delta-rho^*} imply that
\[
\begin{aligned}
\max \Delta^* (G_1) = \max \Delta_1 (G_1) & = \max \Delta_1 (G_2) = \max \Delta^* (G_2) \quad \text{and} \\
\max \Delta_{\rho}^* (G_1) = \max \Delta_{\rho} (G_1) & = \max \Delta_{\rho} (G_2) = \max \Delta_{\rho}^* (G_2) \,. \qedhere
\end{aligned}
\]
\end{proof}

A monoid $H$ is {\it strongly primary} if $H \ne H^{\times}$ and, for each $a \in H \setminus H^{\times}$, there is some $n \in \N$ such that $(H \setminus H^{\times})^n \subset aH$. Numerical monoids and the multiplicative monoids of nonzero elements of local one-dimensional Mori domains are strongly primary. For the arithmetic of strongly primary monoids we refer to \cite{Ge-Go-Tr21}. Every strongly primary monoid is weakly Krull, whence $\mathcal B (D_{2n})$ is not strongly primary by Theorem \ref{3.3}.

\begin{corollary} \label{6.11}
Let $n \ge 3$ be odd.
\begin{enumerate}
\item If $m \in \N_{\ge 3}$ with $m \ne n$, then $\mathcal L (D_{2n}) \ne \mathcal L (D_{2m})$.

\item If $G$ is a finite nilpotent group but not a non-abelian $2$-group, then  $\mathcal L (D_{2n}) \ne \mathcal L (G)$.

\item If $H$ is a strongly primary monoid, then $\mathcal L (D_{2n}) \ne \mathcal L (H)$.
\end{enumerate}
\end{corollary}

\begin{proof}
1. Let $m \in \N_{\ge 3}$ with $m \ne n$ and  assume to the contrary that $\mathcal L(D_{2n})=\mathcal L(D_{2m})$. Then Lemma \ref{6.10} and Corollary \ref{6.9} imply that $2n=\mathsf D(D_{2n})=\mathsf D(D_{2m})$ and $2n-2=\max\Delta^*_{\rho}(D_{2n})=\max\Delta^*_{\rho}(D_{2m})$. If $m$ is odd, then $2n=\mathsf D(D_{2m})=2m$, a contradiction. Thus $m\ge 4$ is even. Let $G = D_{2m}$ and let $U$ be an atom of length $\mathsf D(G)$  over $G$. By \cite[Theorem 4.2]{Oh-Zh20a}, there exist $\alpha, \tau\in G$ with $G=\langle \alpha, \tau\colon \alpha^m=\tau^2=1_G, \tau\alpha=\alpha^{-1}\tau\rangle$ such that $U=\alpha^{[3m/2-2]}\bdot \tau\bdot \alpha^{m/2}\tau$. It follows by the definition of $\Delta^*_{\rho}(G)$ that
$$\max\Delta^*_{\rho}(G)=\min \Delta(\{\supp(U\bdot U^{-1})\})\le \min\Delta(\{\alpha, \alpha^{-1}\})=m-2<\mathsf D(G)-2=\max\Delta^*_{\rho}(D_{2n})\,,$$	
a contradiction.

2. Let $G$ be a finite nilpotent group  such that $\mathcal L (G) = \mathcal L (D_{2n})$. Then $G$ is not abelian by \cite[Theorem 4.4]{Oh19b}. Thus  Theorem \ref{6.6} implies $\max\Delta^*(G)=\max\Delta^*(D_{2n})=\mathsf D(D_{2n})-2=\mathsf D(G)-2$, and hence $G$ has a subgroup $G_1$ with $\mathsf D(G_1)=\mathsf D(G)$ such that $G_1$ is generated by elements of order $2$. Since $G_1$ is also a nilpotent group, it follows that $G_1$ is a $2$-group by \cite[Corollary 2.4]{G-K-O-R99}. Thus Lemma \ref{2.2}.2 implies that $G$ is a $2$-group.	

3. Let $H$ be a strongly primary monoid. Then Proposition \ref{6.1} and \cite[Theorem 5.5]{Ge-Sc-Zh17b} show that the set of elasticities $\{ \rho (L) \colon L \in \mathcal L (H) \}$ and $\{ \rho (L) \colon L \in \mathcal L (G) \}$ are distinct, whence $\mathcal L (H) \ne \mathcal L (G)$.
\end{proof}

\smallskip
\noindent
{\bf Acknowledgement.} We thank the referees for their careful reading. Their comments helped to improve the presentation of the paper.

\smallskip
\noindent
{\bf Note added in proof.} Monoids of product-one sequences over subsets of infinite dihedral groups were recently studied in \cite{Fa-Zh22a}.


\providecommand{\bysame}{\leavevmode\hbox to3em{\hrulefill}\thinspace}
\providecommand{\MR}{\relax\ifhmode\unskip\space\fi MR }
\providecommand{\MRhref}[2]{%
  \href{http://www.ams.org/mathscinet-getitem?mr=#1}{#2}
}
\providecommand{\href}[2]{#2}

\end{document}